\documentclass[11pt]{article}
\usepackage{eurosym}
\usepackage{amsfonts}
\usepackage{amssymb}
\usepackage{amsthm}
\usepackage{amsmath}
\usepackage{graphicx}
\usepackage{empheq}
\usepackage{indentfirst}
\usepackage{cite}
\usepackage{mathrsfs}
\usepackage{cases}
\usepackage{graphics}
\usepackage{xcolor}
\usepackage{bm}
\usepackage{makeidx}
\usepackage[T1]{fontenc}
\usepackage{times}

\newtheorem{theorem}{\textbf{Theorem}}[section]
\newtheorem{lemma}{\textbf{Lemma}}[section]
\newtheorem{proposition}{\textbf{Proposition}}[section]
\newtheorem{corollary}{\textbf{Corollary}}[section]
\newtheorem{remark}{\textbf{Remark}}[section]
\newtheorem{definition}{\textbf{Definition}}[section]

\allowdisplaybreaks[4]

\def\be{\begin{equation}}
	\def\ee{\end{equation}}
\def\bea{\begin{eqnarray}}
	\def\eea{\end{eqnarray}}
\def\bt{\begin{theorem}}
	\def\et{\end{theorem}}
\def\bl{\begin{lemma}}
	\def\el{\end{lemma}}
\def\br{\begin{remark}}
	\def\er{\end{remark}}
\def\bp{\begin{proposition}}
	\def\ep{\end{proposition}}
\def\bc{\begin{corollary}}
	\def\ec{\end{corollary}}
\def\bd{\begin{definition}}
	\def\ed{\end{definition}}

\begin{document}
	
	\title{Long-time behavior of a nonlocal Cahn--Hilliard equation with nonlocal dynamic boundary condition and singular potentials}
	\author{
		Maoyin Lv \thanks{%
			School of Mathematical Sciences, Fudan University, Shanghai
			200433, P. R. China. Email: \texttt{mylv22@m.fudan.edu.cn} }, \ \ \
		Hao Wu \thanks{%
			Corresponding author. School of Mathematical Sciences, Fudan University, Shanghai 200433, P. R. China. Email: \texttt{haowufd@fudan.edu.cn} } }
	\date{\today }
	\maketitle


	\begin{abstract}
		\noindent
We investigate the long-time behavior of a nonlocal Cahn--Hilliard equation in a bounded domain $\Omega\subset\mathbb{R}^d$ $(d\in\{2,3\})$,
subject to a kinetic rate-dependent nonlocal dynamic boundary condition.
The kinetic rate $1/L$, with $L\in[0,+\infty)$, distinguishes different types of bulk-surface interactions.
For general singular potentials, including the physically relevant logarithmic potential,
we establish the existence of a global attractor $\mathcal{A}_m^L$ in a suitable complete metric space for any $L\in[0,+\infty)$.
Moreover, we verify that the global attractor $\mathcal{A}_m^0$ is stable with respect to perturbations $\mathcal{A}_m^L$ for small $L>0$.
When $L\in(0,+\infty)$, based on the strict separation property of global weak solutions, we further prove the existence of exponential attractors via a short-trajectory type technique, which also implies that the global attractor has finite fractal dimension.
Finally, for this case, we show that every global weak solution converges to a single equilibrium in $\mathcal{L}^\infty$ as time goes to infinity, using a generalized {\L}ojasiewicz--Simon inequality and an Alikakos--Moser type iteration.
		
		\medskip
		
	\end{abstract}
	
	
	\section{Introduction}
	In this paper, we investigate the following nonlocal Cahn--Hilliard equation
	\begin{align}
		\begin{cases}
		\partial_{t}\varphi=\Delta\mu,&\text{in }\Omega\times(0,+\infty),\\
		\mu=a_{\Omega}\varphi-J\ast\varphi+F'(\varphi),&\text{in }\Omega\times(0,+\infty),
		\end{cases}\label{nonlocal-CH}
	\end{align}
	subject to the nonlocal dynamic boundary condition
	\begin{align}
		\begin{cases}
		\partial_{t}\psi=\Delta_{\Gamma}\theta-\partial_{\mathbf{n}}\mu,&\text{on }\Gamma\times(0,+\infty),\\
		\theta=a_{\Gamma}\psi-K\circledast\psi+G'(\psi),&\text{on }\Gamma\times(0,+\infty),\\
			L\partial_{\mathbf{n}}\mu=\theta-\mu,&\text{on }\Gamma\times(0,+\infty),
		\end{cases}\label{nonlocal-BC}
	\end{align}
	and initial conditions
	\begin{align}
		\varphi|_{t=0}=\varphi_{0}\ \ \text{in }\Omega\quad\text{ and }\quad
		\psi|_{t=0}=\psi_{0}\ \ \text{on }\Gamma.\label{initial}
	\end{align}
The Cahn--Hilliard equation serves as an efficient diffuse interface model for the study of phase segregation phenomena in binary mixtures.
It has broad applications in modeling related phenomena across diverse fields, including materials science, image inpainting, biology, and fluid mechanics. Its nonlocal variant, as in \eqref{nonlocal-CH}, was introduced to describe possible long-range interactions between particles in the interacting materials; see, e.g., \cite{BH,BH05,GL96,GL97}.
Moreover, non-trivial boundary effects have attracted considerable attention, and various forms of dynamic boundary conditions have been explored in the existing literature;
see, e.g., \cite{Gal06,GMS,KLLM,LW} and the recent review paper \cite{W}.

Recently, an extended model consisting of the nonlocal Cahn--Hilliard equation \eqref{nonlocal-CH}
and the nonlocal dynamic boundary condition \eqref{nonlocal-BC} was proposed to describe phase separation processes
with long-range interactions both within the bulk material and on its boundary (cf. \cite{KS}).
The well-posedness of the initial boundary value problem \eqref{nonlocal-CH}--\eqref{initial}
has been established in \cite{KS} for regular potentials and later in \cite{LvWu-4} for singular potentials.
In \eqref{nonlocal-BC}, the parameter $L\in[0,+\infty]$ distinguishes different types of bulk-surface interactions,
and the coefficient $1/L$ can be interpreted as the reaction rate (cf. \cite{KLLM}).
Since $L=+\infty$ implies that the bulk and surface subsystems are completely decoupled,
this situation is less interesting and will not be considered here.
In this study, we shall focus on the case $L\in[0,+\infty)$
and investigate the long-time behavior of problem \eqref{nonlocal-CH}--\eqref{initial} with singular potentials.

	In \eqref{nonlocal-CH}, $\Omega\subset\mathbb{R}^{d}$ $(d\in\{2,3\})$ is a smooth bounded domain with boundary $\Gamma:=\partial\Omega$,
    and the symbol $\Delta$ denotes the Laplace operator in $\Omega$.
    The functions $\varphi:\Omega\times(0,+\infty)\rightarrow[-1,1]$ and $\mu:\Omega\times(0,+\infty)\rightarrow\mathbb{R}$
    denote the bulk phase-field variable and the bulk chemical potential, respectively.
    The symbol $\Delta_{\Gamma}$ stands for the Laplace--Beltrami operator on $\Gamma$,
    the bold symbol $\mathbf{n}$ denotes the outward unit normal vector on the boundary,
    and $\partial_{\mathbf{n}}$ represents the outward normal derivative on $\Gamma$.
    The functions $\psi:\Gamma\times(0,+\infty)\rightarrow[-1,1]$ and $\theta:\Gamma\times(0,+\infty)\rightarrow\mathbb{R}$
    represent the surface phase-field variable and the surface chemical potential, respectively.
    The total free energy functional associated with the system \eqref{nonlocal-CH}--\eqref{nonlocal-BC} is defined as
	\begin{align}
		E(\boldsymbol{\varphi}):= E_{\mathrm{bulk}}(\varphi)+E_{\mathrm{surf}}(\psi),\quad \boldsymbol{\varphi}:=(\varphi, \psi),
		\label{totalenergy}
	\end{align}
	where the bulk free energy $E_{\text{bulk}}$ and the surface free energy $E_{\text{surf}}$ are given by
	\begin{align}
		&E_{\mathrm{bulk}}(\varphi) :=\frac{1}{4}\int_{\Omega}\int_{\Omega}J(x-y)|\varphi(x)-\varphi(y)|^{2} \,\mathrm{d}y\,\mathrm{d}x+\int_{\Omega}F(\varphi(x))\,\mathrm{d}x,\notag\\
		&E_{\mathrm{surf}}(\psi) :=\frac{1}{4}\int_{\Gamma}\int_{\Gamma}K(x-y)|\psi(x)-\psi(y)|^{2} \,\mathrm{d}S_{y}\,\mathrm{d}S_{x}+\int_{\Gamma}G(\psi(x))\,\mathrm{d}S_{x}.\notag
	\end{align}
	Then the chemical potentials $\mu$ and $\theta$ can be expressed as Fr\'echet derivatives of the bulk and surface free energies, respectively.
	The mutual interactions between particles, both short- and long-range, are modeled via convolution integrals.
    These integrals are weighted by suitable interaction kernels  $J,K:\mathbb{R}^{d}\rightarrow\mathbb{R}$,
    which are assumed to be even, i.e., $J(x)=J(-x)$ and $K(x)=K(-x)$ for all $x\in\mathbb{R}^{d}$.
    The symbols ``$\ast$'' in \eqref{nonlocal-CH}$_2$ and ``$\circledast$'' in \eqref{nonlocal-BC}$_2$
    denote the convolutions in the bulk and on the boundary, respectively, that is,
	\begin{align*}
		&(J\ast\varphi)(x,t) :=\int_{\Omega}J(x-y)\varphi(y,t)\,\mathrm{d}y,\quad\forall\,(x,t)\in \Omega\times(0,+\infty),\\
		&(K\circledast\psi)(x,t) :=\int_{\Gamma}K(x-y)\psi(y,t)\,\mathrm{d}S_{y},\quad\forall\,(x,t)\in\Gamma\times(0,+\infty).
	\end{align*}
	Moreover, the functions $a_{\Omega}$ and $a_{\Gamma}$ are defined as
	\begin{align*}
		a_{\Omega}(x):=(J\ast1)(x),\qquad a_{\Gamma}(y):=(K\circledast1)(y),
	\end{align*}
	for all $x\in\Omega$ and $y\in\Gamma$.

	The nonlinear potential functions $F$ and $G$ represent the free energy densities in the bulk and on the boundary, respectively.
    To describe the phase separation phenomena, $F$ and $G$ usually have a double-well structure.
    A well-known physically relevant example is the logarithmic potential \cite{CH}:
	\begin{align}
		&\mathcal{W}_{\mathrm{log}}(s) :=\frac{\Theta}{2}[(1+s)\mathrm{ln}(1+s)+(1-s)\mathrm{ln}(1-s)]-\frac{\Theta_{0}}{2}s^{2},\quad s\in(-1,1),\label{log}
	\end{align}
where $\Theta$ is the absolute temperature of the system and $\Theta_0$ represents the critical temperature of phase separation.
Phase separation processes occur when $\Theta_{0}>\Theta>0$. Under this condition, we find that $\mathcal{W}_{\mathrm{log}}$ is non-convex with two minima $\pm s_\ast\in (-1,1)$,
    where $s_\ast$ is the positive root of the equation $\mathcal{W}_{\mathrm{log}}'(s)=0$.
Since $\mathcal{W}_{\text{log}}'(s)\to\pm\infty$ as $s\to\pm1$, the function $\mathcal{W}_{\text{log}}$ is usually referred to as a singular potential in the literature.
	The nonlinearities $F'$ in \eqref{nonlocal-CH}$_2$ and $G'$ in \eqref{nonlocal-BC}$_2$
    denote the derivatives of the potentials $F$ and $G$, respectively.
    Moreover, when non-smooth potentials are taken into account,
    $F'$ and $G'$ correspond to the subdifferential of the convex part (which may be multi-valued graphs) plus the derivative of the smooth concave perturbations.

In \eqref{nonlocal-CH}--\eqref{initial}, the bulk and boundary chemical potentials $\mu$, $\theta$
are coupled through the boundary condition \eqref{nonlocal-BC}$_3$.
This condition models possible adsorption or desorption processes between the materials in the bulk and on the boundary (see \cite{KLLM}).
	Sufficiently regular solutions to problem \eqref{nonlocal-CH}--\eqref{initial} satisfy the properties of mass conservation and energy dissipation, that is,
	\begin{align}
		\int_{\Omega}\varphi(t)\,\mathrm{d}x +\int_{\Gamma}\psi(t)\,\mathrm{d}S =\int_{\Omega}\varphi_{0}\,\mathrm{d}x+\int_{\Gamma}\psi_{0}\,\mathrm{d}S, \quad\forall\,t\in[0,+\infty), \notag
	\end{align}
	and
	\begin{align}
		& \frac{\mathrm{d}}{\mathrm{d}t}E(\boldsymbol{\varphi}(t)) +\int_{\Omega}|\nabla\mu(t)|^{2} \,\mathrm{d}x +\int_{\Gamma}|\nabla_{\Gamma}\theta(t)|^{2}\,\mathrm{d}S
        \notag\\
        &\qquad +\chi(L)\int_{\Gamma}|\theta(t)-\mu(t)|^{2}\,\mathrm{d}S=0, \quad\forall\,t\in(0,+\infty), \notag
	\end{align}
	with
	\begin{align*}
		\chi(L)=
		\begin{cases}
			1/L,&\text{if }L\in(0,+\infty),\\[1mm]
			0,&\text{if }L=0.
		\end{cases}
	\end{align*}
	Here, the symbol $\nabla$ denotes the usual gradient operator in the bulk, and $\nabla_{\Gamma}$ denotes the tangential (surface) gradient operator on the boundary.

	The nonlocal Cahn--Hilliard equation \eqref{nonlocal-CH}
    was rigorously derived through a stochastic argument in the seminal work \cite{GL96,GL97}.
    It incorporates both long-range repulsive interactions between different species and short-range hard collisions between all particles.
   This equation serves as a macroscopic limit of microscopic phase segregation models with particle-conserving dynamics.
	When studying the evolution within a bounded domain, suitable boundary and initial conditions must be considered. A typical choice is the homogeneous Neumann boundary condition for the bulk chemical potential, that is,
	\begin{align}
	     \partial_\mathbf{n}\mu=0,\quad\text{on }\Gamma\times(0,+\infty).\label{Neumann}
	\end{align}
    The nonlocal Cahn--Hilliard equation \eqref{nonlocal-CH} subject to \eqref{Neumann} and suitable initial condition has been extensively studied from various aspects.
    We refer to \cite{ABG,BH,BH05,DRST,GZ,GGG,PS} for results concerning well-posedness and regularity properties of solutions.
    Studies on its coupling with fluid equations can be found in \cite{CFG22,CFG,FGGS,FG,GGGP}.
    For the property of strict separation from the pure states $\pm 1$, one can see \cite{GGG,GGG23,Gior,Po,GP}.
    Convergence results of the nonlocal Cahn--Hilliard equation to its local counterpart are presented in \cite{AH,DRST,DST,DST20,HKP}.
    Regarding the long-time behavior, we refer to \cite{ABG,FG12,GGG,GG}.
    In particular, the authors in \cite{GG} proved the existence of exponential attractors for regular potentials,
    and established a similar result for the viscous nonlocal Cahn--Hilliard equation with a singular potential.
    They also demonstrated the convergence of a global solution to a single steady state as $t\to +\infty$.
    Subsequently, the authors in \cite{GGG} proved the strict separation property in two dimensions using an Alikakos--Moser iteration argument.
    This property enabled them to extend the results in \cite{GG} to the nonlocal Cahn--Hilliard equation with a singular potential.

	The system \eqref{nonlocal-CH}--\eqref{initial} under investigation was rigorously derived in \cite{KS}
    as the gradient flow of the nonlocal free energy \eqref{totalenergy},
    with respect to a suitable inner product of order $H^{-1}$ containing both bulk and surface contributions.	
    In \cite{KS}, the authors studied the problem \eqref{nonlocal-CH}--\eqref{initial}
    with a boundary penalty term and regular potentials satisfying suitable growth conditions.
    They first established the weak well-posedness for the case $L\in(0,+\infty)$ using a gradient flow approach
    and then investigated the asymptotic limits as the relaxation parameter $L$ tends to $0$ and $+\infty$.
    Under certain additional assumptions, they further obtained higher-order regularity for the solution and established the strong well-posedness
    for the problem \eqref{nonlocal-CH}--\eqref{initial} with a boundary penalty term.
    In our recent work \cite{LvWu-4}, we analyzed the problem \eqref{nonlocal-CH}--\eqref{initial} with singular potentials,
    including the physically relevant logarithmic potential \eqref{log}.
    We first established the existence of global weak solutions when $L\in(0,+\infty)$
    using the Yosida approximation for singular potentials and a suitable Faedo--Galerkin scheme.
    Then we verified the asymptotic limits as $L\to0$ and $L\to+\infty$,
    which also imply the existence of global weak solutions for the limit cases $L=0$ and $L=+\infty$.
    Under additional assumptions on the interaction kernels,
    we also established the convergence rates of the Yosida approximation as the approximating parameter $\varepsilon\to0$,
    as well as the asymptotic limits as $L\to0$ and $L\to+\infty$.
    Furthermore, we demonstrated the regularity propagation and established the strict separation property
    for the case $L\in(0,+\infty)$ by means of a suitable De Giorgi's iteration scheme.
    Finally, we mention \cite{G}, in which the author considered a fractional Cahn--Hilliard equation subject to a fractional dynamic boundary condition.

	In this study, we aim to investigate the long-time behavior of problem \eqref{nonlocal-CH}--\eqref{initial},
    including the existence of global and exponential attractors, and the convergence to a single equilibrium as $t\to +\infty$.
	\begin{itemize}
	\item [(1)] \textbf{Global attractor.} The global attractor is the smallest compact and invariant set in the phase space
    that attracts all bounded sets of initial data as time goes to infinity. For the case $L\in[0,+\infty)$, we prove the existence of a global attractor
    for the dynamical system $(\mathfrak{X}_m,\mathcal{S}^L(t))$ associated with problem \eqref{nonlocal-CH}--\eqref{initial} (see Theorem \ref{global}).
    The proof relies on a general result concerning the existence of global
attractors for semigroups $\mathcal{S}(t)$ acting on a complete metric space $\mathcal{X}$ (see \cite[Corollary 6]{PZ}).
This general result relaxes the usual requirement of strong continuity $\mathcal{S}(t)\in C(\mathcal{X},\mathcal{X})$
to the weaker condition that $\mathcal{S}(t)$ be a closed map.     Consequently, we only need to verify the following three conditions:
    \begin{itemize}
        \item The semigroup $\mathcal{S}^L(t):\mathfrak{X}_m\to\mathfrak{X}_m$ is a closed map;
        \item The semigroup $\mathcal{S}^L(t)$ possesses a connected compact attracting set $\mathcal{K}$;
        \item $\mathcal{S}^L(t)\mathcal{K}\subset\mathcal{K}$ holds for sufficiently large $t$.
    \end{itemize}
    After establishing the existence of a global attractor $\mathcal{A}_m^L$ for $L\in[0,+\infty)$,
    we investigate the stability of the global attractor $\mathcal{A}_m^0$ with respect to perturbations $\mathcal{A}_{m}^L$ for small $L>0$.
    Specifically, we analyze the asymptotic limit of the family $\{\mathcal{A}_m^L\}_{L>0}$ as $L\to0$
    and establish the upper semicontinuity at $L=0$ (see Proposition \ref{stability}).
    We mention that some techniques employed to prove the stability of global attractors at $L=0$ are similar to those used in \cite{GKY},
    where the authors studied the local Cahn--Hilliard equation with dynamic boundary conditions.
	
	\item [(2)] \textbf{Exponential attractor.} An exponential attractor is a semi-invariant and compact set that attracts all bounded sets of the phase space exponentially fast.
    To prove the existence of exponential attractors, the following strict separation property
    $$\|\boldsymbol{\varphi}(t)\|_{\mathcal{L}^\infty}\leq 1-\delta(\tau),\quad \forall\, t\geq\tau,
    $$
    for any $\tau>0$, plays a crucial role, as it enables us to overcome those difficulties caused by the singular potentials.
    Here, the function space $\mathcal{L}^\infty$ is defined as  $\mathcal{L}^\infty:= L^\infty(\Omega)\times L^\infty(\Gamma)$.
    Inspired by \cite{GG},  we establish the existence of exponential attractors (see Theorem \ref{exponential}), using a short-trajectory technique devised in \cite{EZ04}.
    The finite dimensionality of an exponential attractor immediately implies that the global attractor has finite fractal dimension (see Corollary \ref{finite-global}).
    The proof proceeds in two main steps. First, we derive some continuous dependence estimates (see Lemmas \ref{Holder-Lip}--\ref{compactness})
    and apply Lemma \ref{abstract} to conclude the existence of a (discrete) exponential attractor $\mathcal{E}_d$ for the discrete semigroup $\{\mathbb{S}^n:=\mathcal{S}^L(nT)\}_{n\in\mathbb{N}}$.
    Then, following arguments similar to those in \cite[Proof of Theorem 2.8]{GG} (or \cite[Proof of Theorem 4.2]{EZ04}),
    we conclude that
    \[\mathcal{E}=\bigcup_{t\in[0,T]}\mathcal{S}^L(t)\mathcal{E}_d\]
    is the desired exponential attractor for the continuous-time case.
	
	\item [(3)] \textbf{Convergence to equilibrium.} Under the additional assumption that the singular potentials $F$, $G$ are real analytic on $(-1,1)$,
    we prove that every global weak solution converges to a single equilibrium as $t\to+\infty$ (see Theorem \ref{equilibrium}).
    The proof is based on the \L ojasiewicz--Simon approach. First, we apply an abstract result \cite[Theorem 6]{GG03} to derive a generalized \L ojasiewicz--Simon inequality (see Lemma \ref{LS}).
    This inequality enables us to prove the existence of a steady state $\boldsymbol{\varphi}_\infty\in\mathcal{H}^1$ such that
    \[\|\boldsymbol{\varphi}(t)-\boldsymbol{\varphi}_\infty\|_{\mathcal{L}^2}\to0\quad\text{as }t\to+\infty.\]
    Here, the function spaces $\mathcal{H}^1$ and $\mathcal{L}^2$ are given by $\mathcal{H}^1:=H^1(\Omega)\times H^1(\Gamma)$,  $\mathcal{L}^2:=L^2(\Omega)\times L^2(\Gamma)$, respectively.
    Combining the convergence in $\mathcal{L}^2$ with an $\mathcal{L}^2$--$\mathcal{L}^\infty$ smoothing property (see Lemma \ref{L2-L infty}), we further obtain
    \[\|\boldsymbol{\varphi}(t)-\boldsymbol{\varphi}_\infty\|_{\mathcal{L}^\infty}\to 0\quad\text{as }t\to+\infty.\]
The proof of Lemma \ref{L2-L infty} relies on an Alikakos--Moser type argument (cf. \cite{Gal12})
and the regularity of the stationary solution $\boldsymbol{\varphi}_\infty$.
	\end{itemize}

\medskip
\noindent\textbf{Outline of the paper.}
In Section 2, we first introduce notation and function spaces used in the subsequent analysis, and then present the main results of this paper.
In Section 3, we prove the existence of a global attractor for $L\in[0,+\infty)$,
and study the stability of the family of global attractors at $L=0$.
Sections 4 and 5 address the existence of exponential attractors and the convergence to a single equilibrium, respectively, for the case $L\in(0,+\infty)$.
Some useful analytical tools are collected in the Appendix.

	\section{Main Results}
	\setcounter{equation}{0}
	\subsection{Notation and preliminaries}
	For any real Banach space $X$, we denote its norm by $\|\cdot\|_X$, its dual space by $X'$
	and the duality pairing between $X'$ and $X$ by
	$\langle\cdot,\cdot\rangle_{X',X}$. If $X$ is a Hilbert space,
	its inner product will be denoted by $(\cdot,\cdot)_X$.
	The space $L^q(0,T;X)$ ($1\leq q\leq +\infty$)
	denotes the set of all strongly measurable $q$-integrable functions with
	values in $X$, or, if $q=+\infty$, essentially bounded functions.
	The space $C([0,T];X)$ denotes the Banach space of all bounded and
	continuous functions $u:[ 0,T] \rightarrow X$ equipped with the supremum
	norm, while $C_{w}([0,T];X)$ denotes the topological vector space of all
	bounded and weakly continuous functions.
	
	Let $\Omega$ be a bounded domain in $\mathbb{R}^d$ ($d\in \{2,3\}$) with a sufficiently smooth boundary $\Gamma:=\partial \Omega$. We use $|\Omega|$ and $|\Gamma|$
	to denote the Lebesgue measure of $\Omega$ and the Hausdorff measure of $\Gamma$, respectively.
	For any $1\leq q\leq +\infty$, $k\in \mathbb{N}$, the standard Lebesgue and Sobolev spaces on $\Omega$ are denoted by $L^{q}(\Omega )$ and $W^{k,q}(\Omega)$. Here, we use $\mathbb{N}$ for the set of natural numbers including zero.
	For $s\geq 0$ and $q\in [1,+\infty )$, we denote by $H^{s,q}(\Omega )$ the Bessel-potential spaces and by $W^{s,q}(\Omega )$ the Slobodeckij spaces.
	If $q=2$, it holds $H^{s,2}(\Omega)=W^{s,2}(\Omega )$ for all $s$ and these spaces are Hilbert spaces.
	We shall use the notation $H^s(\Omega)=H^{s,2}(\Omega)=W^{s,2}(\Omega )$ and $H^0(\Omega)$ can be identified with $L^2(\Omega)$.
	The Lebesgue spaces, Sobolev spaces, and Slobodeckij spaces on the boundary $\Gamma$ can be defined analogously,
	provided that $\Gamma$ is sufficiently regular.
	We write $H^s(\Gamma)=H^{s,2}(\Gamma)=W^{s,2}(\Gamma)$ and identify $H^0(\Gamma)$ with $L^2(\Gamma)$.
	Hereafter, the following shortcuts will be applied:
	\begin{align*}
		&H:=L^{2}(\Omega),\quad H_{\Gamma}:=L^{2}(\Gamma),\quad V:=H^{1}(\Omega),\quad V_{\Gamma}:=H^{1}(\Gamma).
	\end{align*}
	Next, we introduce the product spaces
	$$\mathcal{L}^{q}:=L^{q}(\Omega)\times L^{q}(\Gamma)\quad\mathrm{and}\quad\mathcal{H}^{k}:=H^{k}(\Omega)\times H^{k}(\Gamma),$$
	for $q\in [1,+\infty]$ and $k\in \mathbb{N}$.
	Like before, we can identify  $\mathcal{H}^{0}$ with $\mathcal{L}^{2}$.
	For any $k\in \mathbb{N}$, $\mathcal{H}^{k}$ is a Hilbert space endowed with the standard inner product
	$$
	((y,y_{\Gamma}),(z,z_{\Gamma}))_{\mathcal{H}^{k}}:=(y,z)_{H^{k}(\Omega)}+(y_{\Gamma},z_{\Gamma})_{H^{k}(\Gamma)},\quad\forall\, (y,y_{\Gamma}), (z,z_{\Gamma})\in\mathcal{H}^{k}
	$$
	and the induced norm $\Vert\cdot\Vert_{\mathcal{H}^{k}}:=(\cdot,\cdot)_{\mathcal{H}^{k}}^{1/2}$.
	We introduce the duality pairing
	\begin{align*}
		\langle (y,y_\Gamma),(\zeta, \zeta_\Gamma)\rangle_{(\mathcal{H}^1)',\mathcal{H}^1}
		= (y,\zeta)_{L^2(\Omega)}+ (y_\Gamma, \zeta_\Gamma)_{L^2(\Gamma)},
		\quad \forall\, (y,y_\Gamma)\in \mathcal{L}^2,\ (\zeta, \zeta_\Gamma)\in \mathcal{H}^1.
	\end{align*}
	By the Riesz representation theorem, this product
	can be extended to a duality pairing on $(\mathcal{H}^1)'\times \mathcal{H}^1$.
	
	For any $k\in\mathbb{Z}^+$, we introduce the Hilbert space
	$$
	\mathcal{V}^{k}:=\big\{(y,y_{\Gamma})\in\mathcal{H}^{k}\;:\;y|_{\Gamma}=y_{\Gamma}\ \ \text{a.e. on }\Gamma\big\},
	$$
	endowed with the inner product $(\cdot,\cdot)_{\mathcal{V}^{k}}:=(\cdot,\cdot)_{\mathcal{H}^{k}}$ and the associated norm $\Vert\cdot\Vert_{\mathcal{V}^{k}}:=\Vert\cdot\Vert_{\mathcal{H}^{k}}$.
	Here, $y|_{\Gamma}$ stands for the trace of $y\in H^k(\Omega)$ on the boundary $\Gamma$, which makes sense for $k\in \mathbb{Z}^+$.
	The duality pairing on $(\mathcal{V}^1)'\times \mathcal{V}^1$ can be defined similarly.
	
	For any given $m\in\mathbb{R}$, we set
	$$
	\mathcal{L}^{2}_{(m)}:=\big\{(y,y_{\Gamma})\in\mathcal{L}^{2}\;:\;\overline{m}(y,y_{\Gamma})=m\big\},
	$$
	where the generalized mean is defined as
	\begin{align}
		\overline{m}(y,y_{\Gamma}):=\frac{|\Omega|\langle y\rangle_{\Omega}+|\Gamma|\langle y_{\Gamma}\rangle_{\Gamma}}{|\Omega|+|\Gamma|},\label{gmean}
	\end{align}
    with
    \[\langle y\rangle_\Omega=\frac{1}{|\Omega|}\langle y,1\rangle_{V',V},\quad \langle y_\Gamma\rangle_\Gamma=\frac{1}{|\Gamma|}\langle y_\Gamma,1\rangle_{V_\Gamma',V_\Gamma}.\]
	Then we define the projection operator $\mathbf{P}:\mathcal{L}^{2}\rightarrow\mathcal{L}^{2}_{(0)}$ by
	\begin{align*}
		\mathbf{P}(y,y_{\Gamma})=(y-\overline{m}(y,y_{\Gamma}),y_{\Gamma}-\overline{m}(y,y_{\Gamma})),\quad\forall\,(y,y_\Gamma)\in \mathcal{L}^2.
	\end{align*}
	The closed linear subspaces
	$$
	\mathcal{H}_{(0)}^k=\mathcal{H}^{k}\cap\mathcal{L}_{(0)}^{2},
	\qquad \mathcal{V}_{(0)}^k=\mathcal{V}^{k}\cap\mathcal{L}_{(0)}^{2},
	\qquad k\in \mathbb{Z}^+,
	$$
	are Hilbert spaces endowed with the inner products $(\cdot,\cdot)_{\mathcal{H}^{k}}$
	and the associated norms $\Vert\cdot\Vert_{\mathcal{H}^{k}}$, respectively.
	For $L\in [0,+\infty)$ and $k\in \mathbb{Z}^+$,  we introduce the notation
	$$
	\mathcal{H}^{k}_{L}:=
	\begin{cases}
		\mathcal{H}^k,\quad \text{if}\ L\in (0,+\infty),\\
		\mathcal{V}^{k},\quad \ \text{if}\ L=0,
	\end{cases}\qquad
	\mathcal{H}^{k}_{L,0}:=
	\begin{cases}
		\mathcal{H}_{(0)}^k,\quad \text{if}\ L\in (0,+\infty),\\
		\mathcal{V}^{k}_{(0)},\quad \ \text{if}\ L=0.
	\end{cases}
	$$
	Consider the bilinear form
	\begin{align}
		a_{L}((y,y_{\Gamma}),(z,z_{\Gamma}))
        & :=\int_{\Omega}\nabla y\cdot\nabla z \,\mathrm{d}x +\int_{\Gamma}\nabla_{\Gamma}y_{\Gamma}\cdot\nabla_{\Gamma}z_{\Gamma}\,\mathrm{d}S
        \notag\\
        & \quad \ +\chi(L)\int_{\Gamma}(y-y_{\Gamma})(z-z_{\Gamma})\,\mathrm{d}S,
		\notag
	\end{align}
	for all $(y,y_{\Gamma}), (z,z_{\Gamma})\in \mathcal{H}^{1}$, where
	$$
	\chi(L)=\begin{cases}
		1/L,\quad \text{if}\ L\in (0,+\infty),\\
		0,\qquad \ \text{if}\ L=0.
	\end{cases}
	$$
	For $L\in [0,+\infty)$, we define the inner product on $\mathcal{H}_{L,0}^1$ by $(\cdot,\cdot)_{\mathcal{H}_{L,0}^1}:=a_L(\cdot,\cdot)$, and for any $(y,y_{\Gamma})\in \mathcal{H}^{1}_{L,0}$, we define
	\begin{align}
		\Vert(y,y_{\Gamma})\Vert_{\mathcal{H}^{1}_{L,0}}:=((y,y_{\Gamma}),(y,y_{\Gamma}))_{\mathcal{H}^{1}_{L,0}}^{1/2}= [ a_{L}((y,y_{\Gamma}),(y,y_{\Gamma}))]^{1/2}.
		\label{norm-hL}
	\end{align}
	We note that for $(y,y_{\Gamma})\in \mathcal{V}_{(0)}^1\subseteq\mathcal{H}^{1}_{L,0}$,
    $\Vert(y,y_{\Gamma})\Vert_{\mathcal{H}^{1}_{L,0}}$ does not depend on $L$,
    since the third term in $a_L$ vanishes.
	The following Poincar\'{e}-type inequality has been proved in \cite[Lemma A.1]{KL}.
	\begin{lemma}
		There exists a constant $C_\mathrm{P}>0$ depending only on $L\in [0,+\infty)$ and $\Omega$ such that
		\begin{align}
			\|(y,y_{\Gamma})\|_{\mathcal{L}^2}\leq C_\mathrm{P} \Vert(y,y_{\Gamma})\Vert_{\mathcal{H}^{1}_{L,0}},\quad \forall\, (y,y_{\Gamma})\in \mathcal{H}^{1}_{L,0}.
			\label{Poin}
		\end{align}
	\end{lemma}
	\noindent
	Hence, for every $L\in [0,+\infty)$, $\mathcal{H}^{1}_{L,0}$ is a Hilbert space
    with the inner product $(\cdot,\cdot)_{\mathcal{H}^{1}_{L,0}}$.
    The induced norm $\Vert\cdot\Vert_{\mathcal{H}^{1}_{L,0}}$ prescribed in \eqref{norm-hL}
    is equivalent to the standard one $\Vert\cdot\Vert_{\mathcal{H}^{1}}$ on $\mathcal{H}^{1}_{L,0}$.

	For $L\in[0,+\infty)$, let us consider the following elliptic boundary value problem
	\begin{align}
		\begin{cases}
			-\Delta u = y,&\quad \text{in }\Omega,\\
			-\Delta_{\Gamma}u_\Gamma +\partial_{\mathbf{n}}u= y_{\Gamma},&\quad \text{on }\Gamma,\\
			L\partial_{\mathbf{n}}u =u_\Gamma-u, &\quad \mathrm{on\;}\Gamma.
		\end{cases}
		\label{Elliptic}
	\end{align}
	Define the space
	$$
	\mathcal{H}_{L,0}^{-1}=
	\begin{cases}
		\mathcal{H}_{(0)}^{-1} :=\big\{(y,y_{\Gamma})\in(\mathcal{H}^{1})'\;:\;\overline{m}(y,y_{\Gamma})=0\big\},\quad \text{if}\ L\in(0,+\infty),\\
		\mathcal{V}_{(0)}^{-1} :=\big\{(y,y_{\Gamma})\in(\mathcal{V}^{1})'\;:\;\overline{m}(y,y_{\Gamma})=0\big\},\quad \ \,\text{if}\ L=0,
	\end{cases}
	$$
	where $\overline{m}$ is given by \eqref{gmean} if $L\in (0,+\infty)$, and for $L=0$, we take
	$$
	\overline{m}(y,y_{\Gamma})=\frac{\langle (y,y_{\Gamma}),(1,1)\rangle_{(\mathcal{V}^{1})', \mathcal{V}^{1}}}{|\Omega|+|\Gamma|}.
	$$
	Then the chain of inclusions holds
	$$
	\mathcal{H}^{1}_{L,0} \subset \mathcal{L}_{(0)}^2\subset \mathcal{H}_{L,0}^{-1} \subset (\mathcal{H}_{L}^{1})'.
	$$
	It has been shown in \cite[Theorem 3.3]{KL} that for every $(y,y_{\Gamma})\in\mathcal{H}_{L,0}^{-1}$,
    problem \eqref{Elliptic} admits a unique weak solution $(u,u_\Gamma)\in\mathcal{H}_{L,0}^{1}$ satisfying the weak formulation
	\begin{align}
		a_L((u,u_{\Gamma}),(\zeta,\zeta_{\Gamma})) = \langle(y,y_{\Gamma}),(\zeta,\zeta_{\Gamma})\rangle_{(\mathcal{H}_L^{1})',\mathcal{H}_L^{1}},
		\quad \forall\, (\zeta,\zeta_{\Gamma})\in\mathcal{H}_L^{1},
		\notag
	\end{align}
	and the $\mathcal{H}^{1}$-estimate
	\begin{align}
		\|(u,u_{\Gamma})\|_{\mathcal{H}^1}\leq C\|(y,y_{\Gamma})\|_{(\mathcal{H}^{1}_L)'},\notag
	\end{align}
	for some constant $C>0$ depending only on $L$ and $\Omega$.
    Furthermore, if the domain $\Omega$ is of class $C^{k+2}$ and $(y,y_{\Gamma})\in \mathcal{H}_{L,0}^{k}$, $k\in \mathbb{N}$,
    then $(u,u_\Gamma)\in \mathcal{H}^{k+2}$ and the following regularity estimate holds
	\begin{align}
		\|(u,u_{\Gamma})\|_{\mathcal{H}^{k+2}}\leq C\|(y,y_{\Gamma})\|_{\mathcal{H}^{k}}.
		\notag
	\end{align}
	The above facts enable us to define the solution operator $$\mathfrak{S}^{L}:\mathcal{H}_{L,0}^{-1}\rightarrow\mathcal{H}_{L,0}^{1},\quad(y,y_{\Gamma})\mapsto (u,u_\Gamma)=\mathfrak{S}^{L}(y,y_{\Gamma})=(\mathfrak{S}^{L}_{\Omega}(y,y_{\Gamma}),\mathfrak{S}^{L}_{\Gamma}(y,y_{\Gamma})).
	$$
	Similar results for the special case $L=0$ have also been presented in \cite{CF15}.
    A direct calculation yields that
	$$
	((u,u_{\Gamma}), (z,z_\Gamma))_{\mathcal{L}^2}
	=((u,u_{\Gamma}), \mathfrak{S}^{L}(z,z_\Gamma))_{\mathcal{H}^1_{L,0}},\quad \forall\, (u,u_{\Gamma})\in \mathcal{H}_{L,0}^1,\ (z,z_\Gamma)\in \mathcal{L}^2_{(0)}.
	$$
	Thanks to \cite[Corollary 3.5]{KL}, we can introduce the inner product on $\mathcal{H}_{L,0}^{-1}$  as
	\begin{align}
		((y,y_{\Gamma}),(z,z_{\Gamma}))_{L,0,*}&:=(\mathfrak{S}^{L}(y,y_{\Gamma}),\mathfrak{S}^{L}(z,z_{\Gamma}))_{\mathcal{H}^{1}_{L,0}},
		\quad \forall\, (y,y_{\Gamma}), (z,z_{\Gamma})\in \mathcal{H}_{L,0}^{-1}.\notag
	\end{align}
	The associated norm $\Vert(y,y_{\Gamma})\Vert_{L,0,*} :=((y,y_{\Gamma}),(y,y_{\Gamma}))_{L,0,*}^{1/2}$
	is equivalent to the standard dual norm $\|\cdot\|_{(\mathcal{H}_L^1)'}$ on $\mathcal{H}_{L,0}^{-1}$.
	Then it follows that
	\begin{align}
		\|(y,y_{\Gamma})\|_{L,\ast}&:=\left(\Vert(y,y_{\Gamma})-\overline{m}(y,y_{\Gamma}) \mathbf{1}\Vert_{L,0,*}^2+ |\overline{m}(y,y_{\Gamma})|^2\right)^{1/2},
		\quad \forall\, (y,y_{\Gamma})\in (\mathcal{H}_L^{1})',\notag
	\end{align}
	is equivalent to the usual dual norm $\|\cdot\|_{(\mathcal{H}_L^1)'}$ on $(\mathcal{H}_L^{1})'$.
   Finally, let us introduce the following higher-order function space
    \[\mathcal{W}_{L,\mathbf{n}}^2:=\{\boldsymbol{z}=(z,z_\Gamma)\in \mathcal{H}^2:L\partial_\mathbf{n} z=z_\Gamma-z\ \text{ a.e. on }\ \Gamma\}.\]
    Then, we have the following density result.
    \begin{lemma}
        \label{density}
        For any $L\in(0,+\infty)$, $\mathcal{W}_{L,\mathbf{n}}^2$ is dense in $\mathcal{H}^1$.
        In particular, the following chain of inclusions holds
        \begin{align}
            \mathcal{W}_{L,\mathbf{n}}^2\subset \mathcal{H}^1\subset\mathcal{L}^2\subset(\mathcal{H}^1)'\subset(\mathcal{W}_{L,\mathbf{n}}^2)'.\label{chain}
        \end{align}
    \end{lemma}
    \begin{proof}
        Let $\boldsymbol{z}\in\mathcal{H}^1$ be arbitrary, for any $n\in\mathbb{Z}^+$,
        there exists a $\boldsymbol{z}_n\in\mathcal{H}^2$ such that $\|\boldsymbol{z}_n-\boldsymbol{z}\|_{\mathcal{H}^1}\leq 1/2n$.
        Then, we consider the following bulk-surface parabolic system
        \begin{align}
            \begin{cases}
                \partial_t u-\Delta u=0,&\text{in }\Omega\times(0,1),\\
                L\partial_\mathbf{n}u=u_\Gamma-u,&\text{on }\Gamma\times(0,1),\\
                \partial_t u_\Gamma-\Delta_\Gamma u_\Gamma+\partial_\mathbf{n}u=0,&\text{on }\Gamma\times(0,1),\\
                (u,u_\Gamma)|_{t=0}=(z_n,z_{\Gamma,n}),&\text{in }\Omega\times\Gamma.
            \end{cases}\label{dense-1}
        \end{align}
By the standard Faedo--Galerkin method, we can show that problem \eqref{dense-1} admits a unique weak solution $\boldsymbol{u}_n$ satisfying $\boldsymbol{u}_n\in {L}^2(0,1;\mathcal{H}^1)\cap L^\infty(0,1;\mathcal{L}^2)$ and $\partial_t\boldsymbol{u}_n\in L^2(0,1;(\mathcal{H}^1)')$. As the initial datum $\boldsymbol{z}_n\in\mathcal{H}^2$, we can improve the regularity of $\boldsymbol{u}_n$ and conclude that
        \[\boldsymbol{u}_n\in L^2(0,1;\mathcal{H}^3)\cap L^\infty(0,1;\mathcal{H}^2),\quad\partial_t \boldsymbol{u}_n\in L^2(0,1;\mathcal{H}^1)\cap L^\infty(0,1;\mathcal{L}^2).\]
        By the Lions--Magenes theorem, it holds
        $\boldsymbol{u}_n\in C([0,1];\mathcal{H}^2)$.
        Then, there exists a sufficiently large $k\in\mathbb{Z}^+$ such that $\|\boldsymbol{u}_n(1/{k})-\boldsymbol{z}_n\|_{\mathcal{H}^2}\leq 1/2n$.
        Consequently, we conclude that
        $$\left\|\boldsymbol{z}-\boldsymbol{u}_n\left(\frac{1}{k}\right)\right\|_{\mathcal{H}^1}\leq \|\boldsymbol{z}-\boldsymbol{z}_n\|_{\mathcal{H}^1} +\left\|\boldsymbol{z}_n-\boldsymbol{u}_n\left(\frac{1}{k}\right)\right\|_{\mathcal{H}^1}\leq \frac{1}{n},$$
         and $\boldsymbol{u}_n(1/k)\in \mathcal{W}_{L,\mathbf{n}}^2$.
         As a consequence, $\mathcal{W}_{L,\mathbf{n}}^2$ is dense in $\mathcal{H}^1$,
         and the chain of inclusions \eqref{chain} holds. The proof of Lemma \ref{density} is complete.
    \end{proof}

    \begin{remark}\rm
    \label{compact-embedding}
        The property \eqref{chain} together with the Aubin--Lions--Simon lemma
        implies that the function space $\mathbb{V}_1$ is compactly embedded into $\mathbb{V}$,
        where
        $$\mathbb{V}_1:=L^2(0,T;\mathcal{L}^2)\cap H^1(0,T;(\mathcal{W}_{L,\mathbf{n}}^2)'),\quad\mathbb{V}:=L^2(0,T;(\mathcal{H}^1)'),$$
        for any $T\in(0,+\infty)$.
        The compact embedding $\mathbb{V}_1\hookrightarrow\hookrightarrow\mathbb{V}$ is crucial for us to apply Lemma \ref{abstract} to prove the existence of (discrete) exponential attractors. We note that Lemma \ref{density} does not hold in the case $L=0$. Instead, $\mathcal{W}_{0,\mathbf{n}}^2=\mathcal{V}^2$ is dense in $\mathcal{V}^1$ when $L=0$.
    \end{remark}

	\subsection{Problem setting}
    For an arbitrary but given final time $T\in(0,+\infty)$,
    we denote
    $$Q_T:=\Omega\times(0,T)\quad \text{and}\quad \Sigma_T:=\Gamma\times(0,T).
    $$
    If $T=+\infty$, we simply set $$Q:=\Omega\times(0,+\infty)\quad \text{and}\quad \Sigma:=\Gamma\times(0,+\infty).
    $$
	In view of the decomposition of the bulk and surface potentials
	$$F=\widehat{\beta}+\widehat{\pi},\quad G=\widehat{\beta}_{\Gamma}+\widehat{\pi}_{\Gamma},$$
	we reformulate our target problem as follows:
	\begin{align}
		\begin{cases}
		\partial_{t}\varphi=\Delta\mu,&\text{in }Q,\\
		\mu=a_{\Omega}\varphi-J\ast\varphi+\beta(\varphi)+\pi(\varphi),&\text{in }Q,\\
		\partial_{t}\psi=\Delta_{\Gamma}\theta-\partial_{\mathbf{n}}\mu,&\text{on }\Sigma,\\
		\theta=a_{\Gamma}\psi-K\circledast\psi+\beta_{\Gamma}(\psi)+\pi_{\Gamma}(\psi),&\text{on }\Sigma,\\
		L\partial_{\mathbf{n}}\mu=\theta-\mu,\quad L\in[0,+\infty),
		&\text{on }\Sigma,\\
		\varphi|_{t=0}=\varphi_{0},&\text{in }\Omega,\\
		\psi|_{t=0}=\psi_{0},&\text{on }\Gamma.
		\end{cases}\label{model}
	\end{align}
	Then the total free energy of the system \eqref{model} can be expressed equivalently as
	\begin{align}
		E\big(\boldsymbol{\varphi}\big)&=\frac{1}{2}\int_{\Omega}a_{\Omega}\varphi^{2}\,\mathrm{d}x-\frac{1}{2}\int_{\Omega}(J\ast\varphi)\varphi\,\mathrm{d}x+\int_{\Omega}(\widehat{\beta}(\varphi)+\widehat{\pi}(\varphi))\,\mathrm{d}x\notag\\
		&\quad+\frac{1}{2}\int_{\Gamma}a_{\Gamma}\psi^{2}\,\mathrm{d}S-\frac{1}{2}\int_{\Gamma}(K\circledast\psi)\psi\,\mathrm{d}S+\int_{\Gamma}(\widehat{\beta}_{\Gamma}(\psi)+\widehat{\pi}_{\Gamma}(\psi))\,\mathrm{d}S.\notag
	\end{align}
	Throughout this paper, we make the following basic assumptions.
	\begin{description}
		\item[$\mathbf{(A1)}$] The convolution kernels $J,K:\mathbb{R}^{d}\rightarrow\mathbb{R}$ are even,
        i.e., $J(x)=J(-x)$ and $K(x)=K(-x)$ for almost all $x\in\mathbb{R}^{d}$,
        nonnegative almost everywhere and satisfy $J\in W^{1,1}(\mathbb{R}^{d})$ and $K\in W^{2,r}(\mathbb{R}^{d})$ with $r>1$.
        We note that the regularity assumption on $K$ is higher than that on $J$
        since the traces $K(x-\cdot)|_\Gamma$ and $\nabla_{\Gamma}K(x-\cdot)|_\Gamma$ must exist
        and belong to $L^r(\Gamma)$ for all $x\in\Gamma$ (cf. \cite{KS}).
        In addition, we suppose that
		\begin{align}
			&a_{\ast}:=\inf_{x\in\Omega}\int_{\Omega}J(x-y)\,\mathrm{d}y>0,&&a_{\circledast}:=\inf_{x\in\Gamma}\int_{\Gamma}K(x-y)\,\mathrm{d}S_{y}>0,\label{2.1}\\
			&a^{\ast}:=\sup_{x\in\Omega}\int_{\Omega}J(x-y)\,\mathrm{d}y<+\infty,&&a^{\circledast}:=\sup_{x\in\Gamma}\int_{\Gamma}K(x-y)\,\mathrm{d}S_{y}<+\infty,\label{2.2}\\
			&b^{\ast}:=\sup_{x\in\Omega}\int_{\Omega}|\nabla J(x-y)|\,\mathrm{d}y<+\infty,&&b^{\circledast}:=\sup_{x\in\Gamma}\int_{\Gamma}|\nabla_{\Gamma}K(x-y)|\,\mathrm{d}S_{y}<+\infty.\label{2.3}
		\end{align}
		\item[$\mathbf{(A2)}$] The nonlinear convex functions  $\widehat{\beta}$, $\widehat{\beta}_{\Gamma}\in C([-1,1])\cap C^{2}(-1,1)$.
        Their derivatives are denoted by $\beta=\widehat{\beta}'$, $\beta_{\Gamma}=\widehat{\beta}_{\Gamma}'$
        such that $\beta$, $\beta_{\Gamma}\in C^{1}(-1,1)$ are monotone increasing functions satisfying
		\begin{align*}
			&\lim_{s\rightarrow-1}\beta(s)=-\infty,\quad\lim_{s\rightarrow1}\beta(s)=+\infty,\\
			&\lim_{s\rightarrow-1}\beta_{\Gamma}(s)=-\infty,\quad\lim_{s\rightarrow1}\beta_{\Gamma}(s)=+\infty,
		\end{align*}
		and the derivatives $\beta'$, $\beta_{\Gamma}'$ fulfill
		$$\beta'(s)\geq\alpha, \quad\beta_{\Gamma}'(s)\geq\alpha,\quad\forall\,s\in(-1,1)
        $$
		for some constant $\alpha>0$.
		We also extend $\widehat{\beta}(s)=\widehat{\beta}_\Gamma(s)=+\infty$ for any $s\notin[-1,1]$.
		Without loss of generality, we assume
        $$
        \widehat{\beta}(0)=\widehat{\beta}_\Gamma(0)=\beta(0)=\beta_\Gamma(0)=0.
        $$
		This also entails that $\widehat{\beta}(s)$, $\widehat{\beta}_\Gamma(s)\geq0$ for all $s\in[-1,1]$.
		\item[$\mathbf{(A3)}$] $\widehat{\pi}$, $\widehat{\pi}_{\Gamma}\in C^{1}(\mathbb{R})$
        and $\pi:=\widehat{\pi}'$, $\pi_{\Gamma}:=\widehat{\pi}_{\Gamma}'$ are differentiable and Lipschitz continuous on $\mathbb{R}$ with Lipschitz constants $\gamma_{1}$ and $\gamma_{2}$, respectively.
		Furthermore, $\gamma_{1}$ and $\gamma_{2}$ satisfy
		\begin{align*}
			0<\gamma_{1}<a_{\ast}+\frac{\alpha}{1+\alpha},\quad0<\gamma_{2}<a_{\circledast}+\frac{\alpha}{1+\alpha}.
		\end{align*}
		\item[$\mathbf{(A4)}$] The initial datum $\boldsymbol{\varphi}_{0}=(\varphi_{0},\psi_{0})\in\mathcal{L}^{2}$ satisfies
        \[\widehat{\beta}(\varphi_{0})\in L^{1}(\Omega),\quad\widehat{\beta}_{\Gamma}(\psi_{0})\in L^{1}(\Gamma)\quad\text{and}\quad m_0=\overline{m}(\boldsymbol{\varphi}_{0})\in(-1,1).\]
	\end{description}

    \begin{remark}\rm
        \label{compact-operator}
        Under the assumption $(\mathbf{A1})$, the operator
        \[\mathbb{J}:(\varphi,\psi)\mapsto(J\ast\varphi,K\circledast\psi)\]
        is self-adjoint and compact from $\mathcal{L}^2$ to itself,
        which is a direct corollary of the compact embedding $\mathcal{H}^1\hookrightarrow\hookrightarrow\mathcal{L}^2$.
        According to $J\in W^{1,1}(\mathbb{R}^d)$, $K\in W^{2,r}(\mathbb{R}^d)$ and the Arzel$\grave{\mathrm{a}}$--Ascoli theorem,
        one can easily check that $\mathbb{J}$ is also compact from $\mathcal{L}^\infty$ to $C(\overline{\Omega})\times C(\Gamma)$.
        These results will be used to verify a generalized version of the \L ojasiewicz--Simon inequality (see Lemma \ref{LS}).
    \end{remark}
	
		\begin{definition}\rm
		\label{weakdefn}
		Let $T\in(0,+\infty)$ be an arbitrary but given final time and $L\in[0,+\infty)$.
        The function pair $(\boldsymbol{\varphi},\boldsymbol{\mu})$ is called a weak solution to problem \eqref{model} on $[0,T]$,
		if the following conditions are fulfilled:
		\begin{description}
			\item[$\mathrm{(i)}$] The functions $(\boldsymbol{\varphi},\boldsymbol{\mu})$ have the following regularity
			\begin{align}
				&\boldsymbol{\varphi}\in H^{1}(0,T;(\mathcal{H}_L^{1})')\cap L^{\infty}(0,T;\mathcal{L}^{2})\cap L^{2}(0,T;\mathcal{H}^{1}),\notag\\[1mm]
				&\mu\in L^{2}(0,T;V),\quad\theta\in L^{2}(0,T;V_{\Gamma}),\notag
			\end{align}
			and
			\begin{align}
				&\varphi\in L^{\infty}(Q_T)\ \text{ with }\ |\varphi(x,t)|<1\ \text{ for a.a. } (x,t)\in Q_T,\notag\\[1mm]
				&\psi\in L^{\infty}(\Sigma_T)\ \text{ with }\ |\psi(x,t)|<1\ \text{ for a.a. } (x,t)\in \Sigma_T.\notag
			\end{align}
			\item[$\mathrm{(ii)}$] The following variational formulation
			\begin{align}
				\langle\partial_{t}\boldsymbol{\varphi},\boldsymbol{z}\rangle_{(\mathcal{H}_L^{1})',\mathcal{H}_L^{1}}
                & =-\int_{\Omega}\nabla\mu\cdot\nabla z\,\mathrm{d}x-\int_{\Gamma}\nabla_{\Gamma}\theta\cdot\nabla_{\Gamma} z_{\Gamma}\,\mathrm{d}S
                \notag\\
                &\quad -\chi(L)\int_\Gamma (\theta-\mu)(z_\Gamma-z)\,\mathrm{d}S,\notag
			\end{align}
			holds for all $\boldsymbol{z}\in\mathcal{H}_L^1$ and almost all $t\in(0,T)$.
			The bulk and boundary chemical potentials $\mu$, $\theta$ satisfy
			\begin{align}
				&\mu=a_{\Omega}\varphi-J\ast\varphi+\beta(\varphi)+\pi(\varphi),\qquad\ \ \text{a.e. in }Q_T,\notag\\
				&\theta=a_{\Gamma}\psi-K\circledast\psi+\beta_{\Gamma}(\psi)+\pi_{\Gamma}(\psi),\quad\text{a.e. on }\Sigma_T.\notag
			\end{align}
            Furthermore, the initial conditions $\varphi|_{t=0}=\varphi_0$ and $\psi|_{t=0}=\psi_0$ are satisfied almost everywhere in $\Omega$ and on $\Gamma$, respectively.
			\item[$\mathrm{(iii)}$] The energy equality
			\begin{align}
				E(\boldsymbol{\varphi}(t))+\int_{0}^{t}\Big(\Vert\nabla\mu(s)\Vert_{H}^{2}+\Vert\nabla_{\Gamma}\theta(s)\Vert_{H_{\Gamma}}^{2}+\chi(L)\Vert\theta(s)-\mu(s)\Vert_{H_{\Gamma}}^{2}\Big)\,\mathrm{d}s= E(\boldsymbol{\varphi}_{0})\label{energyeq}
			\end{align}
			holds for all $t\in[0,T]$.
		\end{description}
	\end{definition}

	As a preliminary result, we have the following proposition on the well-posedness of problem \eqref{model} (see \cite{LvWu-4}).
	\begin{proposition}
	\label{well-posedness}
		Let $\Omega\subset\mathbb{R}^{d}$ $(d\in\{2,3\})$ be a bounded domain with smooth boundary $\Gamma=\partial\Omega$
        and $T\in(0,+\infty)$ be an arbitrary but given final time.	
        Suppose that the assumptions $\mathbf{(A1)}$--$\mathbf{(A4)}$ hold.
        Then, problem \eqref{model} admits a weak solution in the sense of Definition \ref{weakdefn}.
        The following continuous dependence estimate implies that the weak solution is unique: 	
        let $(\boldsymbol{\varphi}_{i},\boldsymbol{\mu}_{i})$, $i\in\{1,2\}$,
		be two weak solutions to problem \eqref{model}
		corresponding to the initial data $\boldsymbol{\varphi}_{0,i}$ with
		\begin{align*}
				\overline{m}(\boldsymbol{\varphi}_{0,1})=\overline{m}(\boldsymbol{\varphi}_{0,2})= m_0.
		\end{align*}
		Then, for all $t\in[0,T]$, it holds
		\begin{align}
			&\Vert\boldsymbol{\varphi}_{1}(t)-\boldsymbol{\varphi}_{2}(t)\Vert_{L,0,\ast}^{2}+\int_{0}^{t}\Vert\boldsymbol{\varphi}_{1}(s)-\boldsymbol{\varphi}_{2}(s)\Vert_{\mathcal{L}^{2}}^{2}\,\mathrm{d}s\leq C\Vert\boldsymbol{\varphi}_{0,1}-\boldsymbol{\varphi}_{0,2}\Vert_{L,0,\ast}^{2}, \label{contiesti}
		\end{align}
		 where the positive constant $C$ depends on the coefficients in assumptions, $\Omega$, $\Gamma$ and $T$.
	\end{proposition}
	
	\noindent\textbf{A sketch of the proof for Proposition \ref{well-posedness}.}
		We first focus on the existence of global weak solutions. When $L\in(0,+\infty)$,
        we consider the following approximating problem with the singular potentials $\beta$, $\beta_\Gamma$
        replaced by their Yosida approximations $\beta_\varepsilon$, $\beta_{\Gamma,\varepsilon}$:
		\begin{align}
			\begin{cases}
			\partial_{t}\varphi_\varepsilon =\Delta\mu_\varepsilon, &\text{in }Q_T,\\
			\mu_\varepsilon =a_{\Omega}\varphi_\varepsilon-J\ast\varphi_\varepsilon+\beta_{\varepsilon}(\varphi_\varepsilon)+\pi(\varphi_\varepsilon),&\text{in }Q_T,\\
			\partial_{t}\psi_\varepsilon =\Delta_{\Gamma}\theta_\varepsilon-\partial_{\mathbf{n}}\mu_\varepsilon,&\text{on }\Sigma_T,\\
			\theta_\varepsilon=a_{\Gamma} \psi_\varepsilon-K\circledast\psi_\varepsilon+\beta_{\Gamma,\varepsilon}(\psi_\varepsilon) +\pi_{\Gamma}(\psi_\varepsilon), &\text{on }\Sigma_T,\\
			L\partial_{\mathbf{n}}\mu_\varepsilon =\theta_\varepsilon-\mu_\varepsilon, &\text{on }\Sigma_T,\\
			(\varphi_\varepsilon, \psi_\varepsilon)|_{t=0} =(\varphi_{0},\psi_0),&\text{in }\Omega\times\Gamma,
			\end{cases}\label{appro-model}
		\end{align}
        where $\varepsilon\in(0,\varepsilon^\ast)$, and
        \[0< \varepsilon^\ast\leq \min\Big\{\frac{1}{2\|J\|_{L^1(\Omega)}+2\gamma_1+1},\frac{1}{2\|K\|_{L^1(\Gamma)}+2\gamma_2+1}\Big\}<1.\]
	Then problem \eqref{appro-model} can be solved by a suitable Faedo--Galerkin scheme.
    After deriving \emph{a priori} estimates that are uniform with respect to $\varepsilon\in(0,\varepsilon^\ast)$,
    we can conclude the existence of a global weak solution to problem \eqref{model} with $L\in(0,+\infty)$
    by passing to the limit $\varepsilon \to 0$ with the aid of compactness arguments.
    %
    When $L=0$, the existence of weak solutions can be obtained by studying the asymptotic limit as $L\to0$.
    Finally, the continuous dependence estimate \eqref{contiesti} can be derived by the standard energy method. For further details, see \cite{LvWu-4}.
    \hfill$\square$

	\subsection{Statement of results}
   	In this section, we state our main results about the long-time behavior of global weak solutions to problem \eqref{model}.
    First, we introduce the dynamical system associated with problem \eqref{model}. For any $m\in[0,1)$, define the phase space
	\begin{align}
		\mathfrak{X}_m=
			\big\{ \boldsymbol{\varphi}:=(\varphi,\psi) \in \mathcal{L}^2:\, \widehat{\beta}(\varphi)\in L^1(\Omega),\ \widehat{\beta}_\Gamma(\psi)\in L^1(\Gamma)\ \text{ and }\ |\overline{m}(\boldsymbol{\varphi})|\leq m\big\}\notag
	\end{align}
	endowed with the metric
	\begin{align}
		\mathbf{d}(\boldsymbol{\varphi}_1,\boldsymbol{\varphi}_2) =\|\boldsymbol{\varphi}_1-\boldsymbol{\varphi}_2\|_{\mathcal{L}^2} 
+ \int_\Omega\big|\widehat{\beta}(\varphi_1) -\widehat{\beta}(\varphi_2)\big|\,\mathrm{d}x +\int_\Gamma\big|\widehat{\beta}_\Gamma(\psi_1) -\widehat{\beta}_\Gamma(\psi_2)\big|\,\mathrm{d}S.
\notag
	\end{align}
    Applying the same argument as that for \cite[Lemma 3.8]{RS04}, we can check that $(\mathfrak{X}_m,\mathbf{d})$ is a complete metric space. In addition, thanks to Proposition \ref{well-posedness}, we can define the semigroup
	$$
    \mathcal{S}^L(t): \mathfrak{X}_m\to\mathfrak{X}_m,\quad \mathcal{S}^L(t)\boldsymbol{\varphi}_0=\boldsymbol{\varphi}^L(t),\quad\forall\,t\geq0,
    $$
	where $\boldsymbol{\varphi}^L$ is the unique global weak solution to problem \eqref{model} with $L\in[0,+\infty)$,
    corresponding to the initial datum  $\boldsymbol{\varphi}_0\in\mathfrak{X}_m$.

Our first result concerns the existence of the global attractor when $L\in[0,+\infty)$.
In order to deal with the case $L=0$, we need the following additional assumption:
\begin{itemize}
    \item [$(\mathbf{A5})$] $J\in W^{1,1}(\mathbb{R}^d)\cap L^2(\mathbb{R}^d)$.
\end{itemize}
	\begin{theorem}
		\label{global}
		Let $\Omega\subset \mathbb{R}^d$ $(d\in\{2,3\})$ be a bounded domain with a smooth boundary $\Gamma=\partial\Omega$,
        the assumptions $(\mathbf{A1})$--$\mathbf{(A4)}$ be satisfied and $L\in[0,+\infty)$.
        Assume in addition that the assumption $(\mathbf{A5})$ holds when $L=0$.
        Then, for every fixed $m\in[0,1)$, the dynamical system $(\mathfrak{X}_m,\mathcal{S}^L(t))$ has a connected global attractor $\mathcal{A}_m^L$ that is bounded in $\mathfrak{X}_m\cap \mathcal{H}^1$.
	\end{theorem}

    \begin{remark}\rm
        To prove the existence of a global attractor,
        we need to establish some dissipative estimates and show the existence of an absorbing set in $\mathcal{H}^1$.
        Such dissipative estimates for the case $L\in(0,+\infty)$ can be obtained by deriving a similar estimate
        for the approximating system \eqref{appro-model} and passing to the limit as $\varepsilon\to0$.
        The additional assumption $(\mathbf{A5})$ enables us to derive dissipative estimates that are uniform with respect to $L\in(0,1)$,
        thus we can pass to the limit as $L\to0$ to obtain the dissipative estimates for the case $L=0$.
        Please see the proof of Lemma \ref{regular-absorbing} or \cite[Lemma 5.5]{LvWu-4} for more details.
    \end{remark}
	
	In order to prove the existence of exponential attractors and convergence to a single equilibrium,
    the strict separation property of global weak solutions plays a significant role.
    {As in \cite{LvWu-4}}, we make the following additional assumptions:
	
	\begin{description}
		\item[$\mathbf{(A6)}$] The bulk and boundary potentials coincide, i.e., $\beta=\beta_\Gamma$.
		
		\item[$\mathbf{(A7)}$] As $\delta\rightarrow0$, for some constant $\kappa_\ast$, it holds
		\begin{align*}
			&	\frac{1}{\beta(1-2\delta)}=O\Big(\frac{1}{|\mathrm{ln}(\delta)|^{\kappa_\ast}}\Big),\quad\frac{1}{|\beta(-1+2\delta)|}=O\Big(\frac{1}{|\mathrm{ln}(\delta)|^{\kappa_\ast}}\Big),
		\end{align*}
		with $\kappa_\ast>0$ if $d=3$ and $\kappa_\ast>1/2$ if $d=2$.
		
		\item[$\mathbf{(A8)}$] There exist $\delta_{0}\in(0,1/2)$ and $\widetilde{C}_0\geq1$ such that for any $\delta\in(0,\delta_{0})$, it holds
		\begin{align*}
			&\frac{1}{\beta'(1-2\delta)}\leq \widetilde{C}_0\delta,\quad\frac{1}{\beta'(-1+2\delta)}\leq \widetilde{C}_0\delta.
		\end{align*}
		
		\item[$\mathbf{(A9)}$] There exists $\delta_{1}\in (0,1)$ such that $\beta'$ is monotone non-decreasing on $[1-\delta_{1},1)$
        and non-increasing on $(-1,-1+\delta_{1}]$. 
	\end{description}

    \begin{remark}
        \rm
        In assumption $(\mathbf{A7})$, we only need to assume $\kappa_\ast>0$ in the three-dimensional case.
        This is because the other two stronger assumptions $(\mathbf{A8})$ and $(\mathbf{A9})$ need to be imposed
        in the proof of the instantaneous strict separation property in three dimensions (see \cite[Remark 2.4]{LvWu-4} for further details).
    \end{remark}

\begin{theorem}
    \label{exponential}
   Let $\Omega\subset \mathbb{R}^d$ $(d\in\{2,3\})$ be a bounded domain with a smooth boundary $\Gamma=\partial\Omega$, the assumptions $(\mathbf{A1})$--$\mathbf{(A4)}$ be satisfied and $L\in(0,+\infty)$. In addition, we assume
   \begin{itemize}
       \item [(1)] If $d=2$, $(\mathbf{A6})$ and $(\mathbf{A7})$ hold.

       \item [(2)] If $d=3$, $(\mathbf{A6})$--$(\mathbf{A9})$ hold.
   \end{itemize}
   Then, for every fixed $m\in[0,1)$, there exists an exponential attractor $\mathcal{E}^L_m$ bounded in $\mathcal{H}^1$
   for the dynamical system $(\mathfrak{X}_m,\mathcal{S}^L(t))$ that satisfies the following properties:
    \begin{itemize}
        \item [(i)] Semi-invariance: $\mathcal{S}^L(t)\mathcal{E}^L_m\subset\mathcal{E}^L_m$ \ \ for every $t\geq0$.

        \item [(ii)] Exponential attraction property: for any $\nu\in(0,1)$ and $q\in(2,+\infty)$,
        there exists a constant $\kappa_{\nu,q}>0$
        and a positive monotone increasing function $\mathbb{M}_{\nu,q}$ such that,
        for every bounded set $B\subset \mathfrak{X}_m$ with $R=\sup_{\boldsymbol{b}\in B}\|\boldsymbol{b}\|_{\mathcal{L}^2}$,
        it holds
        \[\mathrm{dist}_{\mathcal{H}^{1-\nu}\cap\mathcal{L}^q}(\mathcal{S}^L(t)B,\mathcal{E}^L_m)\leq \mathbb{M}_{\nu,q}(R) e^{-\kappa_{\nu,q} t},\quad\forall\,t\geq0,\]
       where $\mathrm{dist}_{\mathcal{H}^{1-\nu}\cap \mathcal{L}^q}$ denotes the non-symmetric Hausdorff semidistance
       between sets with respect to the norm of $\mathcal{H}^{1-\nu}\cap \mathcal{L}^q$
       defined as $$\mathrm{dist}_{\mathcal{H}^{1-\nu}\cap \mathcal{L}^q}(A,B)=\sup_{\boldsymbol{a}\in A}\inf_{\boldsymbol{b}\in B}\|\boldsymbol{a}-\boldsymbol{b}\|_{\mathcal{H}^{1-\nu}}+\sup_{\boldsymbol{a}\in A}\inf_{\boldsymbol{b}\in B}\|\boldsymbol{a}-\boldsymbol{b}\|_{\mathcal{L}^{q}}.$$

        \item[(iii)] Finite fractal dimension: for any $\nu\in(0,1)$ and $q\in(2,+\infty)$,
        there exist two positive constants $C_{m,\nu}$ and $C_{m,q}$ such that
        \[\mathrm{dim}_{\mathrm{F},\mathcal{H}^{1-\nu}}(\mathcal{E}^L_m)\leq C_{m,\nu}<+\infty,\quad\mathrm{dim}_{\mathrm{F},\mathcal{L}^{q}}(\mathcal{E}^L_m)\leq C_{m,q}<+\infty.\]
    \end{itemize}
\end{theorem}

In view of the above theorem, we can immediately deduce that
\begin{corollary}
\label{finite-global}
    Under the assumptions of Theorem \ref{exponential}, the global attractor $\mathcal{A}^L_m$ is bounded in $\mathcal{H}^1$ and has finite fractal dimension, that is,
    \[\mathrm{dim}_{\mathrm{F},\mathcal{H}^{1-\nu}}(\mathcal{A}^L_m)\leq C_{m,\nu}<+\infty,\quad\mathrm{dim}_{\mathrm{F},\mathcal{L}^{q}}(\mathcal{A}^L_m)\leq C_{m,q}<+\infty.\]
\end{corollary}
	
	 The last result says that every global weak solution converges to a single equilibrium as $t\to+\infty$.
	
	\begin{theorem}
		\label{equilibrium}
		  Let the assumptions in Theorem \ref{exponential} hold. In addition, we assume that $\widehat{\beta}$, $\widehat{\beta}_\Gamma$ are real analytic on $(-1,1)$
          and $\widehat{\pi}$, $\widehat{\pi}_\Gamma$ are real analytic on $\mathbb{R}$.
          Then, every global weak solution $\boldsymbol{\varphi}$ to problem \eqref{model} satisfies
		 \begin{align}
		 \lim_{t\to+\infty}\|\boldsymbol{\varphi}(t)-\boldsymbol{\varphi}_\infty\|_{\mathcal{L}^\infty}=0,\label{convergence}
		 \end{align}
		 where $\boldsymbol{\varphi}_\infty$ is a solution to the following stationary problem that is independent of $L$:
		 \begin{align}
		 	\begin{cases}
		 		\mu_\infty=a_\Omega\varphi_\infty-J\ast\varphi_\infty+\beta(\varphi_\infty)+\pi(\varphi_\infty),&\text{a.e. in }\Omega,\\
		 		\theta_\infty=a_{\Gamma}\psi_\infty-K\circledast\psi_\infty+\beta_{\Gamma}(\psi_\infty)+\pi_{\Gamma}(\psi_\infty),&\text{a.e. on }\Gamma,\\
		 		\mu_\infty=\theta_\infty=\text{constant},&\\
		 		\overline{m}(\boldsymbol{\varphi}_\infty)=\overline{m}(\boldsymbol{\varphi}_0),&
		 	\end{cases}\label{stationary}
		 \end{align}
	with
    \begin{align}
	\mu_\infty=\theta_\infty&=\frac{1}{|\Omega|}\int_\Omega(a_\Omega\varphi_\infty-J\ast\varphi_\infty+\beta(\varphi_\infty)+\pi(\varphi_\infty))\,\mathrm{d}x\notag\\
	&=\frac{1}{|\Gamma|}\int_\Gamma(a_{\Gamma}\psi_\infty-K\circledast\psi_\infty+\beta_{\Gamma}(\psi_\infty)+\pi_{\Gamma}(\psi_\infty))\,\mathrm{d}S.\label{mu_s}
\end{align}
	\end{theorem}

Before ending this section, we give some comments on the problem setting and results.

 \begin{remark}\rm
    For the phase space $\mathfrak{X}_m$, we can also define the following metric
    \[\widetilde{\mathbf{d}}(\boldsymbol{\varphi}_1,\boldsymbol{\varphi}_2)=\|\boldsymbol{\varphi}_1-\boldsymbol{\varphi}_2\|_{\mathcal{L}^2} +\Big|\int_\Omega\widehat{\beta}(\varphi_1)-\widehat{\beta}(\varphi_2)\,\mathrm{d}x\Big| +\Big|\int_\Gamma\widehat{\beta}_\Gamma(\psi_1)-\widehat{\beta}_\Gamma(\psi_2)\,\mathrm{d}S\Big|.\]
    A similar metric was also used in the study of the (local) Cahn--Hilliard equation with dynamic boundary conditions (see \cite{FW}). Under assumption $\mathbf{(A2)}$, we can verify that $(\mathfrak{X}_m,\widetilde{\mathbf{d}})$ is a complete metric space. Actually, for a Cauchy sequence $\{\bm{\varphi}_n\}\subset (\mathfrak{X}_m,\widetilde{\mathbf{d}})$ with the limit denoted by $\bm{\varphi}$, we easily find $\bm{\varphi}\in \mathcal{L}^2$. Furthermore, at least for a subsequence (not relabeled for simplicity), it holds $\bm{\varphi}_n\to \bm{\varphi}$ almost everywhere in $\Omega \times \Gamma$.
    Hence, we can apply Fatou's lemma to conclude
    \begin{align}
    \int_\Omega \widehat{\beta}(\varphi)\,\mathrm{d}x \leq \liminf_{n\to +\infty}  \int_\Omega \widehat{\beta}(\varphi_n)\,\mathrm{d}x,
    \quad
     \int_\Gamma \widehat{\beta}_\Gamma(\psi)\,\mathrm{d}S \leq
    \liminf_{n\to +\infty}  \int_\Gamma \widehat{\beta}_\Gamma (\psi_n)\,\mathrm{d}S.
    \label{conv-Fatou}
    \end{align}
    Besides, by Lebesgue's dominated convergence theorem, we have
    $$
    \overline{m}(\bm{\varphi}) = \lim_{n\to +\infty} \overline{m}(\bm{\varphi}_n).
    $$
    As a result, $\bm{\varphi}\in \mathfrak{X}_m$.
Since $\widehat{\beta},\widehat{\beta}_\Gamma\in C([-1,1])$, we can apply Lebesgue's dominated convergence theorem to conclude
    \begin{align}
    \lim_{n\to +\infty}  \int_\Omega \widehat{\beta}(\varphi_n)\,\mathrm{d}x = \int_\Omega \widehat{\beta}(\varphi)\,\mathrm{d}x,
    \quad
    \lim_{n\to +\infty}  \int_\Gamma \widehat{\beta}_\Gamma (\psi_n)\,\mathrm{d}S = \int_\Gamma \widehat{\beta}_\Gamma(\psi)\,\mathrm{d}S.
    \label{conv-betah}
   \end{align}
On the other hand, by the definition of $\widetilde{\mathbf{d}}$, $\{\int_\Omega \widehat{\beta}(\varphi_n)\,\mathrm{d}x\}$, $\{ \int_\Gamma \widehat{\beta}_\Gamma (\psi_n)\,\mathrm{d}S \}$ converge. Hence, \eqref{conv-betah} holds for the whole sequence, which implies that $\bm{\varphi}_n\to \bm{\varphi}$ in $(\mathfrak{X}_m,\widetilde{\mathbf{d}})$.

    For more singular potentials $\widehat{\beta},\widehat{\beta}_\Gamma$ that may not be continuous at $\pm 1$, some differences arise. For example, let us consider the following potential functions (see \cite{MZ04,Schimperna})
    \[
    \widehat{\beta}(s)=\widehat{\beta}_\Gamma(s)=-\text{ln}(1-s^2),\quad s\in(-1,1).
    \]
    It is obvious that $\widehat{\beta}, \widehat{\beta}_\Gamma\in C^2(-1,1)$, but $\widehat{\beta}, \widehat{\beta}_\Gamma\notin C([-1,1])$. We can still verify that  $\bm{\varphi}\in \mathfrak{X}_m$. However, due to the unboundedness of $\widehat{\beta}, \widehat{\beta}_\Gamma$ near $\pm 1$, we are unable to recover the convergence \eqref{conv-betah} using Lebesgue's dominated convergence theorem. Thus, it is not clear whether $\bm{\varphi}_n\to \bm{\varphi}$ in $(\mathfrak{X}_m,\widetilde{\mathbf{d}})$ holds.
    Based on the observation above, we adopt a definition of the metric similar to that in \cite{RS04}, so that our results are valid for a broader class of nonlinearities $\widehat{\beta}, \widehat{\beta}_\Gamma$ (see also Remark \ref{nonlinear-beta}  below).
    \end{remark}

     \begin{remark} \label{nonlinear-beta} \rm
      We note that the logarithmic potential \eqref{log} (with continuous extension at $\pm1$) satisfies the assumptions $\mathbf{(A2)}$, $(\mathbf{A6})$--$\mathbf{(A9)}$. Indeed, the continuity of  $\widehat{\beta}, \widehat{\beta}_\Gamma$ at $\pm 1$ are not essential. For more singular potentials $\widehat{\beta},\widehat{\beta}_\Gamma\in C^2(-1,1)$ instead of $C([-1,1])\cap C^2(-1,1)$, the results established in \cite{LvWu-4} are still valid, including the well-posendess, the regularity propagation and the strict separation property of global weak solutions. Furthermore, one can check that all results obtained in this study hold as well.
      Nevertheless, the stronger singularity of potential functions yields some differences in the analysis, for instance, in the derivation of the energy equality. In \cite{LvWu-4}, we first used the regularizing property of global weak solutions to establish
          \[E(\boldsymbol{\varphi}(t))-E(\boldsymbol{\varphi}(\tau))=-\int_\tau^t \|\partial_t\boldsymbol{\varphi}(s)\|_{L,0,\ast}^2\,\mathrm{d}s,\quad\forall\,t\geq\tau>0.\]
          Then, by the continuity property $\widehat{\beta},\widehat{\beta}_\Gamma\in C([-1,1])$ and Lebesgue's dominated convergence theorem, we could pass to the limit as $\tau\to0$ to obtain the energy equality \eqref{energyeq}.
          For more singular potentials $\widehat{\beta},\widehat{\beta}_\Gamma\in C^2(-1,1)$ only, the argument based on Lebesgue's dominated convergence theorem does not apply. Instead, one can employ the generalized chain rule \cite[Lemma 4.1]{RS04} for maximal monotone operators to establish the energy equality.
\end{remark}

\section{Existence of the Global Attractor for $L\in[0,+\infty)$}
\setcounter{equation}{0}
	
In this section, we establish the existence of the global attractor $\mathcal{A}_m^L$
for the dynamical system $(\mathfrak{X}_m,\mathcal{S}^L(t))$ for any $L\in[0,+\infty)$
and study the stability of the family $\{\mathcal{A}_m^L\}_{L\geq0}$ at $L=0$.

\subsection{Existence}	
\begin{lemma}
\label{dissipative}
Let $(\boldsymbol{\varphi},\boldsymbol{\mu})$ be the unique global weak solution to problem \eqref{model}
subject to the initial datum $\boldsymbol{\varphi}_0\in\mathfrak{X}_m$.
Then, for all $t\geq0$, it holds
\begin{align}
&E(\boldsymbol{\varphi}(t))
+\omega\int_t^{t+1}\|\mathbf{P}\boldsymbol{\mu}(s)\|_{\mathcal{H}_{L,0}^1}^2\,\mathrm{d}s\leq
E(\boldsymbol{\varphi}_0)e^{-\omega t}+M_1\big(1+\widehat{\beta}(\overline{m}(\boldsymbol{\varphi}_0)) +\widehat{\beta}_\Gamma(\overline{m}(\boldsymbol{\varphi}_0))\big),
\label{diss}
\end{align}
where the positive constants $\omega$ and $M_1$ depend on $J$, $K$, $\Omega$, $\Gamma$ and the parameters in \eqref{model},
but are independent of the initial data.
\end{lemma}
\begin{proof}
Let $(\boldsymbol{\varphi}^L,\boldsymbol{\mu}^L)$ be the unique global weak solution to problem \eqref{model} corresponding to $L\in[0,+\infty)$
and $(\boldsymbol{\varphi}_\varepsilon^L,\boldsymbol{\mu}_\varepsilon^L)$ be the unique weak solution to the approximating problem \eqref{appro-model}
corresponding to $(\varepsilon,L)\in(0,\varepsilon^\ast)\times(0,+\infty)$.
The total free energy of the approximating problem is given by
\begin{align}
	E_{\varepsilon}(\boldsymbol{\varphi}^L_\varepsilon)&=\frac{1}{2}\int_{\Omega}a_{\Omega}(\varphi_\varepsilon^L)^{2}\,\mathrm{d}x-\frac{1}{2}\int_{\Omega}(J\ast\varphi^L_\varepsilon)\varphi^L_\varepsilon\,\mathrm{d}x+\int_{\Omega}(\widehat{\beta}_{\varepsilon}(\varphi^L_\varepsilon)+\widehat{\pi}(\varphi^L_\varepsilon))\,\mathrm{d}x\notag\\
	&\quad+\frac{1}{2}\int_{\Gamma}a_{\Gamma}(\psi_\varepsilon^L)^{2}\,\mathrm{d}S-\frac{1}{2}\int_{\Gamma}(K\circledast\psi^L_\varepsilon)\psi^L_\varepsilon\,\mathrm{d}S+\int_{\Gamma}(\widehat{\beta}_{\Gamma,\varepsilon}(\psi^L_\varepsilon)+\widehat{\pi}_{\Gamma}(\psi^L_\varepsilon))\,\mathrm{d}S.\notag
\end{align}
We first derive dissipative estimates for the approximating solutions.
Let us claim that there exists $\overline{\varepsilon}\in(0,\varepsilon^\ast)$ such that,
for all $\varepsilon\in(0,\overline{\varepsilon})$ and $t\geq0$, it holds
\begin{align}
	&E_\varepsilon(\boldsymbol{\varphi}_\varepsilon^L(t)) +\omega\int_t^{t+1}\|\mathbf{P} \boldsymbol{\mu}_\varepsilon^L(s)\|_{\mathcal{H}_{L,0}^1}^2\,\mathrm{d}s \leq
	E_\varepsilon(\boldsymbol{\varphi}_0)e^{-\omega t}+M_1(1+\widehat{\beta}_\varepsilon(\overline{m}(\boldsymbol{\varphi}_0))
    +\widehat{\beta}_{\Gamma,\varepsilon}(\overline{m}(\boldsymbol{\varphi}_0))),
    \label{approdiss}
\end{align}
where the positive constants $\omega$ and $M_1$ depend on $J$, $K$, $\Omega$, $\Gamma$, and the parameters in system \eqref{model},
but are independent of the initial data and $\varepsilon$.
Below we provide a formal proof of \eqref{approdiss} and a rigorous justification can be done
by performing the same computations within a Galerkin approximation scheme
(see the proof of\cite[Proposition 3.1]{LvWu-4} for details).
According to \cite[Lemma 3.11]{GGG}, for any $\varepsilon\in(0,\overline{\varepsilon})$, it holds
\begin{align*}
	E_\varepsilon(\boldsymbol{z})\geq\Big( \frac{1}{4\overline{\varepsilon}}-\frac{\|J\|_{L^1(\Omega)}+\|K\|_{L^1(\Gamma)}}{2}-(\gamma_{1}+\gamma_{2})\Big)\|\boldsymbol{z}\|_{\mathcal{L}^2}^2-C(|\Omega|+|\Gamma|),
\end{align*}
where the constant $C>0$ depends on $\overline{\varepsilon}$, but is independent of $\varepsilon\in(0,\overline{
\varepsilon})$.
Therefore, for any $\Lambda>0$, there exists a constant ${C}_\Lambda>0$ such that
\begin{align}
	E_\varepsilon(\boldsymbol{z})\geq\Lambda\|\boldsymbol{z}\|_{\mathcal{L}^2}^2-{C_\Lambda}(|\Omega|+|\Gamma|),\quad\text{for all }\varepsilon\in(0,\overline{\varepsilon}),\label{below}
\end{align}
provided that $\overline{\varepsilon}$ is small enough.
%
%
Recalling that the following energy equality holds (cf. \cite[(3.26)]{LvWu-4})
\begin{align}
\frac{\mathrm{d}}{\mathrm{d}t}E_\varepsilon	(\boldsymbol{\varphi}_\varepsilon^L)+\|\nabla\mu_\varepsilon^L\|_H^2+\|\nabla_{\Gamma}\theta_\varepsilon^L\|_{H_\Gamma}^2+\frac{1}{L}\|\theta_\varepsilon^L-\mu_\varepsilon^L\|_{H_\Gamma}^2=0 \quad\text{for all }t>0.\label{energyequ}
\end{align}
In order to reconstruct the energy functional on the left-hand side,
testing \eqref{appro-model}$_2$ by $\varphi_\varepsilon^L-\overline{m}(\boldsymbol{\varphi}_0)$
and \eqref{appro-model}$_4$ by $\psi_\varepsilon^L-\overline{m}(\boldsymbol{\varphi}_0)$, we obtain
\begin{align}
	&\int_{\Omega}\beta_\varepsilon(\varphi_\varepsilon^L)(\varphi_\varepsilon^L-\overline{m}(\boldsymbol{\varphi}_0))\,\mathrm{d}x+\int_\Gamma \beta_{\Gamma,\varepsilon}(\psi_\varepsilon^L)(\psi_\varepsilon^L-\overline{m}(\boldsymbol{\varphi}_0))\,\mathrm{d}S \notag\\
	&\quad=\int_\Omega\mu_\varepsilon^L(\varphi_\varepsilon^L- \overline{m}(\boldsymbol{\varphi}_0))\,\mathrm{d}x +\int_\Gamma \theta_\varepsilon^L(\psi_\varepsilon^L- \overline{m}(\boldsymbol{\varphi}_0))\,\mathrm{d}S \notag\\
	&\qquad+\int_\Omega (J\ast\varphi_\varepsilon^L)(\varphi_\varepsilon^L- \overline{m}(\boldsymbol{\varphi}_0))\,\mathrm{d}x +\int_\Gamma (K\circledast\psi_\varepsilon^L)(\psi_\varepsilon^L- \overline{m}(\boldsymbol{\varphi}_0))\,\mathrm{d}S \notag\\
	&\qquad-\int_\Omega a_\Omega\varphi_\varepsilon^L(\varphi_\varepsilon^L- \overline{m}(\boldsymbol{\varphi}_0))\,\mathrm{d}x-\int_\Omega\pi(\varphi_\varepsilon^L)(\varphi_\varepsilon^L- \overline{m}(\boldsymbol{\varphi}_0))\,\mathrm{d}x \notag\\
	&\qquad-\int_\Gamma a_\Gamma\psi_\varepsilon^L(\psi_\varepsilon^L- \overline{m}(\boldsymbol{\varphi}_0))\,\mathrm{d}S-\int_\Gamma \pi_\Gamma(\psi_\varepsilon^L)(\psi_\varepsilon^L- \overline{m}(\boldsymbol{\varphi}_0))\,\mathrm{d}S.
    \label{L-3}
\end{align}
By the generalized Poincar\'e inequality \eqref{Poin}, the first line on the right-hand side of \eqref{L-3} can be estimated as follows:
\begin{align}
	&\int_\Omega\mu_\varepsilon^L(\varphi_\varepsilon^L-\overline{m}(\boldsymbol{\varphi}_0))\,\mathrm{d}x +\int_\Gamma \theta_\varepsilon^L(\psi_\varepsilon^L-\overline{m}(\boldsymbol{\varphi}_0))\,\mathrm{d}S \notag\\
	&\quad=\int_\Omega(\mu_\varepsilon^L-\overline{m}(\boldsymbol{\mu}_\varepsilon^L))\varphi_\varepsilon^L\,\mathrm{d}x +\int_\Gamma (\theta_\varepsilon^L-\overline{m}(\boldsymbol{\mu}_\varepsilon^L)) \psi_\varepsilon^L\,\mathrm{d}S \notag\\
	&\quad\leq C_1\Big(\|\nabla\mu_\varepsilon^L\|_H^2 +\|\nabla_{\Gamma}\theta_\varepsilon^L\|_{H_\Gamma}^2+\frac{1}{L}\|\theta_\varepsilon^L-\mu_\varepsilon^L\|_{H_\Gamma}^2\Big)^{\frac{1}{2}}\|\boldsymbol{\varphi}_\varepsilon^L\|_{\mathcal{L}^2}. \label{L-4}
\end{align}
For the other terms on the right-hand side of \eqref{L-3},
by \eqref{2.2}, $\mathbf{(A3)}$, H\"older's inequality and the generalized Poincar\'e inequality \eqref{Poin}, we obtain
\begin{align}
&\int_\Omega (J\ast\varphi_\varepsilon^L)(\varphi_\varepsilon^L- \overline{m}(\boldsymbol{\varphi}_0))\,\mathrm{d}x+\int_\Gamma (K\circledast\psi_\varepsilon^L)(\psi_\varepsilon^L- \overline{m}(\boldsymbol{\varphi}_0))\,\mathrm{d}S\notag\\
&\qquad-\int_\Omega a_\Omega\varphi_\varepsilon^L(\varphi_\varepsilon^L- \overline{m}(\boldsymbol{\varphi}_0))\,\mathrm{d}x
-\int_\Omega\pi(\varphi_\varepsilon^L)(\varphi_\varepsilon^L- \overline{m}(\boldsymbol{\varphi}_0))\,\mathrm{d}x\notag\\
&\qquad-\int_\Gamma a_\Gamma\psi_\varepsilon^L(\psi_\varepsilon^L
-\overline{m}(\boldsymbol{\varphi}_0))\,\mathrm{d}S
-\int_\Gamma \pi_\Gamma(\psi_\varepsilon^L)(\psi_\varepsilon^L- \overline{m}(\boldsymbol{\varphi}_0))\,\mathrm{d}S\notag\\
&\quad\leq C_2\big(1+\|\boldsymbol{\varphi}_\varepsilon^L\|_{\mathcal{L}^2}^2\big).
\label{L-5}
\end{align}
Next, by the convexity of $\widehat{\beta}_\varepsilon$ and $\widehat{\beta}_{\Gamma,\varepsilon}$, it holds
\begin{align*}
	E_\varepsilon(\boldsymbol{\varphi}_\varepsilon^L)&\leq \int_\Omega\beta_\varepsilon(\varphi_\varepsilon^L)(\varphi_\varepsilon^L- \overline{m}(\boldsymbol{\varphi}_0))\,\mathrm{d}x +\int_\Gamma\beta_{\Gamma,\varepsilon}(\psi_\varepsilon^L)(\psi_\varepsilon^L- \overline{m}(\boldsymbol{\varphi}_0))\,\mathrm{d}S\\
	&\quad+\Big(\frac{a_\ast+a_\circledast+\|J\|_{L^1(\Omega)} +\|K\|_{L^1(\Gamma)}}{2}+\gamma_{1}+\gamma_2\Big) \|\boldsymbol{\varphi}_\varepsilon^L\|_{\mathcal{L}^2}^2\\
	&\quad+\widehat{\beta}_\varepsilon( \overline{m}(\boldsymbol{\varphi}_0))|\Omega| +\widehat{\beta}_{\Gamma,\varepsilon}( \overline{m}(\boldsymbol{\varphi}_0))|\Gamma|\\
    &\quad+\left(|\widehat{\pi}(0)|+\frac{|\pi(0)|^2}{2\gamma_1}\right)|\Omega|+\left(|\widehat{\pi}_\Gamma(0)|+\frac{|\pi_\Gamma(0)|^2}{2\gamma_2}\right)|\Gamma|.
\end{align*}
Taking \eqref{L-3}, \eqref{L-4} and \eqref{L-5} into account, we obtain
\begin{align}
	E_\varepsilon(\boldsymbol{\varphi}_\varepsilon^L)&\leq \frac{C_1}{4}\Big(\|\nabla\mu_\varepsilon^L\|_H^2 +\|\nabla_{\Gamma}\theta_\varepsilon^L\|_{H_\Gamma}^2 +\frac{1}{L}\|\theta_\varepsilon^L-\mu_\varepsilon^L\|_{H_\Gamma}^2\Big) \notag\\
	&\quad+C_2+\Big(C_1+C_2+\frac{a_\ast+a_\circledast+\|J\|_{L^1(\Omega)}+\|K\|_{L^1(\Gamma)}}{2}+\gamma_{1}+\gamma_2\Big)\|\boldsymbol{\varphi}_\varepsilon^L\|_{\mathcal{L}^2}^2\notag\\
	&\quad+\widehat{\beta}_\varepsilon( \overline{m}(\boldsymbol{\varphi}_0))|\Omega| +\widehat{\beta}_{\Gamma,\varepsilon}( \overline{m}(\boldsymbol{\varphi}_0))|\Gamma| \notag\\
    &\quad+\left(|\widehat{\pi}(0)|+\frac{|\pi(0)|^2}{2\gamma_1}\right)|\Omega|+\left(|\widehat{\pi}_\Gamma(0)|+\frac{|\pi_\Gamma(0)|^2}{2\gamma_2}\right)|\Gamma|.\label{L-6}
\end{align}
In light of \eqref{below} and \eqref{L-6}, there exists $\overline{\varepsilon}>0$ sufficiently small
such that for any $\varepsilon\in (0,\overline{\varepsilon})$, we have
\begin{align}
	\frac{1}{2}E_\varepsilon(\boldsymbol{\varphi}_\varepsilon^L)&\leq\frac{C_1}{4}\Big(\|\nabla\mu_\varepsilon^L\|_H^2+\|\nabla_{\Gamma}\theta_\varepsilon^L\|_{H_\Gamma}^2+\frac{1}{L}\|\theta_\varepsilon^L-\mu_\varepsilon^L\|_{H_\Gamma}^2\Big)+C_2 \notag\\
&\quad+\left(\widehat{\beta}_\varepsilon( \overline{m}(\boldsymbol{\varphi}_0)) +|\widehat{\pi}(0)|+\frac{|\pi(0)|^2}{2\gamma_1}+{C}_\Lambda\right)|\Omega|\notag\\	&\quad+\left(\widehat{\beta}_{\Gamma,\varepsilon}(\overline{m}(\boldsymbol{\varphi}_0)) +|\widehat{\pi}_\Gamma(0)|+\frac{|\pi_\Gamma(0)|^2}{2\gamma_2}+{C}_\Lambda\right)|\Gamma|.
\label{L-7}
\end{align}
Multiplying \eqref{L-7} by a small positive constant, then adding to \eqref{energyequ},
we find the differential inequality
\begin{align}
&\frac{\mathrm{d}}{\mathrm{d}t}E_\varepsilon(\boldsymbol{\varphi}_\varepsilon^L)+\omega\Big(E_\varepsilon(\boldsymbol{\varphi}_\varepsilon^L) +\|\mathbf{P}\boldsymbol{\mu}_\varepsilon^L\|_{\mathcal{H}_{L,0}^1}^2\Big)\leq C(1+\widehat{\beta}_\varepsilon( \overline{m}(\boldsymbol{\varphi}_0)) +\widehat{\beta}_{\Gamma,\varepsilon}(\overline{m}(\boldsymbol{\varphi}_0))),
\notag
\end{align}
for some $\omega>0$ independent of $\varepsilon$. An application of Gronwall's inequality yields \eqref{approdiss}.
Then passing to the limit as $\varepsilon\to0$, we can conclude \eqref{diss} for the case $L\in(0,+\infty)$.

Finally, we establish \eqref{diss} with $L=0$ by passing to the limit as $L\to0$.
For this purpose, we need to show that the constants $\omega$ and $M_1$ are independent of $L\in(0,1)$.
Examining the above estimates, we find that only the constant $C_1$ in \eqref{L-4} may depend on $L$.
Nevertheless, by the same procedure as we did in \cite[Lemma 4.1]{LvWu-4}, the constant $C_1$ can be refined to be independent of $L\in(0,1)$.
Hence, passing to the limit as $L\to0$ in \eqref{diss} for $L\in(0,1)$, we obtain the dissipative estimate \eqref{diss} for $L=0$.
\end{proof}

 The following lemma gives an improved dissipativity of the semigroup $\mathcal{S}^L(t)$, ensuring the existence of a bounded absorbing set in the more regular space $\mathfrak{X}_m\cap\mathcal{H}^1$.

  \begin{lemma}
 	\label{regular-absorbing}
 There exists a constant $R>0$ such that the ball
 $$\mathcal{B}=B_{\mathcal{H}^1}(\boldsymbol{0},R)\cap \mathfrak{X}_m$$
 is a bounded absorbing set for $\mathcal{S}^L(t)$ in $\mathfrak{X}_m\cap \mathcal{H}^1$,
 where $B_{\mathcal{H}^1}(\boldsymbol{0},R)$ denotes the ball in $\mathcal{H}^1$ centered at $\boldsymbol{0}$ with radius $R$.
 Namely, for every bounded set $\mathcal{B}_0\subset\mathfrak{X}_m$, there exists a time $t_0:=t_0(\mathcal{B}_0)>0$ such that
  \begin{align*}
	\mathcal{S}^L(t)\mathcal{B}_0\subset\mathcal{B},\quad\forall\,t\geq t_0.
\end{align*}
\end{lemma}
\begin{proof}
We first consider the case $L\in(0,+\infty)$. Let $(\boldsymbol{\varphi}_\varepsilon^L,\boldsymbol{\mu}_\varepsilon^L)$
be the solution to the approximate problem \eqref{appro-model}
with $(\varepsilon,L)\in(0,\overline{\varepsilon})\times(0,+\infty)$
and $(\boldsymbol{\varphi}^L,\boldsymbol{\mu}^L)$ be the unique global weak solution to problem \eqref{model} with $L\in[0,+\infty)$.
Taking the difference quotient in the approximating system \eqref{appro-model} and
using the same argument as in \cite[(3.68)]{LvWu-4}, we obtain
 \begin{align}
 	\Big\|\frac{\boldsymbol{\varphi}^L_\varepsilon(t+1+h)-\boldsymbol{\varphi}^L_\varepsilon(t+1)}{h}\Big\|_{L,0,\ast}^2\leq C\int_t^{t+1+h}\|\partial_t\boldsymbol{\varphi}^L_\varepsilon(s)\|_{L,0,\ast}^2\,\mathrm{d}s,\quad\forall\,t\geq0,\notag
 \end{align}
 which, together with $\|\mathbf{P}\boldsymbol{\mu}_\varepsilon^L\|_{\mathcal{H}_{L,0}^1}=\|\partial_t\boldsymbol{\varphi}_\varepsilon^L\|_{L,0,\ast}$ and \eqref{approdiss}, implies that
 \begin{align}
 &\Big\|\frac{\boldsymbol{\varphi}^L_\varepsilon(t+1+h)-\boldsymbol{\varphi}^L_\varepsilon(t+1)}{h}\Big\|_{L,0,\ast}^2\notag\\
 &\quad\leq C_3 E_\varepsilon(\boldsymbol{\varphi}_0)e^{-\omega t}+C_3(1+\widehat{\beta}_\varepsilon( \overline{m}(\boldsymbol{\varphi}_0)) +\widehat{\beta}_{\Gamma,\varepsilon}(\overline{m}(\boldsymbol{\varphi}_0))),\quad\forall\,t\geq0,\label{Lv-16}
 \end{align}
 where the constant $C_3>0$ is independent of the initial data and $\varepsilon\in(0,\overline{\varepsilon})$.
 Since the right-hand side of \eqref{Lv-16} is independent of $h\in(0,1)$,
 we can pass to the limit as $h\to0^+$ in \eqref{Lv-16} to obtain
 \begin{align}
 	\|\partial_t\boldsymbol{\varphi}^L_\varepsilon(t+1)\|_{L,0,\ast}^2 \leq C _3 E_\varepsilon(\boldsymbol{\varphi}_0)e^{-\omega t}+C_3(1+\widehat{\beta}_\varepsilon( \overline{m}(\boldsymbol{\varphi}_0)) +\widehat{\beta}_{\Gamma,\varepsilon}(\overline{m}(\boldsymbol{\varphi}_0))), \quad\forall\,t\geq0.
    \notag
 \end{align}
 By the definition of $\mathfrak{S}^L$, we see that
 \begin{align}
 	\|\mathbf{P}\boldsymbol{\mu}^L_\varepsilon(t+1)\|_{\mathcal{H}_{L,0}^1}^2 &=\|\mathfrak{S}^L(\partial_t \boldsymbol{\varphi}^L_\varepsilon(t+1))\|_{\mathcal{H}_{L,0}^1}^2
 \notag\\
    &=\|\partial_t\boldsymbol{\varphi}^L_\varepsilon(t+1)\|_{L,0,\ast}^2
    \notag\\
    &\leq C_3 E_\varepsilon(\boldsymbol{\varphi}_0)e^{-\omega t}+C_3(1+\widehat{\beta}_\varepsilon( \overline{m}(\boldsymbol{\varphi}_0)) +\widehat{\beta}_{\Gamma,\varepsilon}(\overline{m}(\boldsymbol{\varphi}_0))), \quad\forall\,t\geq0.\label{Lv-18}
 \end{align}
Taking the gradient of \eqref{appro-model}$_2$ and testing the resultant by $\nabla\varphi_\varepsilon^L$, we obtain
\begin{align*}
&\int_{\Omega}(a_\Omega+\beta_\varepsilon'(\varphi_\varepsilon^L)+\pi'(\varphi_\varepsilon^L))|\nabla\varphi_\varepsilon^L|^2\,\mathrm{d}x\\
&\quad= \int_\Omega\nabla\mu_\varepsilon^L\cdot\nabla\varphi_\varepsilon^L\,\mathrm{d}x-\int_\Omega\varphi_\varepsilon^L\nabla a_\Omega\cdot\nabla \varphi_\varepsilon^L\,\mathrm{d}x+\int_\Omega(\nabla J\ast\varphi_\varepsilon^L)\cdot\nabla\varphi_\varepsilon^L\,\mathrm{d}x.
\end{align*}
Then, by \eqref{2.1}, \eqref{2.3}, H\"older's inequality and Young's inequality for convolution, we get
\begin{align}
	\frac{\chi_1}{2}\|\nabla\varphi_\varepsilon^L\|_H^2\leq\frac{3}{2\chi_1}\|\nabla\mu_\varepsilon^L\|_H^2+\frac{3}{2\chi_1}\big(b^\ast+\|\nabla J\|_{L^1(\Omega)}^2\big)\|\varphi_\varepsilon^L\|_H^2,\label{L-1}
\end{align}
where $\chi_1=\alpha/(1+\alpha)+a_\ast-\gamma_{1}$. Similarly, it holds
\begin{align}
	\frac{\chi_2}{2}\|\nabla_\Gamma\psi_\varepsilon^L\|_{H_\Gamma}^2\leq\frac{3}{2\chi_2}\|\nabla_\Gamma\theta_\varepsilon^L\|_{H_\Gamma}^2+\frac{3}{2\chi_2}\big(b^\circledast+\|\nabla_\Gamma K\|_{L^1(\Gamma)}^2\big)\|\psi_\varepsilon^L\|_{H_\Gamma}^2,\label{L-2}
\end{align}
where $\chi_2=\alpha/(1+\alpha)+a_\circledast-\gamma_{2}$.
Hence, by \eqref{approdiss}, \eqref{below},  \eqref{Lv-18}, \eqref{L-1}, \eqref{L-2} and the facts
 \[0\leq \widehat{\beta}_\varepsilon(s)\leq \widehat{\beta}(s),\quad 0\leq \widehat{\beta}_{\Gamma,\varepsilon}(s)\leq \widehat{\beta}_\Gamma(s),\quad\forall\,s\in\mathbb{R},\]
 we obtain
 \begin{align}
 	\|\boldsymbol{\varphi}^L_\varepsilon(t+1)\|_{\mathcal{H}^1}^2\leq C E(\boldsymbol{\varphi}_0)e^{-\omega t}+C(1+\widehat{\beta}( \overline{m}(\boldsymbol{\varphi}_0)) +\widehat{\beta}_{\Gamma}( \overline{m}(\boldsymbol{\varphi}_0))), \quad\forall\,t\geq0.\label{Lv-19}
 \end{align}
 Since the right-hand side of \eqref{Lv-19} is independent of $\varepsilon\in(0,\overline{\varepsilon})$,
 after passing to the limit as $\varepsilon\to0$ in \eqref{Lv-19}, we see that
  \begin{align}
 	\|\boldsymbol{\varphi}^L(t+1)\|_{\mathcal{H}^1}^2\leq C E(\boldsymbol{\varphi}_0)e^{-\omega t}+C(1+\widehat{\beta}( \overline{m}(\boldsymbol{\varphi}_0)) +\widehat{\beta}_{\Gamma}( \overline{m}(\boldsymbol{\varphi}_0))), \quad\forall\,t\geq0.\label{Lv-20}
 \end{align}
 Taking a sufficiently large $R>0$, for any bounded set $\mathcal{B}_0\subset\mathfrak{X}_m$,
 we find there exists a time $t_0:=t_0(\mathcal{B}_0)>0$ such that
 \begin{align}
 	\mathcal{S}^L(t)\mathcal{B}_0\subset\mathcal{B},\qquad\forall\, t\geq t_0.\notag
 \end{align}
 This completes the proof of Lemma \ref{regular-absorbing} for $L\in(0,+\infty)$.

 Concerning the case $L=0$, we observe that only the constant $C_3$ in \eqref{Lv-18} may depend on $L\in(0,1)$.
 Under the additional assumption $(\mathbf{A5})$,
 the constant $C_3$ can be refined to be independent of $L\in(0,1)$ following the argument as in \cite[Lemma 5.5]{LvWu-4}.
 Hence, we can pass to the limit as $L\to0$ in \eqref{Lv-20}
 and obtain a similar result for the case $L=0$.
 The proof of Lemma \ref{regular-absorbing} is now complete.
\end{proof}


\noindent\textbf{Proof of Theorem \ref{global}.}
The dynamical system $(\mathfrak{X}_m,\mathcal{S}^L(t))$ is dissipative owing to Lemma \ref{dissipative}.
Moreover, the continuous dependence estimate \eqref{contiesti} implies that $\{\mathcal{S}^L(t)\}_{t\geq0}$
is a \emph{closed semigroup} on the phase space $\mathfrak{X}_m$ in the sense of \cite{PZ}.
%
%
From Lemma \ref{regular-absorbing}, we infer that $\mathcal{B}$ is a connected compact absorbing set for the dynamical system $(\mathfrak{X}_m,\mathcal{S}^L(t))$,
thus $\mathcal{B}$ is attracting as well.
Since $\mathcal{S}^L(t)\mathcal{B} \subset \mathcal{B}$ for every $t$ large enough,
the existence of the global attractor is an immediate consequence of the abstract result \cite[Corollary 6]{PZ}.
\hfill$\square$

\subsection{Stability of the global attractor for the case $L=0$}

We proceed to investigate the stability of the global attractor $\mathcal{A}_m^0$
associated with the dynamical system $(\mathfrak{X}_m,\mathcal{S}^0(t))$, with respect to perturbations $\mathcal{A}_m^L$ for small $L>0$.
Roughly speaking, this means that we need to study the asymptotic limit as $L\to0$ of the family $\{\mathcal{A}_m^L\}_{L>0}$
for the dynamical system $(\mathfrak{X}_m,\mathcal{S}^L(t))$ for $L>0$.
To provide a rigorous notion of such limit, we recall the following definition introduced in \cite[Theorem 7.2.8]{Chueshov}.

\begin{definition}\rm
\label{upper-semi}
    Let $X$ be a Banach space, and $\mathcal{M}$ be a metric space.
    Suppose that for any $\lambda\in\mathcal{M}$, $(X,\{\mathcal{S}^\lambda(t)\}_{t\geq0})$
    is a dynamical system possessing a global attractor $\mathcal{A}^\lambda\subset X$.
    Then, the family $\{\mathcal{A}^\lambda\}_{\lambda\in\mathcal{M}}$ is called upper semicontinuous
    at the point $\lambda_\ast\in\mathcal{M}$ if
    \[\lim_{\lambda\to\lambda_\ast}\text{dist}_X(\mathcal{A}^\lambda,\mathcal{A}^{\lambda_\ast})=0,\]
    where the non-symmetric Hausdorff semidistance is defined as
    \[\text{dist}_X(A,B):=\sup_{\boldsymbol{a}\in A}\inf_{\boldsymbol{b}\in B}\|\boldsymbol{a}-\boldsymbol{b}\|_X.\]
\end{definition}

To prove that the family $\{\mathcal{A}_m^L\}_{L\geq0}$ is upper semicontinuous at $L=0$,
we investigate the asymptotic behavior of a global weak solution $(\boldsymbol{\varphi}^L,\boldsymbol{\mu}^L)$ in the asymptotic limit as $L\to0$.

\begin{lemma}
    \label{asymptotic-Lto0}
    Suppose that the assumptions $(\mathbf{A1})$--$(\mathbf{A5})$ hold,
    and let $(\boldsymbol{\varphi}^L,\boldsymbol{\mu}^L)$ denote the unique weak solution to problem \eqref{model} with $L\in[0,1)$.
    Then, for any $T>1$, we have
    \begin{align}
        \boldsymbol{\varphi}^L\to \boldsymbol{\varphi}^0\quad\text{strongly in }\ C([1,T];\mathcal{L}^2)\ \ \text{as }\ L\to0.
        \label{improved-convergence}
    \end{align}
    Moreover, there exist constants $\widetilde{A}$, $\widetilde{B}>0$
    depending increasingly on $\|\boldsymbol{\varphi}_0\|_{\mathcal{L}^2}$ but not on $L$ such that for all $L\in[0,1)$,
    \begin{align}
        \|\boldsymbol{\varphi}^L-\boldsymbol{\varphi}^0\|_{C([1,T];\mathcal{L}^2)}\leq \widetilde{A} e^{\widetilde{B}T}L^{\frac{1}{4}}.\label{convergence-rate}
    \end{align}
\end{lemma}
\begin{proof}
    First of all, according to \cite[Theorem 2.4]{LvWu-4}, we see that
    \begin{align}
        \|\boldsymbol{\varphi}^L-\boldsymbol{\varphi}^0\|_{L^\infty(0,T;\mathcal{V}_{(0)}^{-1})}+\|\boldsymbol{\varphi}^L-\boldsymbol{\varphi}^0\|_{L^2(0,T;\mathcal{L}^2)}\leq C(T)\sqrt{L},\quad\text{as }L\to0,\label{convergence-1}
    \end{align}
    where the positive constant $C(T)$ depends increasingly on $\|\boldsymbol{\varphi}_0\|_{\mathcal{L}^2}$ and $T$, but not on $L>0$.
    Exploiting the proof of \cite[Theorem 2.4]{LvWu-4} carefully,
    we find that the constant $C(T)$ depends (at most) exponentially on $T$.
    This is a consequence of the application of Gronwall's lemma.
    Therefore, we can find constants $A$, $B>0$ independent of $L$ such that $C(T)=A e^{BT}$.
    Moreover, according to \cite[Lemma 5.5]{LvWu-4}, for any $T>1$,
    \begin{align*}
        \partial_t\boldsymbol{\varphi}^L \ \text{ is bounded in }\ L^2(1,T;\mathcal{L}^2)\ \text{ and }\ \boldsymbol{\varphi}^L\ \text{ is bounded in }\ L^\infty(1,T;\mathcal{H}^1).
    \end{align*}
    By the Aubin--Lions--Simon lemma, we deduce the following convergence (up to a subsequence)
    \[\boldsymbol{\varphi}^L\to \boldsymbol{\varphi}^0\quad\text{strongly in }C([1,T];\mathcal{L}^2)\ \text{as}\ L\to0.\]
    Hence, we obtain \eqref{improved-convergence} and the first estimate in \eqref{convergence-1} can be improved as
      \begin{align}
        \|\boldsymbol{\varphi}^L-\boldsymbol{\varphi}^0\|_{C([1,T];\mathcal{V}_{(0)}^{-1})}\leq A e^{BT}\sqrt{L},\quad\text{as }L\to0.\label{convergence-2}
    \end{align}
   Since $\|\boldsymbol{\varphi}^L\|_{L^\infty(1,T;\mathcal{H}^1)}$ is uniformly bounded with respect to $L\in[0,1)$,
   we infer from the interpolation inequality and \eqref{convergence-2} that
    \[\|\boldsymbol{\varphi}^L-\boldsymbol{\varphi}^0\|_{C([1,T];\mathcal{L}^2)}\leq C\|\boldsymbol{\varphi}^L-\boldsymbol{\varphi}^0\|_{C([1,T];\mathcal{V}_{(0)}^{-1})}^{\frac{1}{2}}\|\boldsymbol{\varphi}^L-\boldsymbol{\varphi}^0\|_{L^\infty(1,T;\mathcal{H}^1)}^{\frac{1}{2}}\leq \widetilde{A} e^{\widetilde{B}T}L^{\frac{1}{4}}.\]
    As a result, we obtain \eqref{convergence-rate} and complete the proof of Lemma \ref{asymptotic-Lto0}.
\end{proof}

We are in a position to establish the following stability result, which is the main result in this subsection. Here, the term \emph{stability} should be understood in the sense of semicontinuity for the family of perturbed global attractors.

\begin{proposition}
    \label{stability}
    Suppose that the assumptions $(\mathbf{A1})$--$(\mathbf{A5})$ hold.
    Then, the family of global attractors $\{\mathcal{A}_m^L\}_{L\geq0}$ is upper semicontinuous at $L=0$ in the sense of Definition \ref{upper-semi}.
\end{proposition}

To prove Proposition \ref{stability}, we will exploit the following abstract result, see \cite[Theorem 7.2.8]{Chueshov}.

\begin{lemma}
    \label{criterion}
    Let $X$ be a Banach space, and let $\mathcal{M}$ be a metric space.
    Suppose that for any $\lambda\in\mathcal{M}$, $(X,\{\mathcal{S}^\lambda(t)\}_{t\geq0})$
    is a dynamical system possessing a global attractor $\mathcal{A}^\lambda\subset X$.
    We further assume that the following conditions hold:
    \begin{itemize}
        \item [(i)] There exists a compact set $K\subset X$ such that $\mathcal{A}^\lambda \subset K$ for all $\lambda\in\mathcal{M}$.
        \item [(ii)] If $\{x_k\}_{k\in\mathbb{N}}\subset X$ and $\{\lambda_k\}_{k\in\mathbb{N}}\subset \mathcal{M}$ are sequences satisfying
        \begin{itemize}
            \item $x_k\in \mathcal{A}^{\lambda_k}$ for all $k\in\mathbb{N}$;
            \item $x_k\to x_\ast$ in $X$ as $k\to+\infty$;
            \item $\lambda_k\to\lambda_\ast$ in $\mathcal{M}$ as $k\to+\infty$;
        \end{itemize}
        then there exists $t_\ast>0$ such that $\mathcal{S}^{\lambda_k}(t)x_k\to \mathcal{S}^{\lambda_\ast}(t)x_\ast$ in $X$ for all $t>t_\ast$.
    \end{itemize}
    Then, the family $\{\mathcal{A}^\lambda\}_{\lambda\in\mathcal{M}}$ is upper semicontinuous at the point $\lambda_\ast$.
\end{lemma}

\noindent\textbf{Proof of Proposition \ref{stability}.}
To apply Lemma \ref{criterion},
we verify the conditions $(i)$ and $(ii)$ imposed therein.
\medskip

\emph{Step 1.} To verify the condition $(i)$, we show that there exists a compact set $\mathcal{K}_m\subset \mathcal{L}^2$,
independent of $L$ such that $\mathcal{A}_m^L\subset \mathcal{K}_m$ for all $L\in[0,1)$.
Indeed, Lemma \ref{regular-absorbing} implies that $\mathcal{K}_m$ can be chosen as $B_{\mathcal{H}^1}(\boldsymbol{0},R)$
such that $\mathcal{A}_m^L\subset\mathcal{K}_m$ for all $L\in[0,1)$.
\medskip

\emph{Step 2.} To verify the condition $(ii)$, let $\{L_k\}_{k\in\mathbb{N}}\subset [0,1)$ be any sequence with $L_k\to0$ as $k\to+\infty$,
and let $\{\boldsymbol{\varphi}^{L_k}\}_{k\in\mathbb{N}}\subset \mathcal{L}^2$ be any sequence with $\boldsymbol{\varphi}^{L_k}\in \mathcal{A}_m^{L_k}$ and $\boldsymbol{\varphi}^{L_k}\to\boldsymbol{\varphi}_\ast$ in $\mathcal{L}^2$ as $k\to+\infty$.
Let $t\geq t_\ast:=1$ and $\epsilon>0$ be arbitrary. Using the continuous dependence estimate \eqref{contiesti}, we can deduce that
\begin{align}
    &\|{\mathcal{S}^{L_k}(t)}\boldsymbol{\varphi}^{L_k}-\mathcal{S}^0(t)\boldsymbol{\varphi}_\ast\|_{\mathcal{L}^2}
    \notag\\[1mm]
    &\quad\leq \|\mathcal{S}^{L_k}(t)\boldsymbol{\varphi}^{L_k}-\mathcal{S}^0(t)\boldsymbol{\varphi}^{L_k}\|_{\mathcal{L}^2}
    +\|\mathcal{S}^{0}(t)\boldsymbol{\varphi}^{L_k}-\mathcal{S}^0(t)\boldsymbol{\varphi}_\ast\|_{\mathcal{L}^2}
    \notag\\
    &\quad\leq \sup_{\boldsymbol{\varphi}\in \mathcal{K}_m}\|\mathcal{S}^{L_k}(t)\boldsymbol{\varphi}-\mathcal{S}^0(t)\boldsymbol{\varphi}\|_{C([1,t];\mathcal{L}^2)}
    \notag\\
    &\qquad
    +C\|\mathcal{S}^{0}(t)\boldsymbol{\varphi}^{L_k}-\mathcal{S}^0(t)\boldsymbol{\varphi}_\ast\|_{L,0,\ast}^{\frac{1}{2}}\|\mathcal{S}^{0}(t)\boldsymbol{\varphi}^{L_k}-\mathcal{S}^0(t)\boldsymbol{\varphi}_\ast\|_{\mathcal{H}^1}^{\frac{1}{2}}
    \notag\\
    &\quad\leq  \sup_{\boldsymbol{\varphi}\in \mathcal{K}_m}\|\mathcal{S}^{L_k}(t)\boldsymbol{\varphi}-\mathcal{S}^0(t)\boldsymbol{\varphi}\|_{C([1,t];\mathcal{L}^2)}+C\|\mathcal{S}^{0}(t)\boldsymbol{\varphi}^{L_k}-\mathcal{S}^0(t)\boldsymbol{\varphi}_\ast\|_{L,0,\ast}^{\frac{1}{2}}\notag\\
     &\quad\leq  \sup_{\boldsymbol{\varphi}\in \mathcal{K}_m}\|\mathcal{S}^{L_k}(t)\boldsymbol{\varphi}-\mathcal{S}^0(t)\boldsymbol{\varphi}\|_{C([1,t];\mathcal{L}^2)}+C\|\boldsymbol{\varphi}^{L_k}-\boldsymbol{\varphi}_\ast\|_{L,0,\ast}^{\frac{1}{2}}.\label{difference-es1}
\end{align}
Since $L_k\to0$ as $k\to+\infty$ and $\mathcal{K}_m$ is bounded,
Lemma \ref{asymptotic-Lto0} implies the existence of a number $N_1\in\mathbb{N}$ such that for all $k\geq N_1$,
the first summand in \eqref{difference-es1} is smaller than $\epsilon/2$.
Furthermore, the convergence $\boldsymbol{\varphi}^{L_k}\to\boldsymbol{\varphi}_\ast$ in $\mathcal{L}^2$ implies that
\[\|\boldsymbol{\varphi}^{L_k}-\boldsymbol{\varphi}_\ast\|_{L,0,\ast}\to0\quad\text{as }k\to+\infty.\]
Hence, there exists a number $N_2\in\mathbb{N}$ such that for all $k\geq N_2$, the second summand in \eqref{difference-es1} is smaller than $\epsilon/2$.
In summary, we get
\[\|\mathcal{S}^{L_k}(t)\boldsymbol{\varphi}^{L_k}-\mathcal{S}^0(t)\boldsymbol{\varphi}_\ast\|_{\mathcal{L}^2}\leq \epsilon \quad\text{for all }k\geq N:=\max\{N_1,N_2\}.\]
Since $\epsilon>0$ is arbitrary, this verifies the condition $(ii)$ in Lemma \ref{criterion}.

Consequently, we can apply Lemma \ref{criterion} on the family $\{\mathcal{A}_m^L\}_{L\geq0}$ to prove Proposition \ref{stability}.
\hfill$\square$

 \section{Existence of Exponential Attractors for $L\in(0,+\infty)$}
 \setcounter{equation}{0}
 In this section, we establish the existence of an exponential attractor for the case $L\in(0,+\infty)$.
 For simplicity, we use $\mathcal{S}$ and $\mathcal{E}$, instead of $\mathcal{S}^L$ and $\mathcal{E}^L_m$,
 to represent the semigroup acting on the phase space $\mathfrak{X}_m$
 and the exponential attractor associated with the dynamical system $(\mathfrak{X}_m,\mathcal{S}(t))$, respectively.

To begin with, we show the uniform $1/2$-H\"older continuity of the mapping $t\mapsto\mathcal{S}(t)\boldsymbol{\varphi}_0$ in the $\mathcal{H}_{L,0}^{-1}$-norm.
 \begin{lemma}
 \label{Holder-Lip}
 Let the assumptions of Theorem \ref{exponential} be satisfied,
 and $\boldsymbol{\varphi}(t)=\mathcal{S}(t)\boldsymbol{\varphi}_0$ with $\boldsymbol{\varphi}_0\in\mathfrak{X}_m$.
 Then, it holds
 \[\|\boldsymbol{\varphi}(t_1)-\boldsymbol{\varphi}(t_2)\|_{L,0,\ast}\leq M_2|t_1-t_2|^\frac{1}{2},\quad\forall\, t_1,t_2\geq 0,\]
 where the constant $M_2>0$ is independent of the initial data, $t_1$ and $t_2$.
  \end{lemma}
\begin{proof}
According to the definition of the operator $\mathfrak{S}^L$ and the energy equality \eqref{energyeq}, we see that
\begin{align*}
    \int_{t_1}^{t_2}\|\partial_t\boldsymbol{\varphi}(s)\|_{L,0,\ast}^2\,\mathrm{d}s = \int_{t_1}^{t_2}\|\mathfrak{S}^L(\partial_t\boldsymbol{\varphi}(s))\|_{\mathcal{H}_{L,0}^1}^2\,\mathrm{d}s =\int_{t_1}^{t_2}\|\mathbf{P}\boldsymbol{\mu}(s)\|_{\mathcal{H}_{L,0}^1}^2\,\mathrm{d}s\leq M_2^2,
\end{align*}
where the constant $M_2>0$ is independent of the initial data, $t_1$ and $t_2$.
Then, we can conclude that
	\begin{align*}
	\|\boldsymbol{\varphi}(t_1)-\boldsymbol{\varphi}(t_2)\|_{L,0,\ast}\leq|t_1-t_2|^\frac{1}{2}\Big(\int_{t_1}^{t_2}\|\partial_t\boldsymbol{\varphi}(s)\|_{L,0,\ast}^2\,\mathrm{d}s\Big)^{\frac{1}{2}}\leq M_2|t_1-t_2|^\frac{1}{2}.
	\end{align*}
    This completes the proof of Lemma \ref{Holder-Lip}.
\end{proof}
%
The following result shows that the semigroup is strongly continuous with respect to the $(\mathcal{H}^1)'$-metric.
\begin{lemma}
	\label{strong-continuous}
	Let $\boldsymbol{\varphi}_i$ $(i=1,2)$ be two solutions to problem \eqref{model}
    subject to the initial data $\boldsymbol{\varphi}_{0,i}\in \mathcal{S}(1)\mathfrak{X}_m$.
    Then, the following estimate holds:
	\begin{align}
	&\|\boldsymbol{\varphi}_1(t)-\boldsymbol{\varphi}_2(t)\|_{(\mathcal{H}^1)'}^2+C_\ast\int_0^t \|\boldsymbol{\varphi}_1(s)-\boldsymbol{\varphi}_2(s)\|_{\mathcal{L}^2}^2\,\mathrm{d}s\leq M_3 e^{\kappa t}\|\boldsymbol{\varphi}_{0,1}-\boldsymbol{\varphi}_{0,2}\|_{(\mathcal{H}^1)'}^2,\quad\forall\,t\geq0,\label{dual-continuous}
	\end{align}
	for some positive constants $\kappa$ and $M_3$, which are independent of $\boldsymbol{\varphi}_{0,i}$.
\end{lemma}
\begin{proof}
Since the initial data $\boldsymbol{\varphi}_{0,i}\in\mathcal{S}(1)\mathfrak{X}_m$,
according to \cite[Theorem 2.6]{LvWu-4}, there exists a constant $\delta^\sharp\in(0,1)$, such that
    \begin{align}
 \|\boldsymbol{\varphi}_i(t)\|_{\mathcal{L}^\infty} \leq 1-\delta^\sharp,\quad i=1,2,\quad\text{for a.a. }t\geq 0.
 \label{uniform-separation}
    \end{align}
    Examining the proof of \cite[Theorem 2.6]{LvWu-4}, we find that the constant $\delta^\sharp$
    depends only on the initial free energies $E(\boldsymbol{\varphi}_{0,i})$
    and the initial mean values $\overline{m}(\boldsymbol{\varphi}_{0,i})$,
    but is independent of $\boldsymbol{\varphi}_{0,i}$.
    Let us denote the difference between the two solutions by
	$$\boldsymbol{\varphi}^\sharp=\boldsymbol{\varphi}_1-\boldsymbol{\varphi}_2,\quad\boldsymbol{\varphi}_0^\sharp=\boldsymbol{\varphi}_{0,1}-\boldsymbol{\varphi}_{0,2}.$$
	Then, $\boldsymbol{\varphi}^\sharp$ satisfies
\begin{align}
    \langle\partial_t\boldsymbol{\varphi}^\sharp,\boldsymbol{z}\rangle_{(\mathcal{H}^1)',\mathcal{H}^1}=-\int_\Omega\nabla\mu^\sharp\cdot\nabla z\,\mathrm{d}x-\int_\Gamma\nabla_\Gamma\theta^\sharp\cdot\nabla_\Gamma z_\Gamma\,\mathrm{d}S-\frac{1}{L}\int_\Gamma (\theta^\sharp-\mu^\sharp)(z_\Gamma-z)\,\mathrm{d}S\label{regular-weak}
\end{align}
for all $\boldsymbol{z}\in\mathcal{H}^1$, with
\begin{align*}
       & \mu^\sharp=a_{\Omega}\varphi^\sharp-J\ast\varphi^\sharp+\beta(\varphi_1)-\beta(\varphi_2)+\pi(\varphi_1)-\pi(\varphi_2),&&\text{in }Q,\\
    	&	\theta^\sharp=a_{\Gamma}\psi^\sharp-K\circledast\psi^\sharp+\beta_{\Gamma}(\psi_1)-\beta_{\Gamma}(\psi_2)+\pi_{\Gamma}(\psi_1)-\pi_{\Gamma}(\psi_2),&&\text{on }\Sigma.
\end{align*}
 Since $\overline{m}({\boldsymbol{\varphi}}^\sharp-\overline{m}(\boldsymbol{\varphi}^\sharp)\boldsymbol{1})=0$,
 we can take the test function $\boldsymbol{z}=\mathfrak{S}^L({\boldsymbol{\varphi}}^\sharp-\overline{m}(\boldsymbol{\varphi}^\sharp)\boldsymbol{1})$ in \eqref{regular-weak}, and then obtain
	\begin{align}
		0&=\frac{1}{2}\frac{\mathrm{d}}{\mathrm{d}t}\|{\boldsymbol{\varphi}}^\sharp-\overline{m}(\boldsymbol{\varphi}^\sharp)\boldsymbol{1}\|_{L,0,\ast}^2+(\boldsymbol{\mu}^\sharp,\boldsymbol{\varphi}^\sharp-\overline{m}(\boldsymbol{\varphi}^\sharp)\boldsymbol{1})_{\mathcal{L}^2}\notag\\
		&=\frac{1}{2}\frac{\mathrm{d}}{\mathrm{d}t}\|\boldsymbol{\varphi}^\sharp-\overline{m}(\boldsymbol{\varphi}^\sharp)\boldsymbol{1}\|_{L,0,\ast}^2-\overline{m}(\boldsymbol{\mu}^\sharp)\overline{m}(\boldsymbol{\varphi}^\sharp)\notag\\
        &\quad+\int_\Omega (a_{\Omega}\varphi^\sharp-J\ast\varphi^\sharp+\beta(\varphi_1)-\beta(\varphi_2)+\pi(\varphi_1)-\pi(\varphi_2))\varphi^\sharp\,\mathrm{d}x\notag\\
		&\quad+\int_\Gamma(a_{\Gamma}\psi^\sharp-K\circledast\psi^\sharp+\beta_{\Gamma}(\psi_1)-\beta_{\Gamma}(\psi_2)+\pi_{\Gamma}(\psi_1)-\pi_{\Gamma}(\psi_2))\psi^\sharp\,\mathrm{d}S\notag\\
		&\geq\frac{1}{2}\frac{\mathrm{d}}{\mathrm{d}t}\|\boldsymbol{\varphi}^\sharp-\overline{m}(\boldsymbol{\varphi}^\sharp)\boldsymbol{1}\|_{L,0,\ast}^2+(a_\ast+\alpha-\gamma_{1})\|\varphi^\sharp\|_{H}^2+(a_\circledast+\alpha-\gamma_{2})\|\psi^\sharp\|_{H_\Gamma}^2\notag\\
		&\quad-\langle\boldsymbol{\varphi}^\sharp-\overline{m}(\boldsymbol{\varphi}^\sharp)\boldsymbol{1},(J\ast\varphi^\sharp,K\circledast\psi^\sharp)\rangle_{(\mathcal{H}^1)',\mathcal{H}^1}-|\overline{m}(\boldsymbol{\mu}^\sharp)||\overline{m}(\boldsymbol{\varphi}^\sharp)|\notag\\
        &\quad-|\overline{m}(\boldsymbol{\varphi}^\sharp)|\Big|\int_\Omega J\ast\varphi^\sharp\,\mathrm{d}x+\int_\Gamma K\circledast\psi^\sharp\,\mathrm{d}S\Big|\notag\\
		&\geq \frac{1}{2}\frac{\mathrm{d}}{\mathrm{d}t}\|\boldsymbol{\varphi}^\sharp-\overline{m}(\boldsymbol{\varphi}^\sharp)\boldsymbol{1}\|_{L,0,\ast}^2+C_\ast\|\boldsymbol{\varphi}^\sharp\|_{\mathcal{L}^2}^2-C\|\boldsymbol{\varphi}^\sharp-\overline{m}(\boldsymbol{\varphi}^\sharp)\boldsymbol{1}\|_{(\mathcal{H}^1)'}\|(J\ast\varphi^\sharp,K\circledast\psi^\sharp)\|_{\mathcal{H}^1}\notag\\
        &\quad-|\overline{m}(\boldsymbol{\mu}^\sharp)||\overline{m}(\boldsymbol{\varphi}^\sharp)|-C|\overline{m}(\boldsymbol{\varphi}^\sharp)|^2-\frac{C_\ast}{6}\|\boldsymbol{\varphi}^\sharp\|_{\mathcal{L}^2}^2,\label{gronwall'}
	\end{align}
     where the constant $C_\ast$ is given by
\begin{align}
    C_\ast:=\min\{a_\ast+\alpha-\gamma_1,a_\circledast+\alpha-\gamma_2\}>0.\label{C_ast}
\end{align}
The strict separation property \eqref{uniform-separation} indicates that $|\overline{m}(\boldsymbol{\mu}^\sharp)|\leq C\|\boldsymbol{\varphi}^\sharp\|_{\mathcal{L}^2}$.
Then, by H\"older's inequality, we get
\begin{align*}
    &\|\boldsymbol{\varphi}^\sharp-\overline{m}(\boldsymbol{\varphi}^\sharp)\boldsymbol{1}\|_{(\mathcal{H}^1)'}\|(J\ast\varphi^\sharp,K\circledast\psi^\sharp)\|_{\mathcal{H}^1}\leq \frac{C_\ast}{6}\|\boldsymbol{\varphi}^\sharp\|_{\mathcal{L}^2}^2+C\|\boldsymbol{\varphi}^\sharp-\overline{m}(\boldsymbol{\varphi}^\sharp)\boldsymbol{1}\|_{L,0,\ast}^2,\\
    &|\overline{m}(\boldsymbol{\mu}^\sharp)||\overline{m}(\boldsymbol{\varphi}^\sharp)|\leq C\|\boldsymbol{\varphi}^\sharp\|_{\mathcal{L}^2}|\overline{m}(\boldsymbol{\varphi}^\sharp)|\leq  \frac{C_\ast}{6}\|\boldsymbol{\varphi}^\sharp\|_{\mathcal{L}^2}^2+C|\overline{m}(\boldsymbol{\varphi}^\sharp)|^2,
\end{align*}
which, together with \eqref{gronwall'}, yields that
	\begin{align}
		&\frac{\mathrm{d}}{\mathrm{d}t}\|\boldsymbol{\varphi}^\sharp-\overline{m}(\boldsymbol{\varphi}^\sharp)\boldsymbol{1}\|_{L,0,\ast}^2+C_\ast\|\boldsymbol{\varphi}^\sharp\|_{\mathcal{L}^2}^2\leq C\|\boldsymbol{\varphi}^\sharp-\overline{m}(\boldsymbol{\varphi}^\sharp)\boldsymbol{1}\|_{L,0,\ast}^2+C|\overline{m}(\boldsymbol{\varphi}_0^\sharp)|^2.\label{gronwall}
	\end{align}
Here, we have used the property of mass conservation $\overline{m}(\boldsymbol{\varphi}^\sharp(t))=\overline{m}(\boldsymbol{\varphi}^\sharp_0)$ for all $t\geq0$.
Applying Gronwall's lemma to \eqref{gronwall}, we find
    \begin{align}
       &\|\boldsymbol{\varphi}^\sharp(t)-\overline{m}(\boldsymbol{\varphi}^\sharp(t))\boldsymbol{1}\|_{L,0,\ast}^2 +C_\ast\int_0^t\|\boldsymbol{\varphi}^\sharp(s)\|_{\mathcal{L}^2}^2\,\mathrm{d}s\notag\\
       &\quad\leq e^{Ct}\|\boldsymbol{\varphi}^\sharp_0-\overline{m}(\boldsymbol{\varphi}^\sharp_0)\boldsymbol{1}\|_{L,0,\ast}^2+Ce^{Ct}|\overline{m}(\boldsymbol{\varphi}^\sharp_0)|^2
       \notag\\
       &\quad \leq Ce^{Ct}\|\boldsymbol{\varphi}^\sharp_0\|_{(\mathcal{H}^1)'}^2,\notag
    \end{align}
  which implies that
   \begin{align}
       &\|\boldsymbol{\varphi}^\sharp(t)-\overline{m}(\boldsymbol{\varphi}^\sharp(t))\boldsymbol{1}\|_{L,0,\ast}^2+|\overline{m}(\boldsymbol{\varphi}^\sharp(t))|^2 +C_\ast\int_0^t\|\boldsymbol{\varphi}^\sharp(s)\|_{\mathcal{L}^2}^2\,\mathrm{d}s\notag\\
       &\quad\leq Ce^{Ct}\|\boldsymbol{\varphi}^\sharp_0\|_{(\mathcal{H}^1)'}^2+|\overline{m}(\boldsymbol{\varphi}^\sharp(t))|^2\notag\\
       &\quad=Ce^{Ct}\|\boldsymbol{\varphi}^\sharp_0\|_{(\mathcal{H}^1)'}^2+|\overline{m}(\boldsymbol{\varphi}^\sharp_0)|^2\notag\\
       &\quad\leq Ce^{Ct}\|\boldsymbol{\varphi}^\sharp_0\|_{(\mathcal{H}^1)'}^2.\notag
    \end{align}
 According to the equivalence of the norms $(\|\boldsymbol{z}-\overline{m}(\boldsymbol{z})\boldsymbol{1}\|_{L,0,\ast}^2+|\overline{m}(\boldsymbol{z})|^2)^{\frac{1}{2}}$ and $\|\boldsymbol{z}\|_{(\mathcal{H}_{L}^1)'}$ on $(\mathcal{H}_L^1)'$,
 we can conclude \eqref{dual-continuous}. This completes the proof of Lemma \ref{strong-continuous}.
\end{proof}
The following two lemmas are crucial to establish the existence of an exponential attractor.
The first result addresses that the semigroup $\mathcal{S}(t)$ is some kind of contraction map,
up to the term $\|\boldsymbol{\varphi}_1-\boldsymbol{\varphi}_2\|_{L^2(0,t;(\mathcal{H}^1)')}$.
\begin{lemma}
	\label{contraction}
	Let the assumptions of Lemma \ref{strong-continuous} hold. Then, for all $t>0$, we have
	\begin{align}
	\|\boldsymbol{\varphi}_1(t)-\boldsymbol{\varphi}_2(t)\|_{(\mathcal{H}^1)'}^2&\leq e^{-C_\ast t}\|\boldsymbol{\varphi}_{0,1}-\boldsymbol{\varphi}_{0,2}\|_{(\mathcal{H}^1)'}^2+M_{4}\int_0^t\|\boldsymbol{\varphi}_1(s)-\boldsymbol{\varphi}_2(s)\|_{(\mathcal{H}^1)'}^2\,\mathrm{d}s,\label{contract}
	\end{align}
	for some positive constant $M_4$ that is independent of the initial data.
\end{lemma}
\begin{proof}
It is easy to check that
\begin{align*}
    \|\boldsymbol{\varphi}^\sharp\|_{\mathcal{L}^2}^2&=\|\boldsymbol{\varphi}^\sharp-\overline{m}(\boldsymbol{\varphi}^\sharp)\boldsymbol{1}+\overline{m}(\boldsymbol{\varphi}^\sharp)\boldsymbol{1}\|_{\mathcal{L}^2}^2\\
    &=\|\boldsymbol{\varphi}^\sharp-\overline{m}(\boldsymbol{\varphi}^\sharp)\boldsymbol{1}\|_{\mathcal{L}^2}^2+|\overline{m}(\boldsymbol{\varphi}^\sharp)|^2(|\Omega|+|\Gamma|)\\
    &\geq \|\boldsymbol{\varphi}^\sharp-\overline{m}(\boldsymbol{\varphi}^\sharp)\boldsymbol{1}\|_{L,0,\ast}^2+|\overline{m}(\boldsymbol{\varphi}^\sharp)|^2(|\Omega|+|\Gamma|),
\end{align*}
which, combined with \eqref{gronwall}, indicates that
\begin{align}
		&\frac{\mathrm{d}}{\mathrm{d}t}\Big(\|\boldsymbol{\varphi}^\sharp-\overline{m}(\boldsymbol{\varphi}^\sharp)\boldsymbol{1}\|_{L,0,\ast}^2+|\overline{m}(\boldsymbol{\varphi}^\sharp)|^2\Big)+C_\ast\Big(\|\boldsymbol{\varphi}^\sharp-\overline{m}(\boldsymbol{\varphi}^\sharp)\boldsymbol{1}\|_{L,0,\ast}^2+|\overline{m}(\boldsymbol{\varphi}^\sharp)|^2\Big)\notag\\
        &\quad\leq C\|\boldsymbol{\varphi}^\sharp-\overline{m}(\boldsymbol{\varphi}^\sharp)\boldsymbol{1}\|_{L,0,\ast}^2+C|\overline{m}(\boldsymbol{\varphi}^\sharp)|^2
        \notag\\
        &\quad \leq C\|\boldsymbol{\varphi}^\sharp\|_{(\mathcal{H}^1)'}^2.\label{gronwall-1}
	\end{align}
Then, applying Gronwall's inequality to \eqref{gronwall-1},
together with the equivalence of the norms $(\|\boldsymbol{z}-\overline{m}(\boldsymbol{z})\boldsymbol{1}\|_{L,0,\ast}^2+|\overline{m}(\boldsymbol{z})|^2)^{\frac{1}{2}}$ and $\|\boldsymbol{z}\|_{(\mathcal{H}_{L}^1)'}$ on $(\mathcal{H}_L^1)'$,
we can conclude \eqref{contract}.
\end{proof}

The following lemma indicates some compactness for the term $\|\boldsymbol{\varphi}_1-\boldsymbol{\varphi}_2\|_{L^2(0,t;(\mathcal{H}^1)')}$ on the right-hand side of \eqref{contract}.
\begin{lemma}
    \label{compactness}
    Let the assumptions of Lemma \ref{strong-continuous} hold. Then, for all $t>0$, the following estimate hold:
    \begin{align}
        &\|\partial_t\boldsymbol{\varphi}_1-\partial_t\boldsymbol{\varphi}_2\|_{L^2(0,t;(\mathcal{W}_{L,\mathbf{n}}^2)')}+C_\ast\int_0^t \|\boldsymbol
        \varphi_1(s)-\boldsymbol{\varphi}_2(s)\|_{\mathcal{L}^2}^2\,\mathrm{d}s\leq M_5 e^{\kappa t}\|\boldsymbol{\varphi}_{0,1}-\boldsymbol{\varphi}_{0,2}\|_{(\mathcal{H}^1)'}^2,\label{lower-esti}
    \end{align}
    for some positive constant $M_5$ and $\kappa$ that are independent of the initial data.
\end{lemma}
\begin{proof}
Taking $\boldsymbol{z}\in \mathcal{W}_{L,\mathbf{n}}^2\subset\mathcal{H}^1$ in \eqref{regular-weak}, we find
\begin{align}
   & \langle\partial_t\boldsymbol{\varphi}^\sharp,\boldsymbol{z}\rangle_{(\mathcal{W}_{L,\mathbf{n}}^2)',\mathcal{W}_{L,\mathbf{n}}^2}
   \notag\\
   &\quad =-\int_\Omega\nabla\mu^\sharp\cdot\nabla z\,\mathrm{d}x-\int_\Gamma\nabla_\Gamma\theta^\sharp\cdot\nabla_\Gamma z_\Gamma\,\mathrm{d}S-\frac{1}{L}\int_\Gamma (\theta^\sharp-\mu^\sharp)(z_\Gamma-z)\,\mathrm{d}S
   \notag\\
   &\quad =\int_\Omega\mu^\sharp \Delta z\,\mathrm{d}x+\int_\Gamma \theta^\sharp\Delta_\Gamma z_\Gamma\,\mathrm{d}S-\int_\Gamma\partial_\mathbf{n}z \mu^\sharp\,\mathrm{d}S-\frac{1}{L}\int_\Gamma (\theta^\sharp-\mu^\sharp)(z_\Gamma-z)\,\mathrm{d}S
   \notag\\
   &\quad =\int_\Omega\mu^\sharp \Delta z\,\mathrm{d}x+\int_\Gamma \theta^\sharp\Delta_\Gamma z_\Gamma\,\mathrm{d}S-\frac{1}{L}\int_\Gamma(z_\Gamma-z)\mu^\sharp\,\mathrm{d}S
   \notag\\
   &\qquad
  -\frac{1}{L}\int_\Gamma (\theta^\sharp-\mu^\sharp)(z_\Gamma-z)\,\mathrm{d}S
   \notag\\
   &\quad =\int_\Omega\mu^\sharp \Delta z\,\mathrm{d}x+\int_\Gamma \theta^\sharp\Delta_\Gamma z_\Gamma\,\mathrm{d}S-\frac{1}{L}\int_\Gamma \theta^\sharp(z_\Gamma-z)\,\mathrm{d}S
   \notag\\
   &\quad \leq C\big(\|\mu^\sharp\|_H+\|\theta^\sharp\|_{H_\Gamma}\big)\|\boldsymbol{z}\|_{\mathcal{H}^2}.
   \label{lower-1}
   \end{align}
   By the strict separation property \eqref{uniform-separation}, we can deduce that
   \begin{align}
     \|\mu^\sharp\|_H+\|\theta^\sharp\|_{H_\Gamma}
     &\leq C\|\boldsymbol{\varphi}^\sharp\|_{\mathcal{L}^2}+\|\beta(\varphi_1)-\beta(\varphi_2)\|_H+\|\beta_\Gamma(\psi_1)-\beta_\Gamma(\psi_2)\|_{H_\Gamma}
     \notag\\
    &\leq C\|\boldsymbol{\varphi}^\sharp\|_{\mathcal{L}^2}.\label{lower-2}
   \end{align}
   Combining \eqref{dual-continuous}, \eqref{lower-1}, \eqref{lower-2}
   and the definition of the dual norm $\|\partial_t\boldsymbol{\varphi}^\sharp\|_{(\mathcal{W}_{L,\mathbf{n}}^2)'}$,
   we arrive at the conclusion \eqref{lower-esti}. This completes the proof of Lemma \ref{compactness}.
\end{proof}

\noindent\textbf{Proof of Theorem \ref{exponential}.}
	In order to apply Lemma \ref{abstract},
    it suffices to verify the existence of an exponential attractor for the restriction of $\mathcal{S}(t)$
    on some properly chosen semi-invariant absorbing set in $\mathfrak{X}_m$.
    Thanks to Lemma \ref{regular-absorbing},
    the ball $\mathcal{B}=B_{\mathcal{H}^1}(\boldsymbol{0},R)\cap \mathfrak{X}_m$ is absorbing for $\mathcal{S}(t)$,
    provided that $R>0$ is sufficiently large.
    Since we want this ball to be semi-invariant with respect to the semigroup,
    we push it forward by the semigroup,
    by defining first the set $\mathcal{B}_1=[\cup_{t\geq0}\mathcal{S}(t)\mathcal{B}]_{\mathcal{L}^2}\cap\mathfrak{X}_m$,
    where $[\cdot]_{\mathcal{L}^2}$ denotes the closure in the space $\mathcal{L}^2$,
    and then the set $\mathbb{B}:=\mathcal{S}(1)\mathcal{B}_1$.
    Thus, $\mathbb{B}$ is a semi-invariant compact subset of the phase space $\mathfrak{X}_m$.
    On the other hand, we infer from Lemma \ref{regular-absorbing} that
	\begin{align*}
	\sup_{t\geq0}\Big(\|\boldsymbol{\varphi}(t)\|_{\mathcal{H}^1}+\|\boldsymbol{\mu}(t)\|_{\mathcal{H}^1} +\|\partial_t\boldsymbol{\varphi}(t)\|_{(\mathcal{H}^1)'}\Big)\leq C_m,
	\end{align*}
	for every trajectory $\boldsymbol{\varphi}$ originating from $\boldsymbol{\varphi}_0\in\mathbb{B}$,
    for some constant $C_m>0$ that is independent of $\boldsymbol{\varphi}_0\in\mathbb{B}$.

    We can now apply the abstract result Lemma \ref{abstract} to the map $\mathbb{S}:=\mathcal{S}(T)$,
    for a fixed $T>0$ such that $e^{-C_\ast T}<1/2$,
    where the constant $C_\ast$ is the same as in Lemma \ref{contraction}.
    To this end, we introduce the spaces
	\begin{align*}
	 &\mathbb{H}:=(\mathcal{H}^1)',\quad\mathbb{V}_1:=L^2(0,T;\mathcal{L}^2)\cap H^1(0,T;(\mathcal{W}_{L,\mathbf{n}}^2)'),\quad\mathbb{V}:=L^2(0,T;(\mathcal{H}^{1})').
	\end{align*}
It is easy to check that $\mathbb{V}_1$ is compactly embedded into $\mathbb{V}$ (see Remark \ref{compact-embedding}).
Then, we introduce the operator $\mathbb{T}:\mathbb{B}\to\mathbb{V}_1$ by $\mathbb{T}\boldsymbol{\varphi}_0:=\boldsymbol{\varphi}\in\mathbb{V}_1$,
where $\boldsymbol{\varphi}$ solves \eqref{model} with $\boldsymbol{\varphi}(0)=\boldsymbol{\varphi}_0\in\mathbb{B}$.
We claim that the maps $\mathbb{S}$, $\mathbb{T}$, the spaces $\mathbb{H}$, $\mathbb{V}_1$, $\mathbb{V}$ satisfy all the assumptions of Lemma \ref{abstract}.
In fact, the global Lipschitz continuity \eqref{T-Lip} of $\mathbb{T}$ is an immediate consequence of Lemma \ref{compactness}
and the estimate \eqref{A2} follows from \eqref{contract}.
Therefore, due to Lemma \ref{abstract}, the semigroup $\mathbb{S}(n)=\mathcal{S}(nT)$ generated by the iterations of
the operator $\mathbb{S}:\mathbb{B}\to\mathbb{B}$ possesses a (discrete) exponential attractor $\mathcal{E}_d$ in $\mathbb{B}$
endowed by the topology of $(\mathcal{H}^1)'$.
In order to construct the exponential attractor $\mathcal{E}$ for the semigroup $\mathcal{S}(t)$ with continuous time,
we note that, due to Lemma \ref{strong-continuous},
the semigroup $\mathcal{S}(t)$ is Lipschitz continuous on $\mathbb{B}$ in the topology of $(\mathcal{H}^1)'$.
Hence, the desired exponential attractor $\mathcal{E}$ for the continuous semigroup $\mathcal{S}(t)$
can be obtained by the standard formula $\mathcal{E}=\bigcup_{t\in[0,T]}\mathcal{S}(t)\mathcal{E}_d$.

    In order to complete the proof,
    we need to verify that $\mathcal{E}$ defined above is the exponential attractor for $\mathcal{S}(t)$ restricted to $\mathbb{B}$
    not only with respect to the $(\mathcal{H}^1)'$-metric,
    but also in a stronger metric.
    This is an immediate consequence of the following facts:
    $\mathbb{B}$ is bounded in $\mathcal{H}^1\cap \mathcal{L}^\infty$ and the interpolation inequalities
    \begin{align}
        &\|\boldsymbol{z}\|_{\mathcal{H}^{1-\nu}}\leq C_\nu\|\boldsymbol{z}\|_{(\mathcal{H}^1)'}^ {\frac{\nu}{2}}\|\boldsymbol{z}\|_{\mathcal{H}^1}^{1-\frac{\nu}{2}},\quad\nu\in(0,1),\label{interpolation}\\
        &\|\boldsymbol{z}\|_{\mathcal{L}^q}\leq C_q\|\boldsymbol{z}\|_{(\mathcal{H}^1)'}^{\frac{1}{q}}\|\boldsymbol{z}\|_{\mathcal{H}^1}^{\frac{1}{q}}\|\boldsymbol{z}\|_{\mathcal{L}^\infty}^{1-\frac{2}{q}},\quad q\in(2,+\infty).\label{interpolation-2}
    \end{align}
    Finally, we verify that the fractal dimension is finite. To this end, let us consider the mapping
    \[\mathcal{S}^\ast:[0,T]\times\mathfrak{X}_m\to\mathfrak{X}_m,\quad(t,\boldsymbol{\varphi})\mapsto \mathcal{S}(t)\boldsymbol{\varphi}.\]
    It is obvious that $\mathcal{E}=\mathcal{S}^\ast([0,T]\times\mathcal{E}_d)$.
    By the interpolation inequality \eqref{interpolation}, Lemmas \ref{Holder-Lip} and \ref{strong-continuous}, we find
    \begin{align}
    &\|\mathcal{S}^\ast(t_1,\boldsymbol{\varphi}_1)-\mathcal{S}^\ast(t_2,\boldsymbol{\varphi}_2)\|_{\mathcal{H}^{1-\nu}}\notag\\
        &\quad=\|\mathcal{S}(t_1)\boldsymbol{\varphi}_1-\mathcal{S}(t_2)\boldsymbol{\varphi}_2\|_{\mathcal{H}^{1-\nu}}\notag\\
        &\quad\leq C_\nu\|\mathcal{S}(t_1)\boldsymbol{\varphi}_1-\mathcal{S}(t_2)\boldsymbol{\varphi}_2\|_{(\mathcal{H}^1)'}^\frac{\nu}{2} \|\mathcal{S}(t_1)\boldsymbol{\varphi}_1-\mathcal{S}(t_2)\boldsymbol{\varphi}_2\|_{\mathcal{H}^1}^{1-\frac{\nu}{2}}\notag\\
        &\quad\leq C_\nu\|\mathcal{S}(t_1)\boldsymbol{\varphi}_1-\mathcal{S}(t_1)\boldsymbol{\varphi}_2\|_{(\mathcal{H}^1)'}^\frac{\nu}{2} + C_\nu\|\mathcal{S}(t_1)\boldsymbol{\varphi}_2 -\mathcal{S}(t_2)\boldsymbol{\varphi}_2 \|_{(\mathcal{H}^1)'}^\frac{\nu}{2} \notag\\
        &\quad\leq C_\nu\Big(\|\boldsymbol{\varphi}_1-\boldsymbol{\varphi}_2\|_{\mathcal{H}^{1-\nu}}^{\frac{\nu}{2}}+|t_1-t_2|^{\frac{\nu}{4}}\Big),\qquad\forall\, t_1,t_2\in[0,T],\notag
    \end{align}
    which yields that
    \begin{align}
        \text{dim}_{\text{F},\mathcal{H}^{1-\nu}}(\mathcal{E})&=\text{dim}_{\text{F},\mathcal{H}^{1-\nu}}(\mathcal{S}^\ast([0,T]\times\mathcal{E}_d))\notag\\
        &\leq \frac{4}{\nu}\text{dim}_{\text{F},\mathbb{R}\times\mathcal{H}^{1-\nu}}([0,T]\times\mathcal{E}_d)\notag\\
        &\leq \frac{4}{\nu}\Big(1+\text{dim}_{\text{F},\mathcal{H}^{1-\nu}}(\mathcal{E}_d)\Big).\label{dim-1}
    \end{align}
    Denote $\mathcal{N}_\varepsilon(\mathcal{E}_d;\mathcal{H}^{1-\nu})$ the minimal number of balls in $\mathcal{H}^{1-\nu}$
    with radius $\varepsilon$ that are necessary to cover $\mathcal{E}_d$.
    Note that if $\|\boldsymbol{u}-\boldsymbol{v}\|_{\mathcal{H}^{1-\nu}}=\varepsilon$, by \eqref{interpolation}, it holds
    \[\|\boldsymbol{u}-\boldsymbol{v}\|_{(\mathcal{H}^1)'}\geq C_\nu^{-\frac{2}{\nu}}\|\boldsymbol{u}-\boldsymbol{v}\|_{\mathcal{H}^{1-\nu}}^{\frac{2}{\nu}}=C_\nu^{-\frac{2}{\nu}}\varepsilon^{\frac{2}{\nu}}:=r_\nu,\]
    this implies that
    \[\mathcal{N}_\varepsilon(\mathcal{E}_d;\mathcal{H}^{1-\nu})\leq \mathcal{N}_{r_\nu}(\mathcal{E}_d;(\mathcal{H}^1)').\]
    Thus, we obtain
    \begin{align}
        \text{dim}_{\text{F},\mathcal{H}^{1-\nu}}(\mathcal{E}_d)&=\limsup_{\varepsilon\to0}\frac{\mathcal{N}_\varepsilon(\mathcal{E}_d;\mathcal{H}^{1-\nu})}{-\text{ln}(\varepsilon)}\notag\\
        &\leq \limsup_{\varepsilon\to0}\frac{\mathcal{N}_{r_\nu}(\mathcal{E}_d;(\mathcal{H}^1)')}{-\text{ln}(\varepsilon)}
        \notag\\
        &\leq \limsup_{\varepsilon\to0}\frac{\mathcal{N}_{\varepsilon^{\frac{4}{\nu}}}(\mathcal{E}_d;(\mathcal{H}^1)')}{-\dfrac{\nu}{4}\,\text{ln}(\varepsilon^{\frac{4}{\nu}})}\notag\\
        &=\frac{4}{\nu}\text{dim}_{\text{F},(\mathcal{H}^1)'}(\mathcal{E}_d)<+\infty,\label{dim-2}
    \end{align}
    where we have used the fact that $r_\nu\geq \varepsilon^{\frac{4}{\nu}}$ for sufficiently small $\varepsilon>0$. Collecting \eqref{dim-1} and \eqref{dim-2}, we find
    \[\text{dim}_{\text{F},\mathcal{H}^{1-\nu}}(\mathcal{E})\leq \frac{4}{\nu}+\frac{16}{\nu^2}\text{dim}_{\text{F},(\mathcal{H}^1)'}(\mathcal{E}_d):=C_{m,\nu}<+\infty.\]
    Similarly, from \eqref{interpolation-2}, we can conclude that there exists a constant $C_{m,q}$ such that
    \[\text{dim}_{\text{F},\mathcal{L}^{q}}(\mathcal{E})\leq C_{m,q}<+\infty.\]
    This completes the proof of Theorem \ref{exponential}.
\hfill$\square$
\section{Convergence to a Single Equilibrium for $L\in(0,+\infty)$}
\setcounter{equation}{0}
Let $\boldsymbol{\varphi}$ be the unique global weak solution to problem \eqref{model}
corresponding to the initial datum $\boldsymbol{\varphi}_0\in\mathfrak{X}_m$ obtained in Proposition \ref{well-posedness}.
In this section, we aim to show that the $\omega$-limit set
\begin{align}
\omega(\boldsymbol{\varphi}_0):=\big\{\boldsymbol{\varphi}_\infty:\exists\, t_n\to+\infty\text{ such that }\boldsymbol{\varphi}(t_n)\to\boldsymbol{\varphi}_\infty\text{ in }\mathcal{L}^2\big\}\notag
\end{align}
is a singleton.

According to \cite[Theorem 2.5]{LvWu-4}, we see that $\boldsymbol{\varphi}\in L^\infty(\tau,+\infty;\mathcal{H}^1)$ for any $\tau>0$,
then $\{\boldsymbol{\varphi}(t)\}_{t\geq\tau}$ is bounded in $\mathcal{H}^1$ and thus relatively compact in $\mathcal{L}^2$.
Hence, $\omega(\boldsymbol{\varphi}_0)$ is nonempty, connected and compact in $\mathcal{L}^2$.
Moreover, the following lemma provides a useful characterization of the $\omega$-limit set $\omega(\boldsymbol{\varphi}_0)$.
\begin{lemma}
    \label{steady-separation}
   Let the assumptions $(\mathbf{A1})$--$(\mathbf{A4})$ be satisfied.
   Then, every element $\boldsymbol{\varphi}_\infty\in\omega(\boldsymbol{\varphi}_0)$
   is a strong solution to the elliptic boundary value problem \eqref{stationary}
   with the associated constant $\mu_\infty=\theta_\infty$ determined by \eqref{mu_s},
   and there exist uniform constants $M_\infty>0$, $\delta_\infty\in(0,1)$ such that
    \begin{align}
        &-1+\delta_\infty\leq \varphi_\infty\leq 1-\delta_\infty,\quad\text{a.e. in }\Omega,\label{phi_s-separation}\\
        &-1+\delta_\infty\leq \psi_\infty\leq 1-\delta_\infty,\quad\text{a.e. on }\Gamma,\label{psi_s-separation}\\
        &|\mu_\infty|\leq M_\infty,\label{mu_s-bounded}
    \end{align}
    hold for all $\boldsymbol{\varphi}_\infty\in\omega(\boldsymbol{\varphi}_0)$.
\end{lemma}
\begin{proof}
First of all, the energy equality \eqref{energyeq} indicates that the energy functional $E:\mathfrak{X}_m\to \mathbb{R}$ serves as a strict Lyapunov function for the semigroup $\mathcal{S}(t)$.
Then, every $\boldsymbol{\varphi}_\infty\in\omega(\boldsymbol{\varphi}_0)$ is a stationary point of $\{\mathcal{S}(t)\}_{t\geq0}$,
that is, $\mathcal{S}(t)\boldsymbol{\varphi}_\infty=\boldsymbol{\varphi}_\infty$ for all $t\geq0$. Thanks to the regularity of global weak solutions of the evolution problem \eqref{model} (cf. \cite[Theorem 2.5]{LvWu-4}), we find
\[\boldsymbol{\varphi}_\infty\in \mathcal{H}^1,\quad\boldsymbol{\mu}_\infty\in \mathcal{H}^2,\quad (\beta(\varphi_\infty),\beta_\Gamma(\psi_\infty))\in\mathcal{H}^1.\]
Hence, $(\boldsymbol{\varphi}_\infty,\boldsymbol{\mu}_\infty)$ is a strong solution to the stationary problem
\begin{align}
    \begin{cases}
        \Delta\mu_\infty=0,&\text{in }\Omega,\\
        \mu_\infty=a_\Omega\varphi_\infty-J\ast\varphi_\infty+\beta(\varphi_\infty)+\pi(\varphi_\infty),&\text{in }\Omega,\\
        L\partial_\mathbf{n}\mu_\infty=\theta_\infty-\mu_\infty,\quad L\in(0,+\infty),&\text{on }\Gamma,\\
        \Delta_\Gamma\theta_\infty-\partial_\mathbf{n}\mu_\infty=0,&\text{on }\Gamma,\\
        \theta_\infty=a_\Gamma\psi_\infty-K\circledast\psi_\infty+\beta_\Gamma(\psi_\infty)+\pi(\psi_\infty),&\text{on }\Gamma.
    \end{cases}\label{stationary'}
\end{align}
Testing \eqref{stationary'}$_1$ by $\mu_\infty$ and \eqref{stationary'}$_4$ by $\theta_\infty$, respectively,
adding the resultants together and using \eqref{stationary'}$_3$, we obtain
\[\int_\Omega |\nabla\mu_\infty|^2\,\mathrm{d}x+\int_\Gamma|\nabla_\Gamma\theta_\infty|^2\,\mathrm{d}S+\frac{1}{L}\int_\Gamma|\theta_\infty-\mu_\infty|^2\,\mathrm{d}S=0.\]
As a result, we can conclude that $\mu_\infty=\theta_\infty$ are constants and the stationary problem is actually independent of the parameter $L$.
Furthermore, integrating \eqref{stationary'}$_2$ in $\Omega$, and \eqref{stationary'}$_5$ on $\Gamma$, we get \eqref{mu_s}.
Thus, $\boldsymbol{\varphi}_\infty$ is a strong solution to the stationary problem \eqref{stationary} with \eqref{mu_s}.

The proof of \eqref{phi_s-separation}--\eqref{mu_s-bounded} follows the idea in \cite[Lemma 4.1]{FW},
where the authors dealt with the (local) Cahn--Hilliard equation with dynamic boundary conditions.
Since $\boldsymbol{\varphi}_\infty$ satisfies \eqref{stationary}, by $(\mathbf{A2})$, we have
    \[\|\varphi_\infty\|_{L^\infty(\Omega)}\leq 1,\quad\|\psi_\infty\|_{L^\infty(\Gamma)}\leq 1.\]
It is easy to check that $\boldsymbol{\varphi}_\infty$ satisfies the following weak formulation
    \begin{align}
        &\int_\Omega\big(a_\Omega\varphi_\infty-J\ast\varphi_\infty+\beta(\varphi_\infty)+\pi(\varphi_\infty)-\mu_\infty\big)z\,\mathrm{d}x\notag\\
        &\qquad+\int_\Gamma\big(a_\Gamma\psi-K\circledast\psi+\beta_\Gamma(\psi_\infty)+\pi_\Gamma(\psi_\infty)-\mu_\infty\big)z_\Gamma\,\mathrm{d}S=0, \quad\forall\,\boldsymbol{z}=(z,z_\Gamma)\in\mathcal{L}^2.\label{steady-form}
    \end{align}
   Since $\overline{m}(\boldsymbol{\varphi}_\infty)=\overline{m}(\boldsymbol{\varphi}_0)=m_0\in(-1,1)$, taking $\boldsymbol{z}=\boldsymbol{\varphi}_\infty-m_0\boldsymbol{1}$ in \eqref{steady-form}, we obtain
    \begin{align}
        &\int_\Omega\beta(\varphi_\infty)(\varphi_\infty-m_0)\,\mathrm{d}x +\int_\Gamma\beta_\Gamma(\psi_\infty)(\psi_\infty- m_0)\,\mathrm{d}S
        \notag\\
        &\quad= -\int_\Omega\big(a_\Omega\varphi_\infty-J\ast\varphi_\infty+\pi(\varphi_\infty)\big)(\varphi_\infty- m_0)\,\mathrm{d}x +\mu_\infty\int_\Omega(\varphi_\infty- m_0)\,\mathrm{d}x\notag\\
        &\qquad-\int_\Gamma\big(a_\Gamma\psi_\infty-K\circledast\psi_\infty+\pi_\Gamma(\psi_\infty)\big)(\psi_\infty- m_0)\,\mathrm{d}S +\mu_\infty\int_\Gamma(\psi_\infty- m_0)\,\mathrm{d}S\notag\\
        &\quad=-\int_\Omega\big(a_\Omega\varphi_\infty-J\ast\varphi_\infty+\pi(\varphi_\infty)\big)(\varphi_\infty- m_0)\,\mathrm{d}x \notag\\
        &\qquad-\int_\Gamma\big(a_\Gamma\psi_\infty-K\circledast\psi_\infty+\pi_\Gamma(\psi_\infty)\big)(\psi_\infty- m_0)\,\mathrm{d}S\leq C,\label{mu_s-bound}
    \end{align}
    where the constant $C$ may depend on $\Omega$, $\Gamma$ and $\boldsymbol{\varphi}_0$,
    but is independent of particular $\boldsymbol{\varphi}_\infty$.
    Using \eqref{mu_s-bound}, the elementary inequality \cite[Proposition A.1]{MZ04} and the definition of $\mu_\infty$ (cf. \eqref{mu_s}),
    we can conclude \eqref{mu_s-bounded}.
%
    Next, from \eqref{stationary} and \eqref{mu_s-bounded}, we see that 
    \begin{align*}
        |\beta({\varphi}_\infty)|
        &\leq a_\Omega|{\varphi}_\infty|+|J\ast\varphi_\infty|+|\pi(\varphi_\infty)|+|\mu_\infty|\\
        &\leq 2a^\ast+\sup_{s\in[-1,1]}|\pi(s)|+M_\infty,\quad\text{a.e. in }\Omega,\\
        |\beta_\Gamma({\psi}_\infty)|
        &\leq a_\Gamma|{\psi}_\infty|+|K\circledast\psi_\infty|+|\pi_\Gamma(\psi_\infty)|+|\mu_\infty|\\
        &\leq 2a^\circledast+\sup_{s\in[-1,1]}|\pi_\Gamma(s)|+M_\infty,\quad\text{a.e. on }\Gamma,
    \end{align*}
    which, together with $(\mathbf{A2})$, lead to \eqref{phi_s-separation} and \eqref{psi_s-separation}.
    This completes the proof of Lemma \ref{steady-separation}.
\end{proof}

To show that the $\omega$-limit set $\omega(\boldsymbol{\varphi}_0)$ is a singleton,
we derive the following generalized  \L ojasiewicz--Simon inequality.
\begin{lemma}
	\label{LS}
	Let $\mathbf{(A1)}$--$\mathbf{(A3)}$ hold and $\widehat{\beta}$, $\widehat{\beta}_\Gamma$ be real analytic on $(-1,1)$,
    $\widehat{\pi}$, $\widehat{\pi}_\Gamma$ be real analytic on $\mathbb{R}$.
    Then, there exist constants $\gamma\in(0,1/2]$, $C>0$ and $\varpi>0$ such that the following inequality holds:
	\begin{align}
		\big|E(\boldsymbol{\varphi})-E(\boldsymbol{\varphi}_\infty)\big|^{1-\gamma}\leq C\|\boldsymbol{\mu}-\overline{m}(\boldsymbol{\mu})\|_{\mathcal{L}^2},\label{L-S}
	\end{align}
	for all $\boldsymbol{\varphi}\in U:=\big\{\boldsymbol{\zeta}\in\mathcal{L}^\infty:\ \|\boldsymbol{\zeta}\|_{\mathcal{L}^\infty}<1-\delta\big\}$
    provided that $\|\boldsymbol{\varphi}-\boldsymbol{\varphi}_\infty\|_{\mathcal{L}^2}\leq\varpi$.
\end{lemma}
\begin{proof}
	We apply the abstract result \cite[Theorem 6]{GG03} to the energy functional $E(\boldsymbol{\varphi})$.
    To begin with, we split $E(\boldsymbol{\varphi})$ into two parts
    \[E(\boldsymbol{\varphi})=\Phi(\boldsymbol{\varphi})+\Psi(\boldsymbol{\varphi}),\]
    where the functional $\Phi:\mathcal{L}^2\to\mathbb{R}\cup\{+\infty\}$ is given by
    \begin{align*}
		\Phi(\boldsymbol{\varphi}):=
		\begin{cases}
		\displaystyle\int_\Omega\Big(\frac{1}{2}a_\Omega\varphi^2+\widehat{\beta}(\varphi)+\widehat{\pi}(\varphi)\Big)\,\mathrm{d}x+\int_\Gamma\Big(\frac{1}{2}a_\Gamma\psi^2+\widehat{\beta}_\Gamma(\psi)+\widehat{\pi}_\Gamma(\psi)\Big)\,\mathrm{d}S,&\text{if }\boldsymbol{\varphi}\in\mathfrak{X}_m,\\[1mm]
		+\infty,&\text{otherwise},
		\end{cases}
	\end{align*}
	with closed effective domain dom$(\Phi)=\mathfrak{X}_m$, and the nonlocal interaction functional $\Psi:\mathcal{L}^2\to\mathbb{R}$ has the form
	\begin{align*}
	\Psi(\boldsymbol{\varphi}):=-\frac{1}{2}\int_\Omega(J\ast\varphi)\varphi\,\mathrm{d}x-\frac{1}{2}\int_\Gamma(K\circledast\psi)\psi\,\mathrm{d}S.
	\end{align*}
    We note that $\Phi$ is convex thanks to the assumptions $\mathbf{(A1)}$--$\mathbf{(A3)}$. Moreover, $\Phi$ is Fr\'echet differentiable on the open subset $U$ of $\mathcal{L}^\infty$,
    with the Fr\'echet derivative $\mathbb{D}\Phi:U\to\mathcal{L}^\infty$ satisfying
	\begin{align*}
	\big\langle\mathbb{D}\Phi(\boldsymbol{\varphi}),\boldsymbol{z}\big\rangle_{\mathcal{L}^2,\mathcal{L}^2} =\int_\Omega\big(\beta(\varphi)+\pi(\varphi)+a_\Omega\varphi\big)z\,\mathrm{d}x+\int_\Gamma\big(\beta_\Gamma(\psi)+\pi_\Gamma(\psi)+a_\Gamma \psi\big)z_\Gamma\,\mathrm{d}S,
	\end{align*}
	for all $\boldsymbol{\varphi}\in U$ and $\boldsymbol{z}\in\mathcal{L}^\infty$.
    The analyticity of $\mathbb{D}\Phi$ as a mapping on $U$ is standard and can be proved exactly similar to, e.g., \cite[Theorem 5.1]{FIP2004}.
    Moreover, due to $\mathbf{(A1)}$--$\mathbf{(A3)}$, it holds
	\begin{align*}
	\big\langle\mathbb{D}\Phi(\boldsymbol{\varphi}_1) -\mathbb{D}\Phi(\boldsymbol{\varphi}_2), \boldsymbol{\varphi}_1 -\boldsymbol{\varphi}_2\big\rangle_{\mathcal{L}^2,\mathcal{L}^2}\geq\min\big\{\alpha+a_\ast-\gamma_{1},\alpha+a_\circledast-\gamma_{2}\big\} \|\boldsymbol{\varphi}_1-\boldsymbol{\varphi}_2\|_{\mathcal{L}^2}^2,
	\end{align*}
	for all $\boldsymbol{\varphi}_1$, $\boldsymbol{\varphi}_2\in U$, and
	\begin{align*}
	\|\mathbb{D}\Phi(\boldsymbol{\varphi}_1)-\mathbb{D}\Phi(\boldsymbol{\varphi}_2)\|_{\mathcal{L}^2}\leq\max\big\{\widetilde{C}_{\text{up}}+\gamma_{1}+a^\ast,\widetilde{C}_{\text{up}}+\gamma_{2}+a^\circledast\big\}\|\boldsymbol{\varphi}_1-\boldsymbol{\varphi}_2\|_{\mathcal{L}^2},
	\end{align*}
	for all $\boldsymbol{\varphi}_1$, $\boldsymbol{\varphi}_2\in U$,
    where
    $$\widetilde{C}_{\text{up}} :=\sup_{s\in[-1+\delta,1-\delta]}|\beta'(s)|=\sup_{s\in[-1+\delta,1-\delta]}|\beta_\Gamma'(s)|.$$
    Moreover, computing the second Fr\'echet derivative $\mathbb{D}^2\Phi$ of $\Phi$,
	\begin{align*}
	\big\langle\mathbb{D}^2 \Phi(\boldsymbol{\varphi}) \boldsymbol{z},\boldsymbol{w}\big\rangle_{\mathcal{L}^2,\mathcal{L}^2} =\int_\Omega(\beta'(\varphi)+\pi'(\varphi)+a_\Omega)zw\,\mathrm{d}x+\int_\Gamma(\beta'_\Gamma(\psi)+\pi_\Gamma'(\psi)+a_\Gamma)z_\Gamma w_\Gamma\,\mathrm{d}S,
	\end{align*}
	we find that $\mathbb{D}^2\Phi(\boldsymbol{\varphi})\in \mathcal{B}(\mathcal{L}^\infty,\mathcal{L}^\infty)$ is an isomorphism for every $\boldsymbol{\varphi}\in U$.
    Concerning the nonlocal interaction functional $\Psi$, we have
	\begin{align*}
		\Psi(\boldsymbol{\varphi})=-\frac{1}{2}\langle(J\ast\varphi,K\circledast\psi),\boldsymbol{\varphi}\rangle_{\mathcal{L}^2,\mathcal{L}^2}, \quad\forall\,\boldsymbol{\varphi}\in \mathcal{L}^2.
	\end{align*}
	We recall that the linear operator $\boldsymbol{\varphi}\mapsto(J\ast\varphi,K\circledast\psi)$ is self-adjoint and compact from $\mathcal{L}^2$ to itself
    and is also compact from $\mathcal{L}^\infty$ to $C(\overline{\Omega})\times C(\Gamma)$ (see Remark \ref{compact-operator}).
    On the other hand, we have the following (orthogonal) sum decomposition of $\mathcal{L}^2=\mathcal{L}^2_{(0)}+\text{span}\{\boldsymbol{1}\}$.
    Thus, the annihilator of $\mathcal{L}^2_{(0)}$ is the one-dimensional subspace of constant functions  $(\mathcal{L}^2_{(0)})^\perp:=\big\{c\overline{m}\in(\mathcal{L}^2)':\ c\in\mathbb{R}\big\}$,
    where $\overline{m}\in(\mathcal{L}^2)'\simeq\mathcal{L}^2$ is given by $\langle\overline{m},\boldsymbol{\varphi}\rangle=\overline{m}(\boldsymbol{\varphi})$, for any $\boldsymbol{\varphi}\in\mathcal{L}^2$.
    As a consequence, the hypotheses of \cite[Theorem 6]{GG03} are satisfied, and the sum
	$$E=\Phi+\Psi:\mathcal{L}^2\to\mathbb{R}\cup\{+\infty\}$$
	is a well-defined, bounded from below functional, with nonempty, closed, and convex effective domain $\text{dom}(E)=\text{dom}(\Phi)=\mathfrak{X}_m$.
    Observing that the Fr\'echet derivative satisfies
	$$
    \mathbb{D}E(\boldsymbol{\varphi})=(a_\Omega\varphi-J\ast\varphi+\beta(\varphi) +\pi(\varphi),a_\Gamma\psi-K\circledast\psi+\beta_\Gamma(\psi)+\pi_\Gamma(\psi))=\boldsymbol{\mu},
    $$
	according to \cite[Theorem 6]{GG03}, we have
	\begin{align*}
		|E(\boldsymbol{\varphi})-E(\boldsymbol{\varphi}_\infty)|^{1-\gamma}&\leq C\inf\big\{\|\mathbb{D}E(\boldsymbol{\varphi})-{\mu}_0\|_{\mathcal{L}^2}:\ {\mu}_0\in (\mathcal{L}^2_{(0)})^\perp\big\}=C\|\boldsymbol{\mu}-\overline{m}(\boldsymbol{\mu})\|_{\mathcal{L}^2},
	\end{align*}
	which implies \eqref{L-S}. This completes the proof of Lemma \ref{LS}.
\end{proof}

Before giving the proof of Theorem \ref{equilibrium}, we prove a $\mathcal{L}^2$--$\mathcal{L}^\infty$ smoothing property that plays a significant role in the derivation of \eqref{convergence}.
To begin with, we denote $(\overline{\boldsymbol{\varphi}},\overline{\boldsymbol{\mu}})$ the difference between the global weak solution $(\boldsymbol{\varphi},\boldsymbol{\mu})$
and a stationary solution $(\boldsymbol{\varphi}_\infty,\boldsymbol{\mu}_\infty)$, that is,
\[\overline{\boldsymbol{\varphi}}:=\boldsymbol{\varphi}-\boldsymbol{\varphi}_\infty,\quad\overline{\boldsymbol{\mu}}:=\boldsymbol{\mu}-\boldsymbol{\mu}_\infty.\]
Then, $(\overline{\boldsymbol{\varphi}},\overline{\boldsymbol{\mu}})$ satisfies the following system
\begin{align}
    \begin{cases}
        \partial_t \overline{\varphi}=\Delta\overline{\mu},&\text{a.e. in }\Omega\times(\tau,+\infty),\\[1mm]
        \overline{\mu}=a_\Omega \overline{\varphi}-J\ast\overline{\varphi}+F'(\varphi)-F'(\varphi_\infty),&\text{a.e. in }\Omega\times(\tau,+\infty),\\[1mm]
        L\partial_{\mathbf{n}}\overline{\mu}=\overline{\theta}-\overline{\mu},&\text{a.e. on }\Gamma\times(\tau,+\infty),\\[1mm]
        \partial_t\overline{\psi}=\Delta_\Gamma\overline{\theta}-\partial_{\mathbf{n}}\overline{\mu},&\text{a.e. on }\Gamma\times(\tau,+\infty),\\[1mm]
        \overline{\theta}=a_\Gamma\overline{\psi}-K\circledast\overline{\psi}+G'(\psi)-G'(\psi_\infty),&\text{a.e. on }\Gamma\times(\tau,+\infty),
    \end{cases}\label{regular-difference}
\end{align}
for any $\tau>0$.

\begin{lemma}
	\label{L2-L infty}
	Let the assumptions of Theorem \ref{equilibrium} be satisfied.
    Then, for any $\tau>0$, the following $\mathcal{L}^2$--$\mathcal{L}^\infty$ smoothing property holds:
	\begin{align}
	\sup_{t\geq2\tau}\|\overline{\boldsymbol{\varphi}}(t)\|_{\mathcal{L}^\infty}\leq M_6\sup_{t\geq\tau}\|\overline{\boldsymbol{\varphi}}(t)\|_{\mathcal{L}^2}^2,\label{smoothing}
	\end{align}
    for some positive constant $M_6$ depending on $\tau$, $\Omega$, $\Gamma$, and the parameters of system \eqref{model}.
\end{lemma}
\begin{proof}
	For $p>1$, testing \eqref{regular-difference}$_1$ by $|\overline{\varphi}|^{p-1}\overline{\varphi}$
    and \eqref{regular-difference}$_4$ by $|\overline{\psi}|^{p-1}\overline{\psi}$, we obtain
	\begin{align}
		&\frac{1}{p+1}\frac{\mathrm{d}}{\mathrm{d}t}\Big(\int_\Omega|\overline{\varphi}|^{p+1}\,\mathrm{d}x+\int_\Gamma|\overline{\psi}|^{p+1}\,\mathrm{d}S\Big)\notag\\
		&\qquad+\underbrace{p\int_\Omega \nabla\overline{\mu}\cdot|\overline{\varphi}|^{p-1}\nabla\overline{\varphi}\,\mathrm{d}x+p\int_\Gamma \nabla_\Gamma\overline{\theta}\cdot|\overline{\psi}|^{p-1}\nabla_\Gamma\overline{\psi}\,\mathrm{d}S}_{I_1}\notag\\
		&\quad=\underbrace{\int_\Gamma\partial_{\mathbf{n}}\overline{\mu}(|\overline{\varphi}|^{p-1}\overline{\varphi}-|\overline{\psi}|^{p-1}\overline{\psi})\,\mathrm{d}S}_{I_2}.\label{eq1}
	\end{align}
Taking the equations of $\overline{\mu}$ and $\overline{\theta}$ into account, the term $I_1$ can be rewritten as
\begin{align}
	\frac{1}{p}I_1&=\int_\Omega(\overline{\varphi} \nabla a_\Omega+a_\Omega\nabla\overline{\varphi} -\nabla J\ast\overline{\varphi} )\cdot|\overline{\varphi} |^{p-1}\nabla\overline{\varphi} \,\mathrm{d}x\notag\\
	&\quad+\int_\Omega(F''(\varphi)\nabla\varphi-F''(\varphi_\infty)\nabla\varphi_\infty)\cdot|\overline{\varphi} |^{p-1}\nabla\overline{\varphi} \,\mathrm{d}x\notag\\
	&\quad+\int_\Gamma(\overline{\psi} \nabla_\Gamma a_\Gamma+a_\Gamma\nabla_\Gamma\overline{\psi} -\nabla_\Gamma K\circledast\overline{\psi} )\cdot|\overline{\psi} |^{p-1}\nabla_\Gamma\overline{\psi} \,\mathrm{d}S\notag\\
	&\quad+\int_\Gamma(G''(\psi)\nabla_\Gamma\psi-G''(\psi_\infty)\nabla_\Gamma\psi_\infty)\cdot|\overline{\psi} |^{p-1}\nabla_\Gamma\overline{\psi} \,\mathrm{d}S\notag\\
	&=\int_\Omega(a_\Omega+F''(\varphi))|\overline{\varphi} |^{p-1}|\nabla\overline{\varphi} |^2\,\mathrm{d}x+\int_\Gamma(a_\Gamma+G''(\psi))|\overline{\psi} |^{p-1}|\nabla_\Gamma\overline{\psi} |^2\,\mathrm{d}S\notag\\
&\quad+\int_\Omega\Big((F''(\varphi)-F''(\varphi_\infty))\nabla\varphi_\infty+\overline{\varphi} \nabla a_\Omega\Big)\cdot|\overline{\varphi} |^{p-1}\nabla\overline{\varphi} \,\mathrm{d}x\notag\\	
&\quad+\int_\Gamma\Big((G''(\psi)-G''(\psi_\infty))\nabla_\Gamma\psi_\infty+\overline{\psi} \nabla_\Gamma a_\Gamma\Big)\cdot|\overline{\psi} |^{p-1}\nabla_\Gamma\overline{\psi} \,\mathrm{d}S\notag\\
    &\quad-\int_\Omega(\nabla J\ast\overline{\varphi} )\cdot|\overline{\varphi} |^{p-1}\nabla\overline{\varphi} \,\mathrm{d}x-\int_\Gamma(\nabla_\Gamma K\circledast\overline{\psi} )\cdot|\overline{\psi} |^{p-1}\nabla_\Gamma\overline{\psi} \,\mathrm{d}S.\label{I1}
\end{align}
Thanks to $(\mathbf{A2})$ and $(\mathbf{A3})$, it holds
\begin{align*}
    &\int_\Omega(a_\Omega+F''(\varphi))|\overline{\varphi} |^{p-1}|\nabla\overline{\varphi} |^2\,\mathrm{d}x+\int_\Gamma(a_\Gamma+G''(\psi))|\overline{\psi} |^{p-1}|\nabla_\Gamma\overline{\psi} |^2\,\mathrm{d}S\\
    &\quad\geq \frac{4C_\ast}{(p+1)^2}\Big(\int_\Omega \Big|\nabla\Big(|\overline{\varphi} |^{\frac{p+1}{2}}\Big)\Big|^2\,\mathrm{d}x+\int_\Gamma \Big|\nabla_\Gamma\Big(|\overline{\psi} |^{\frac{p+1}{2}}\Big)\Big|^2\,\mathrm{d}S\Big),
\end{align*}
where the constant $C_\ast>0$ is determined in \eqref{C_ast}.
Since $\boldsymbol{\varphi}_\infty$ satisfies the stationary problem \eqref{stationary} and $\mu_\infty=\theta_\infty$ is constant,
by \eqref{phi_s-separation} and \eqref{psi_s-separation}, we see that
\begin{align}
    \nabla\varphi_\infty=\frac{\nabla J\ast\varphi_\infty-\nabla a_\Omega\varphi_\infty}{a_\Omega+\beta'(\varphi_\infty)+\pi'(\varphi_\infty)}\in L^\infty(\Omega),\quad \nabla_\Gamma\psi_\infty=\frac{\nabla _\Gamma K\circledast\psi_\infty-\nabla a_\Gamma\psi_\infty}{a_\Gamma+\beta_\Gamma'(\psi_\infty)+\pi_\Gamma'(\psi_\infty)}\in L^\infty(\Gamma).\label{stationary-infinity}
\end{align}
From the proof of \cite[Theorem 2.6]{LvWu-4}, we   find that $\bm{\varphi}$ satisfies
    \begin{align}
    \begin{cases}
    |\varphi(x,t)|\leq 1-\delta,\quad\text{for a.a. }(x,t)\in \Omega\times[\tau,+\infty),\\
        |\varphi(x,t)|\leq 1-\delta,\quad\text{for a.a. }(x,t)\in \Gamma\times[\tau,+\infty), \\
    |\psi(x,t)|\leq 1-\delta,\quad\text{for a.a. }(x,t)\in \Gamma\times[\tau,+\infty),
    \end{cases}
    \label{strict}
    \end{align}
for some $\delta\in(0,1)$ that may depend on $\tau$.
Since $\bm{\varphi}_\infty$ can be viewed as a solution to the evolution problem with itself as the initial datum, $\bm{\varphi}_\infty$ satisfies similar strict separation properties like in \eqref{strict}, which are independent of time. By $(\mathbf{A1})$, \eqref{stationary-infinity} and the strict separation property, we find
\begin{align*}
    &\int_\Omega\Big((F''(\varphi)-F''(\varphi_\infty))\nabla\varphi_\infty+\overline{\varphi} \nabla a_\Omega\Big)\cdot|\overline{\varphi} |^{p-1}\nabla\overline{\varphi} \,\mathrm{d}x\notag\\
    &\qquad+\int_\Gamma\Big((G''(\psi)-G''(\psi_\infty))\nabla_\Gamma\psi_\infty+\overline{\psi} \nabla_\Gamma a_\Gamma\Big)\cdot|\overline{\psi} |^{p-1}\nabla_\Gamma\overline{\psi} \,\mathrm{d}S\\
    &\quad\leq C\Big(\int_\Omega |\overline{\varphi} |^p|\nabla\overline{\varphi} |\,\mathrm{d}x+\int_\Gamma |\overline{\psi} |^p|\nabla_\Gamma\overline{\psi} |\,\mathrm{d}S\Big)\\
    &\quad= C\Big(\int_\Omega (|\overline{\varphi} |^\frac{p-1}{2}|\nabla\overline{\varphi} |)|\overline{\varphi} |^\frac{p+1}{2}\,\mathrm{d}x+\int_\Gamma (|\overline{\psi} |^\frac{p-1}{2}|\nabla_\Gamma\overline{\psi} |)|\overline{\psi} |^\frac{p+1}{2}\,\mathrm{d}S\Big)\\
    &\quad\leq \frac{C_\ast}{(p+1)^2}\Big(\int_\Omega \Big|\nabla\Big(|\overline{\varphi} |^{\frac{p+1}{2}}\Big)\Big|^2\,\mathrm{d}x+\int_\Gamma \Big|\nabla_\Gamma\Big(|\overline{\psi} |^{\frac{p+1}{2}}\Big)\Big|^2\,\mathrm{d}S\Big)\\
    &\qquad+C\Big(\int_\Omega|\overline{\varphi} |^{p+1}\,\mathrm{d}x+\int_\Gamma|\overline{\psi} |^{p+1}\,\mathrm{d}S\Big).
\end{align*}
For the last line of \eqref{I1}, by \cite[(2.15)]{BH05}, it holds
\begin{align*}
    &\Big|\int_\Omega(\nabla J\ast\overline{\varphi} )\cdot|\overline{\varphi} |^{p-1}\nabla\overline{\varphi} \,\mathrm{d}x+\int_\Gamma(\nabla_\Gamma K\circledast\overline{\psi} )\cdot|\overline{\psi} |^{p-1}\nabla_\Gamma\overline{\psi} \,\mathrm{d}S\Big|\\
    &\quad\leq \frac{C_\ast}{(p+1)^2}\Big(\int_\Omega \Big|\nabla\Big(|\overline{\varphi} |^{\frac{p+1}{2}}\Big)\Big|^2\,\mathrm{d}x+\int_\Gamma \Big|\nabla_\Gamma\Big(|\overline{\psi} |^{\frac{p+1}{2}}\Big)\Big|^2\,\mathrm{d}S\Big)\\
    &\qquad+C\Big(\int_\Omega|\overline{\varphi} |^{p+1}\,\mathrm{d}x+\int_\Gamma|\overline{\psi} |^{p+1}\,\mathrm{d}S\Big).
\end{align*}
Collecting the above estimates and \eqref{I1}, we obtain
\begin{align}
    I_1&\geq \frac{2C_\ast p}{(p+1)^2}\Big(\int_\Omega \Big|\nabla\Big(|\overline{\varphi} |^{\frac{p+1}{2}}\Big)\Big|^2\,\mathrm{d}x+\int_\Gamma \Big|\nabla_\Gamma\Big(|\overline{\psi} |^{\frac{p+1}{2}}\Big)\Big|^2\,\mathrm{d}S\Big)\notag\\
    &\qquad-Cp\Big(\int_\Omega|\overline{\varphi} |^{p+1}\,\mathrm{d}x+\int_\Gamma|\overline{\psi} |^{p+1}\,\mathrm{d}S\Big).\label{I_1}
\end{align}
Then for the term $I_2$, using the strict separation properties for $\bm{\varphi}$, $\bm{\varphi}_\infty$, the boundary condition \eqref{regular-difference}$_3$ and Lemma \ref{Mink}, we get
\begin{align}
I_2&=\frac{1}{L}\int_\Gamma(\overline{\theta} -\overline{\mu} )(|\overline{\varphi} |^{p-1}\overline{\varphi} -|\overline{\psi} |^{p-1}\overline{\psi})\,\mathrm{d}S\notag\\
&=\frac{1}{L}\int_\Gamma(a_\Gamma\overline{\psi}-K\circledast\overline{\psi}+G''(\overline{\xi}_\Gamma )\overline{\psi})(|\overline{\varphi} |^{p-1}\overline{\varphi} -|\overline{\psi} |^{p-1}\overline{\psi})\,\mathrm{d}S\notag\\
&\quad-\frac{1}{L}\int_\Gamma(a_\Omega\overline{\varphi} -J\ast\overline{\varphi} +F''(\overline{\xi} )\overline{\varphi} )(|\overline{\varphi} |^{p-1}\overline{\varphi} -|\overline{\psi}|^{p-1}\overline{\psi})\,\mathrm{d}S\notag\\
&\leq C\int_\Gamma|\overline{\varphi} |^{p+1}\,\mathrm{d}S+C\int_\Gamma|\overline{\psi}|^{p+1}\,\mathrm{d}S\notag\\
&\quad+C\Big(\|\overline{\varphi} \|_{L^{p+1}(\Gamma)}^p+\|\overline{\psi}\|_{L^{p+1}(\Gamma)}^p\Big)\Big(\|K\circledast\overline{\psi}\|_{L^{p+1}(\Gamma)}+\|J\ast\overline{\varphi} \|_{L^{p+1}(\Gamma)}\Big)\notag\\
&\leq C\int_\Gamma|\overline{\varphi} |^{p+1}\,\mathrm{d}S+C\int_\Gamma|\overline{\psi}|^{p+1}\,\mathrm{d}S\notag\\
&\quad+C\Big(\|\overline{\varphi} \|_{L^{p+1}(\Gamma)}^p+\|\overline{\psi}\|_{L^{p+1}(\Gamma)}^p\Big)\Big(\|K\|_{L^1(\Gamma)}\|\overline{\psi}\|_{L^{p+1}(\Gamma)}+\|J\|_{W^{1,1}(\mathbb{R}^d)}\|\overline{\varphi} \|_{L^{p+1}(\Omega)}\Big)\notag\\
&\leq C\int_\Gamma|\overline{\varphi} |^{p+1}\,\mathrm{d}S+C\int_\Omega|\overline{\varphi}|^{p+1}\,\mathrm{d}x+C\int_\Gamma|\overline{\psi}|^{p+1}\,\mathrm{d}S\notag\\
&\leq C\||\overline{\varphi} |^{\frac{p+1}{2}}\|_{H^{\frac{3}{4}}(\Omega)}^2+C\int_\Omega|\overline{\varphi}|^{p+1}\,\mathrm{d}x+C\int_\Gamma|\overline{\psi}|^{p+1}\,\mathrm{d}S\notag\\
&\leq C\||\overline{\varphi} |^{\frac{p+1}{2}}\|_H^{\frac{1}{2}}\|\nabla|\overline{\varphi} |^{\frac{p+1}{2}}\|_{H}^{\frac{3}{2}}+C\int_\Omega|\overline{\varphi}|^{p+1}\,\mathrm{d}x+C\int_\Gamma|\overline{\psi}|^{p+1}\,\mathrm{d}S\notag\\
&=C\||\overline{\varphi} |^{\frac{p+1}{2}}\|_H^{\frac{1}{2}}\Big(\frac{2(p+1)^2}{C_\ast p}\Big)^{\frac{3}{4}}\Big(\frac{C_\ast p\|\nabla|\overline{\varphi} |^{\frac{p+1}{2}}\|_{H}^2}{2(p+1)^2}\Big)^{\frac{3}{4}}\notag\\
&\quad +C\int_\Omega|\overline{\varphi}|^{p+1}\,\mathrm{d}x+C\int_\Gamma|\overline{\psi}|^{p+1}\,\mathrm{d}S\notag\\
&\leq \frac{C_\ast p}{2(p+1)^2}\int_\Omega\Big|\nabla|\overline{\varphi} |^{\frac{p+1}{2}}\Big|^2\,\mathrm{d}x+C(p+1)^3\int_\Omega|\overline{\varphi} |^{p+1}\,\mathrm{d}x+C\int_\Gamma|\overline{\psi}|^{p+1}\,\mathrm{d}S.\label{I2}
\end{align}
According to \eqref{eq1}, \eqref{I_1} and \eqref{I2}, we can deduce that
\begin{align}
	&\frac{\mathrm{d}}{\mathrm{d}t}\Big(\int_\Omega|\overline{\varphi}|^{p+1}\,\mathrm{d}x+\int_\Gamma|\overline{\psi}|^{p+1}\,\mathrm{d}S\Big)\notag\\
	&\qquad+\frac{C_\ast p}{(p+1)^2}\int_\Omega\Big|\nabla|\overline{\varphi}|^{\frac{p+1}{2}}\Big|^2\,\mathrm{d}x+\frac{C_\ast p}{(p+1)^2}\int_\Gamma\Big|\nabla_\Gamma|\overline{\psi}|^{\frac{p+1}{2}}\Big|^2\,\mathrm{d}S\notag\\
	&\quad\leq C(p+1)^3\Big(\int_\Omega|\overline{\varphi}|^{p+1}\,\mathrm{d}x+C\int_\Gamma|\overline{\psi}|^{p+1}\,\mathrm{d}S\Big).\label{eq2}
\end{align}
Set $p=2^k-1$ with $k\geq0$ and define
\begin{align}
	\mathcal{Y}_k(t):=\int_\Omega|\overline{\varphi}(t)|^{2^k}\,\mathrm{d}x+\int_\Gamma|\overline{\psi}(t)|^{2^k}\,\mathrm{d}S,\quad\text{for }k\geq0.\notag
\end{align}
With \eqref{eq2}, we can now exploit the scheme in \cite[Theorem 3.2, (3.8)--(3.10)]{Gal12} to derive the following inequality:
\begin{align}
	\mathcal{Y}_k(t)\leq C_\xi(2^k)^{\sigma}\Big(\sup_{s\geq t-\xi/2^k}\mathcal{Y}_{k-1}(s)\Big)^2,\quad\text{for }k\geq1,\label{eq3}
\end{align}
where $t$, $\xi$ are two positive constants such that $t-\xi/2^k>0$, $C_\xi$, $\sigma$ are positive constants independent of $k$,
and the constant $C_\xi$ is bounded away from zero. Set $\xi=\tau$, $t_0=2\tau$ and $t_k=t_{k-1}-\tau/2^k$ for $k\geq1$.
In view of \eqref{eq3}, we have
\begin{align}
	\sup_{t\geq t_{k-1}}\mathcal{Y}_k(t)\leq C_\tau (2^k)^{\sigma}\Big(\sup_{s\geq t_k}\mathcal{Y}_{k-1}(s)\Big)^2,\quad\text{for }k\geq1.\label{eq4}
\end{align}
Next, define
\[C_\sharp:=\sup_{s\geq\tau}\mathcal{Y}_1(s)=\sup_{s\geq\tau}\|\overline{\boldsymbol{\varphi}}(s)\|_{\mathcal{L}^2}^2.\]
Then we can iterate \eqref{eq4} with respect to $k\geq2$ and obtain
\begin{align}
	\sup_{t\geq2\tau}\mathcal{Y}_k(t)\leq \sup_{t\geq t_{k-1}}\mathcal{Y}_k(t)\leq C_\tau^{A_k}2^{\sigma B_k} C_\sharp^{2^k},\label{eq5}
\end{align}
where
\begin{align*}
	&A_k:=1+2+2^2+\cdots+2^k\leq 2^k\sum_{i\geq1}\frac{1}{2^i},\\
	&B_k:=k+2(k-1)+2^2(k-2)+\cdots+2^k\leq 2^k\sum_{i\geq1}\frac{i}{2^i}.
\end{align*}
Hence, taking the $2^k$-root on both sides of \eqref{eq5} and then letting $k\to+\infty$,
we deduce that there exists some positive constant $M_6$ independent of $t$, $k$, $\overline{\boldsymbol{\varphi}}$, $\xi$ and the initial data, such that
\begin{align}
	\sup_{t\geq2\tau} \|\overline{\boldsymbol{\varphi}}(t)\|_{\mathcal{L}^\infty} \leq \lim_{k\to+\infty}\sup_{t\geq2\tau}(\mathcal{Y}_k(t))^{1/2^k}\leq M_6C_\sharp= M_6\sup_{t\geq\tau}\|\overline{\boldsymbol{\varphi}}(t)\|_{\mathcal{L}^2}^2,\notag
\end{align}
which yields \eqref{smoothing}. This completes the proof of Lemma \ref{L2-L infty}.
\end{proof}

\begin{remark}\rm
    It is worth mentioning that the $\mathcal{L}^2$--$\mathcal{L}^\infty$ smoothing property
    is established for the difference between the global weak solution $\boldsymbol{\varphi}$ and a stationary solution $\boldsymbol{\varphi}_\infty$.
    The proof essentially relies on the $\mathcal{L}^\infty$-norm of the gradient $(\nabla\varphi_\infty,\nabla_\Gamma\psi_\infty)$ (see \eqref{stationary-infinity}).
    It seems difficult to establish a similar result for the difference between $\boldsymbol{\varphi}_1$ and $\boldsymbol{\varphi}_2$,
    where $\boldsymbol{\varphi}_i$ are the global weak solution to problem \eqref{model} corresponding to the initial data $\boldsymbol{\varphi}_{0,i}$ $(i\in\{1,2\})$, respectively.
\end{remark}

\noindent\textbf{Proof of Theorem \ref{equilibrium}.}
We now have all the necessary ingredients for the proof:
\begin{itemize}
    \item [(1)] The characterization of $\omega(\boldsymbol{\varphi}_0)$.

    \item [(2)] The energy identity \eqref{energyeq}.

    \item [(3)] The \L ojasiewicz--Simon inequality \eqref{L-S}.

    \item [(4)] The $\mathcal{L}^2$--$\mathcal{L}^\infty$ smoothing property \eqref{smoothing}.
\end{itemize}
Based on the four ingredients above, the proof of Theorem \ref{equilibrium} can be carried out in the same way as that for \cite[Theorem 2.21]{GG}. Hence, we omit the details here.
\hfill $\square$


\appendix

\section{Useful tools}
\setcounter{equation}{0}
We report for the reader's convenience the following abstract result on the existence of exponential attractors (see \cite[Proposition 4.1]{EZ04}).
\begin{lemma}
	\label{abstract}
	Let $\mathbb{H}$, $\mathbb{V}$, $\mathbb{V}_1$ be Banach spaces such that the embedding $\mathbb{V}_1\hookrightarrow\mathbb{V}$ is compact.
    Let $\mathbb{B}$ be a closed bounded subset of  $\,\mathbb{H}$, and let $\mathbb{S}:\mathbb{B}\to \mathbb{B}$ be a map.
    Assume also that there exists a uniformly Lipschitz continuous map $\mathbb{T}: \mathbb{B}\to\mathbb{V}_1$, i.e.,
	\begin{align}
		\|\mathbb{T}b_1-\mathbb{T}b_2\|_{\mathbb{V}_1}\leq K_1\|b_1-b_2\|_{\mathbb{H}},\quad\forall\,b_1,b_2\in \mathbb{B},\label{T-Lip}
	\end{align}
	for some $K_1\geq0$, such that
	\begin{align}
		\|\mathbb{S}b_1-\mathbb{S}b_2\|_{\mathbb{H}}\leq\epsilon\|b_1-b_2\|_{\mathbb{H}}+K_2\|\mathbb{T}b_1-\mathbb{T}b_2\|_{\mathbb{V}},\quad\forall\,b_1,b_2\in \mathbb{B},\label
		{A2}
	\end{align}
	for some $\epsilon<1/2$ and $K_2\geq0$.
    Then, there exists a (discrete) exponential attractor $\mathcal{E}_d\subset \mathbb{B}$ for the semigroup $\big\{\mathbb{S}(n):=\mathbb{S}^n,\,n\in\mathbb{Z}^+\big\}$
    with discrete time in the phase space $\mathbb{H}$.
\end{lemma}

The following Young-type inequality is useful in the proof of Lemma \ref{L2-L infty}.
\begin{lemma}
\label{Mink}
Let $J\in W^{1,1}(\mathbb{R}^d)$ and $\phi\in H^1(\Omega)\cap L^\infty({\Omega})$. Then, there holds
\begin{align}
    \|J\ast\phi\|_{L^p(\Gamma)}\leq C\|J\|_{W^{1,1}(\mathbb{R}^d)}\|\phi\|_{L^p(\Omega)}\quad\forall\,1\leq p\leq +\infty,\notag
\end{align}
    where the constant $C>0$ depends only on $\Omega$, but is independent of $p$.
\end{lemma}
\begin{proof}
    The conclusion is obvious when $p=+\infty$. We first consider the case $p=1$. It is easy to see that
    \begin{align*}
         \|J\ast\phi\|_{L^1(\Gamma)}&\leq \int_\Gamma\int_\Omega |J(x-y)||\phi(y)|\,\mathrm{d}y\,\mathrm{d}S_x \\
         &=\int_\Omega\int_\Gamma |J(x-y)|\,\mathrm{d}S_x\,|\phi(y)|\,\mathrm{d}y \\
         &\leq C\|J\|_{W^{1,1}(\mathbb{R}^d)}\|\phi\|_{L^1(\Omega)}.
    \end{align*}
    When $p>1$, for almost everywhere $x\in\Gamma$, the function $y\mapsto|J(x-y)||\phi(y)|^p$ is integrable in $\Omega$, that is,
    \[|J(x-y)|^{\frac{1}{p}}|\phi(y)|\in L^p_y(\Omega).\]
    Since $|J(x-y)|^{\frac{1}{p'}}\in L^{p'}_y(\Omega)$, we deduce from H\"older's inequality that
    \[|J(x-y)||\phi(y)|=|J(x-y)|^{\frac{1}{p'}}|J(x-y)|^{\frac{1}{p}}|\phi(y)|\in L^1_y(\Omega)\]
    and
    \[\int_\Omega |J(x-y)||\phi(y)|\,\mathrm{d}y\leq C\|J\|_{W^{1,1}(\mathbb{R}^d)}^{\frac{1}{p'}}\Big(\int_\Omega |J(x-y)||\phi(y)|^p\,\mathrm{d}y\Big)^{\frac{1}{p}},\]
    that is,
    \[|J\ast\phi(x)|^p\leq C\|J\|_{W^{1,1}(\mathbb{R}^d)}^{\frac{p}{p'}}(|J|\ast|\phi|^p)(x). \]
    Then, we can conclude that
    \[\|J\ast\phi\|_{L^p(\Gamma)}^p\leq C\|J\|_{W^{1,1}(\mathbb{R}^d)}^{\frac{p}{p'}}\||J|\ast|\phi|^p\|_{L^1(\Gamma)}\leq C\|J\|_{W^{1,1}(\mathbb{R}^d)}^{\frac{p}{p'}+1}\||\phi|^{p}\|_{L^1(\Omega)},\]
    which gives
    \[\|J\ast\phi\|_{L^p(\Gamma)}\leq C\|J\|_{W^{1,1}(\mathbb{R}^d)}\|\phi\|_{L^p(\Omega)}.\]
   This completes the proof of Lemma \ref{Mink}.
\end{proof}
	
	\medskip
	
\noindent \textbf{Acknowledgments.}
The authors thank the anonymous referees for their careful reading of an initial version of this paper and for many helpful comments that allowed them to improve the presentation. H. Wu is a member of the Key Laboratory of Mathematics for Nonlinear Sciences (Fudan University), Ministry of Education of China. The research of H. Wu was partially supported by NNSFC Grant No. 12071084.
	\medskip
	

\end{document}